\documentclass[11pt,english]{smfart} 
\usepackage{bm}
\usepackage[english,francais]{babel}
\usepackage{amscd}
\usepackage{amssymb,url,xspace,smfthm}
\usepackage{paralist}
\usepackage{datetime} 
\usepackage{colortbl}
\usepackage{color}
\usepackage[dvipsnames]{xcolor}
 \usepackage[%colorlinks,linkcolor=BrickRed,
 hypertexnames=true,linktocpage]{hyperref} 
\usepackage{mathtools}
\usepackage{euscript}
\usepackage[all]{xy}
\usepackage{graphicx} 

\usepackage{nameref}
\usepackage[T1]{fontenc}
\usepackage{newtxtext}
\usepackage{newtxmath}
\usepackage{textcomp}
\usepackage[text={150mm,251mm},centering, marginparwidth=75pt]{geometry} 
\usepackage{comment}  
\usepackage{cite}  
\usepackage{thmtools}
\usepackage{enumitem}
\usepackage{letltxmacro} 
\usepackage{nameref}
\usepackage{cleveref}
\usepackage{calc}
\usepackage{interval}
\usepackage{manfnt}
\usepackage{tikz-cd}  
\usepackage[utf8]{inputenc}
\usepackage{marginnote}
\usepackage{euscript}

\usepackage{smfhyperref}
\usepackage[hyperpageref]{backref}

\usepackage{faktor}

\theoremstyle{plain}
\newtheorem{thmx}{Theorem}
\renewcommand{\thethmx}{\Alph{thmx}} 
\newtheorem{thm}{Theorem}[section]  
\newtheorem{lem}[thm]{Lemma}
\newtheorem{claim}[thm]{Claim}

\newtheorem{proposition}[thm]{Proposition}
\newtheorem{cor}[thm]{Corollary}
\newtheorem{property}[thm]{Property}

\newtheorem{corx}[thmx]{Corollary} 
\newtheorem{conjecture}[thm]{Conjecture}
\theoremstyle{definition}
\newtheorem{dfn}[thm]{Definition}

\theoremstyle{remark}
\newtheorem{rem}[thm]{Remark} 
\newtheorem{example}[thm]{Example}


\numberwithin{equation}{thm}  

\theoremstyle{plain}
\newlist{thmlist}{enumerate}{1}
\setlist[thmlist]{wide = 0pt, labelwidth = 2em, labelsep*=0em, itemindent = 0pt, leftmargin = \dimexpr\labelwidth + \labelsep\relax, noitemsep,topsep = 1ex, font=\normalfont, label=(\roman*), ref=\thethm.(\roman{thmlisti})}

\addtotheorempostheadhook[thm]{\crefalias{thmlisti}{thm}}

\addtotheorempostheadhook[assumpsion]{\crefalias{thmlisti}{assumption}}

\addtotheorempostheadhook[cor]{\crefalias{thmlisti}{cor}}

\addtotheorempostheadhook[proposition]{\crefalias{thmlisti}{proposition}}

\addtotheorempostheadhook[dfn]{\crefalias{thmlisti}{dfn}}

\addtotheorempostheadhook[lem]{\crefalias{thmlisti}{lem}}
\addtotheorempostheadhook[main]{\crefalias{thmlisti}{main}}
\addtotheorempostheadhook[property]{\crefalias{thmlisti}{property}}

\addtotheorempostheadhook[rem]{\crefalias{thmlisti}{rem}}

\newlist{thmenum}{enumerate}{1} % also creates a counter called 'propenumi'
\setlist[thmenum]{wide = 0pt, labelwidth = 2em, labelsep*=0em, itemindent = 0pt, leftmargin = \dimexpr\labelwidth + \labelsep\relax, noitemsep,topsep = 1ex, font=\normalfont, label=(\roman*), ref=\thethmx.(\roman{thmenumi})}%{label=\alph*), ref=\thethmx~(\alph*)}
\crefalias{thmenumi}{thmx} 

\newlist{corlist}{enumerate}{1} % also creates a counter called 'propenumi'
\setlist[corlist]{wide = 0pt, labelwidth = 2em, labelsep*=0em, itemindent = 0pt, leftmargin = \dimexpr\labelwidth + \labelsep\relax, noitemsep,topsep = 1ex, font=\normalfont, label=(\roman*), ref=\thecorx.(\roman{corlisti})}%{label=\alph*), ref=\thethmx~(\alph*)}
\crefalias{corlisti}{corx} 

\crefname{lem}{Lemma}{Lemmas} 
\crefname{conjecture}{Conjecture}{Conjectures}
\crefname{thm}{Theorem}{Theorems}
\crefname{proposition}{Proposition}{Propositions}
\crefname{dfn}{Definition}{Definitions}
\crefname{rem}{Remark}{Remarks}
\crefname{cor}{Corollary}{Corollaries}
\crefname{corx}{Corollary}{Corollaries}
\crefname{problem}{Problem}{Problems}
\crefname{property}{Property}{Propertys}
\crefname{thmx}{Theorem}{Theorems}
\crefname{claim}{Claim}{Claims}
\crefname{assumption}{Assumption}{Assumptions}
\crefname{main}{Main Theorem}{Main Theorems}

\def\ep{\varepsilon}

\def\Res{{\rm Res}}
\def\Sp{{\rm Sp}}

\newcommand{\cR}{\mathcal{R}}
\newcommand{\cS}{\mathcal{S}}

\makeatletter
\newcommand*{\rom}[1]{\expandafter\@slowromancap\romannumeral #1@}
\makeatother

\makeatletter     %Replace the Section ...  by the symbol \S ...  in Cref
\newcommand{\crefnames}[3]{%
	\@for\next:=#1\do{%
		\expandafter\crefname\expandafter{\next}{#2}{#3}%
	}%
}
\makeatother

\crefnames{part,chapter,section}{\S}{\S\S}

\newcommand{\cA}{\mathcal A}
\newcommand{\cC}{\mathcal C}
\newcommand{\cD}{\mathcal D}

\newcommand{\cI}{\mathcal I}

\newcommand{\cO}{\mathcal O}

\newcommand{\bB}{\mathbb{B}}
\newcommand{\bC}{\mathbb{C}}
\newcommand{\bD}{\mathbb{D}}

\newcommand{\bF}{\mathbb{F}}

\newcommand{\bP}{\mathbb{P}}
\newcommand{\bQ}{\mathbb{Q}}
\newcommand{\bR}{\mathbb{R}}

\newcommand{\bZ}{\mathbb{Z}}

\newcommand{\xsp}{X^{\! \rm sp}}

\newcommand{\kg}{\mathfrak{g}}

\def\db{\bar{\partial}}

  \def\spec{\textrm{Spec}\,}
 \def\d{\partial}

\def\sn{\sqrt{-1}}

\def\End{{\rm \small  End}}

\def\Sym{{\rm Sym}}

\def\sp{{\rm sp}}

\newcommand{\Spab}{\mathrm{Sp}_{\mathrm{sab}}}
\newcommand{\Sph}{\mathrm{Sp}_{\mathrm{h}}}
\newcommand{\Spp}{\mathrm{Sp}_{\mathrm{p}}}
\newcommand{\Spalg}{\mathrm{Sp}_{\mathrm{alg}}}

\newcommand{\Hom}{{\rm Hom}}

\newcommand{\diae}{{}^\diamond\! E}

 \title[Hyperbolicity and  representations of $\pi_1$]{Hyperbolicity and fundamental groups of complex quasi-projective varieties (II): via non-abelian Hodge theories}

\date{\today} 

\author[B. Cadorel]{Beno\^{i}t Cadorel} 
\email{benoit.cadorel@univ-lorraine.fr}
\address{Institut \'Elie Cartan de Lorraine, Universit\'e de Lorraine, F-54000 Nancy,
	France}
\urladdr{http://www.normalesup.org/~bcadorel/} 
\author[Y. Deng]{Ya Deng}

\email{ya.deng@math.cnrs.fr, deng@imj-prg.fr}
\address{CNRS,  
	Institut de Math\'ematiques de Jussieu-Paris Rive Gauche,
	Sorbonne Universit\'e, Campus Pierre et Marie Curie,
	4 place Jussieu, 75252 Paris Cedex 05, France}
\urladdr{https://ydeng.perso.math.cnrs.fr}

\author[K. Yamanoi]{Katsutoshi Yamanoi}

\email{yamanoi@math.sci.osaka-u.ac.jp}
\address{Department of Mathematics, Graduate School of Science, Osaka University, Toyonaka,  Osaka 560-0043, Japan} 
\urladdr{https://sites.google.com/site/yamanoimath/}

\begin{document}

	\begin{abstract} 
	This is Part~II of a series of three papers. We studies the   hyperbolicity of complex quasi-projective varieties $X$   in the presence of a big and reductive representation $\varrho: \pi_1(X)\to {\rm GL}_N(\mathbb{C})$. For any Galois conjugate variety $X^\sigma$ with $\sigma \in {\rm Aut}(\mathbb{C}/\mathbb{Q})$, we prove the generalized Green–Griffiths–Lang conjecture. When $\varrho$ is  furthermore large, we show that the special subsets of $X^\sigma$ describing the non-hyperbolicity locus coincide, and that this locus is proper exactly when $X$ is of log general type. Moreover, if the Zariski closure of $\varrho(\pi_1(X))$ is semisimple, we prove that there exists  a proper Zariski closed subset $Z \subsetneqq X^\sigma$ such that every subvariety not contained in $Z$ is of log general type and all entire curves in $X^\sigma$ are contained in $Z$. This result extends the theorems of the third author (2010) and of Campana–Claudon–Eyssidieux (2015) from projective to quasi-projective varieties, and yields stronger conclusions even in the projective case.
	\end{abstract}  


%\altkeywords{pseudo Picard hyperbolicité, conjecture de Green-Griffiths-Lang généralisée, variétés spéciales, applications harmoniques vers des immeubles de Bruhat-Tits, variations de structures de Hodge, théorie de Nevanlinna, conjugués de Galois}

\maketitle
  \tableofcontents

	 \section{Introduction} 
	 \subsection{Hyperbolicity and fundamental groups}
The celebrated Bombieri--Lang conjecture states that if $X$ is a variety of general type defined over a number field $k$, then the set of $k$-rational points of $X$ is contained in a proper Zariski closed subset of $X$. Its complex analogue is the Green--Griffiths--Lang conjecture, which asserts that for a complex projective variety $X$ of general type there exists a proper Zariski closed subset $Z \subset X$ containing all non-constant entire curves $f:\bC \to X$. This property, commonly referred to as the \emph{hyperbolicity} of algebraic varieties, is a central topic in complex geometry. For recent progress we refer the reader to the survey of Demailly \cite{Dem20}, especially on the hyperbolicity of general hypersurfaces of high degree in projective space, as well as to the survey of the third author \cite{yamanoi2015kobayashi} from the viewpoint of Nevanlinna theory.

In our first paper \cite{CDY25} of this series,  we proved the \emph{generalized} Green--Griffiths--Lang conjecture (cf.~\cref{conj:GGL}) for complex quasi-projective varieties $X$ admitting a morphism $a:X \to A$ to a semi-abelian variety with $\dim X = \dim a(X)$. From the viewpoint of fundamental groups, this corresponds to the existence of a big representation of $\pi_1(X)$ into the abelian group $\bZ^m$. In the present paper, we study the case where $\pi_1(X)$ admits a Zariski dense and  big representation into a linear algebraic group $G$, which is semisimple, or more generally, reductive. 

Recall that a linear algebraic group $G$ over a field $K$ is called semisimple if it has no non-trivial connected normal solvable algebraic subgroups defined over the algebraic closure of $K$. 
Throughout, we follow the convention that semisimple algebraic groups are non-trivial.
A representation $\varrho:\pi_1(X)\to G(K)$ is said to be \emph{big}, or \emph{generically large} in the terminology of \cite{Kol95}, if for any closed irreducible subvariety $Z \subset X$ of positive dimension containing a \emph{very general} point of $X$, the image $\varrho\bigl({\rm Im}[\pi_1(Z^{\rm norm})\to \pi_1(X)]\bigr)$ is infinite, where $Z^{\rm norm}$ denotes the normalization of $Z$ (see \cref{def:big representation}). A stronger notion of largeness also exists: $\varrho$ is called \emph{large} if this condition holds for every positive-dimensional closed subvariety $Z \subset X$.

The study of hyperbolicity properties of complex projective varieties with big and reductive representations of the fundamental groups 
originates in the work of Mok and Zuo \cite{Mok92,Zuo96}. 
Subsequently, this question was addressed in \cite{Yam10,CCE15}. More precisely, let $X$ be a smooth complex projective variety endowed with a big and Zariski-dense representation $\varrho:\pi_1(X)\to G(\bC)$, where $G$ is a semisimple linear algebraic group over $\bC$. The third author \cite[Proposition~2.1]{Yam10} proved that $X$ admits no Zariski-dense entire curves $f:\bC \to X$, while Campana--Claudon--Eyssidieux \cite[Theorem~1]{CCE15} established that $X$ is of general type.  

In light of the generalized  Green-Griffiths-Lang conjecture (cf. \cref{conj:GGL}), however, one expects an even stronger form of hyperbolicity to hold for \emph{quasi-projective varieties}.  
In addition, it is hoped that the complex hyperbolicity is stable under Galois conjugate varieties $X^{\sigma}$, where $\sigma\in {\rm Aut}(\bC/\bQ)$ is an automorphism (e.g., \cite[VIII, Conjecture 1.3]{Lan97}).
However, since the induced map $X^{\sigma}\to X$ may not be continuous in the classical topology,  the stability of the property of the representations $\pi_1(X^{\sigma})\to G(\bC)$ is unclear.

In this paper,  we confirm these expectations by extending and strengthening the theorems of Campana--Claudon--Eyssidieux and the third author to the setting of complex quasi-projective varieties.

	 \begin{thmx}[=\cref{thm:20220819}]\label{main2}
	 	Let $X$ be a complex   quasi-projective normal variety and let $G$ be a semisimple algebraic group over $\bC$. If  $\varrho:\pi_1(X)\to G(\bC)$ is a big  and Zariski dense representation, then for any automorphism $\sigma\in {\rm Aut}(\bC/\bQ)$,  there is a proper Zariski closed subset $Z\subsetneqq X^\sigma$ where $X^\sigma$ is the Galois conjugate variety of $X$ under $\sigma$ such that
	 	\begin{thmenum}
	 		\item \label{main:log general type}    any closed  subvariety   of $X^\sigma$ not contained in $Z$ is  of log general type. In particular, $X^\sigma$ is of log general type.
\item \label{main:pseudo Picard}   Any holomorphic map $f:\bD^*\to X^\sigma$ from the punctured disk $\bD^*$ to $X^\sigma$ with $f(\bD^*)\not\subseteq Z$ extends to  a holomorphic map from the disk $\bD$ to a projective compactification $\overline{X^\sigma}$ of $X^\sigma$ (i.e. $X^\sigma$ is \emph{pseudo Picard hyperbolic}).   In particular,  all entire curves in $X^\sigma$ lie on $Z$ (i.e. $X^\sigma$ is \emph{pseudo Brody hyperbolic}).
	 	\end{thmenum} 
	 \end{thmx}
It is worth noting that every pseudo Picard hyperbolic variety is pseudo Brody hyperbolic. 
We also mention that the two conditions for the representation $\varrho$ in \cref{main2} are essential to conclude the two statements in \cref{main2}, as discussed in \cref{rem:sharp}. %Both statements in \cref{main2} are new even in the case where $X$ is projective, and the proof of \cref{main2} involves several new techniques. 

For the proof of \cref{main2}, we make substantial use of the following two technical results:
\begin{itemize}
	\item the existence, for any Zariski dense representation $\tau:\pi_1(X)\to G(K)$ with $G$ a reductive group over a non-archimedean local field $K$, of $\tau$-equivariant pluriharmonic maps from the universal cover $\widetilde{X}$ of $X$ to the Bruhat–Tits building $\Delta(G)$ of $G$, as constructed in \cite{BDDM,DM24};  
	\item the Nevanlinna theory developed in \cite{CDY25} for holomorphic maps from ramified coverings of $\bD^*$ with small ramifications to quasi-projective varieties of log general type and maximal quasi-Albanese dimension.
\end{itemize}

It is noteworthy that the condition of bigness for the representations $\varrho$ in \cref{main2} is not particularly restrictive, unlike the requirement for a large representation.  In fact, in \cref{lem:kollar} we demonstrate that any linear representation of $\pi_1(X)$ can be factored through a big representation after taking a finite \'etale cover. This result, combined with \cref{main2}, yields a factorization theorem for linear representations of $\pi_1(X)$.
 	 \begin{corx}\label{main}
 	 	Let $X$ be a complex    quasi-projective  normal variety and let $G$ be a semisimple algebraic group over $\bC$. 
If  $\varrho:\pi_1(X)\to G(\bC)$ is a Zariski dense representation, then there  exist
a  finite \'etale cover \(\nu:\widehat{X}\to X\), a birational and proper morphism \(\mu:\widehat{X}'\to \widehat{X}\), a dominant morphism $f:\widehat{X}'\to Y$  with connected general fibers, and a   big   and Zariski dense representation \(\tau : \pi_{1}(Y) \to G(\bC)\)  such that
 	\begin{itemize}
 	\item   \(f^{\ast} \tau = (\nu\circ\mu)^{\ast}\varrho\). 
 			\item There is a proper Zariski closed subset $Z\subsetneqq Y$ such that any closed   subvariety of $Y$ not contained in $Z$ is  of log general type.  
 			\item  $Y$ is pseudo Picard hyperbolic, and in particular pseudo Brody hyperbolic.  
 	\end{itemize}  
 In particular,   $X$ is not weakly special and does not contain Zariski-dense entire curves. 
 \end{corx} 
Note that by Campana \cite{Cam11}, a quasi-projective variety $X$ is \emph{weakly special} if for any finite \'etale cover $\widehat{X}\to X$ and any proper birational modification $\widehat{X}'\to \widehat{X}$, there exists no   dominant morphism $\widehat{X}'\rightarrow Y$  with $Y$ a positive-dimensional quasi-projective normal variety of log general type.

\cref{main} generalizes the previous work by Mok \cite{Mok92}, Corlette-Simpson \cite{CS08}, and Campana-Claudon-Eyssidieux \cite{CCE15}, in which they proved similar factorisation results.

\subsection{On the generalized Green-Griffiths-Lang conjecture}\label{subsec:20230427}
Building upon \cref{main2}, we further investigate the generalized Green-Griffiths-Lang conjecture (cf. \Cref{conj:GGL}) and its relation to the non-hyperbolicity locus of a smooth quasi-projective variety $X$, under the weaker assumption that $\pi_1(X)$ admits a big and reductive representation. 
Specifically, we introduce four special subsets of $X$ that measure the non-hyperbolicity locus from different perspectives, as defined in \cref{def:special2}. Our main result, given in \cref{main:GGL}, establishes the equivalence of several properties of the conjugate variety $X^\sigma$ under the assumption that $\varrho:\pi_1(X)\to {\rm GL}_N(\bC)$ is a big and reductive representation, and for any automorphism $\sigma\in {\rm Aut}(\bC/\bQ)$. Additionally, we provide a further result regarding the special subsets, as stated in \cref{main:special}.
\begin{dfn}[Special subsets] \label{def:special2}
	Let $X$ be a smooth quasi-projective variety.
	\begin{thmlist}
		\item $\Spab(X) := \overline{\bigcup_{f}f(A_0)}^{\mathrm{Zar}}$, where $f$ ranges over all non-constant rational maps $f:A\dashrightarrow X$ from all non-trivial semi-abelian varieties $A$ to $X$ such that $f$ is regular $A_0\to X$ on a Zariski open subset $A_0\subset A$ whose complement $A\backslash A_0$ has codimension at least two;
		\item $\Sph(X) := \overline{\bigcup_{f}f(\mathbb{C})}^{\mathrm{Zar}}$, where $f$ ranges over all non-constant holomorphic maps from $\mathbb{C}$ to $X$;
		\item $\Spalg(X) := \overline{\bigcup_{V} V}^{\mathrm{Zar}}$, where $V$ ranges over all positive-dimensional closed subvarieties of $X$ which are not of log general type;
		\item $\Spp(X) := \overline{\bigcup_{f}f(\bD^*)}^{\mathrm{Zar}}$, where $f$ ranges over all holomorphic maps from the punctured disk $\bD^*$ to $X$ with essential singularity at the origin, i.e., $f$ has no holomorphic extension $\bar{f}:\mathbb D\to\overline{X}$ to a projective compactification $\overline{X}$.
	\end{thmlist}
	\end{dfn}

The first two sets $\Spab(X)$ and $\Sph(X)$ are introduced by Lang for the compact case.
He made the following two conjectures (cf. \cite[I, 3.5]{Lan97} and \cite[VIII, Conjecture 1.3]{Lan97}):
\begin{itemize}
\item
$\Spab(X)\subsetneqq X$ if and only if $X$ is of general type.
\item
$\Spab(X)=\Sph(X)$.
\end{itemize}
The first assertion implicitly include the following third conjecture:
\begin{itemize}
\item
$\Spab(X)=\Spalg(X)$.
\end{itemize}

The original two conjectures imply the famous strong Green-Griffiths conjecture that varieties of (log) general type are pseudo Brody hyperbolic.
Here we note that, by definition, $X$ is pseudo Brody hyperbolic if and only if $\Sph(X)\subsetneqq X$.
Similarly, $X$ is pseudo Picard hyperbolic if and only if $\Spp(X)\subsetneqq X$.

\begin{thmx}[=\cref{thm:GGL}]\label{main:GGL}
	Let $X$ be a complex  smooth  quasi-projective  variety admitting a big and reductive representation  $\varrho:\pi_1(X)\to {\rm GL}_N(\bC)$. 
Then for any  automorphism $\sigma\in {\rm Aut}(\bC/\bQ)$, the following properties are equivalent:
	\begin{enumerate}[font=\normalfont, label=(\alph*)] 
		\item $X^\sigma$ is of log general type. 
		\item   $\Spp(X^\sigma)\subsetneqq X^\sigma$.
		\item  $\Sph(X^\sigma)\subsetneqq X^\sigma$.
		\item  $\Spalg(X^\sigma)\subsetneqq X^\sigma$.
\item  $\Spab(X^\sigma)\subsetneqq X^\sigma$.
	\end{enumerate} 
\end{thmx} 

We note that the implication $(a)\implies (c)$ in \cref{main:GGL} establishes the strong Green-Griffiths conjecture for $X^\sigma$ for all automorphism $\sigma\in {\rm Aut}(\bC/\bQ)$, provided a big and reductive representation  $\varrho:\pi_1(X)\to {\rm GL}_N(\bC)$ exists.

As for the second and third conjectures of Lang, we obtain the following theorem under the stronger assumption when $\pi_1(X)$ admits a \emph{large} and reductive representation.

\begin{thmx}[=\cref{thm:special}]\label{main:special}
	Let $X$ be a smooth quasi-projective   variety admitting a large and reductive representation $\varrho:\pi_1(X)\to {\rm GL}_N(\bC)$. Then   for any  automorphism $\sigma\in {\rm Aut}(\bC/\bQ)$,  
	\begin{thmenum}
		\item   the four special subsets defined in \cref{def:special2} are the same, i.e., $$\Spalg(X^\sigma)=\Spab(X^\sigma)=\Sph(X^\sigma)=\Spp(X^\sigma).$$  
		\item These special subsets are conjugate under  automorphism   $\sigma$,  i.e., $$\Sp_{\bullet}(X^\sigma)=\Sp_{\bullet}(X)^\sigma,$$
		where $\Sp_{\bullet}$  denotes any of $\Spalg$, $\Spab$, $\Sph$ or $ \Spp$.
	\end{thmenum}
\end{thmx}
Let us mention that, building on the methods and results of this paper, \cref{main2,main:GGL,main:special} were extended by the first and third authors in \cite{DY23b} to the setting of big representations $\pi_1(X)\to {\rm GL}_N(K)$ with ${\rm char}\, K>0$.

Let us also mention that in \cite{Bru22}, Brunebarbe independently discussed related results in the compact case.
However, the approach adopted here differs substantially from his.
For instance, our method crucially applies a strong result of Campana–Păun \cite{CP19} in the proof of \cref{main:log general type}, even in the compact case (see also the proof of \cref{prop:log general type}), whereas the arguments in \cite{Bru22} rely essentially on earlier work \cite{CCE15,Yam10}.

 \subsection{Results in non-abelian Hodge  theory}\label{sec:techniques}
 We believe that some of the new techniques developed in the proof of \cref{main2} are of significant interest in their own right. One such technique is a reduction theorem for Zariski dense representations $\varrho: \pi_1(X)\to G(K)$, where $G$ is a reductive algebraic group defined over a non-Archimedean local field $K$.
 \begin{thmx}[=\cref{thm:KZreduction}] \label{main3}
 	Let $X$ be a complex   quasi-projective normal variety, and let $\varrho:\pi_1(X)\to {\rm GL}_N(K)$ be a reductive representation where $K$ is non-archimedean local field.  Then there exists a quasi-projective normal variety $S_\varrho$ and a dominant morphism $s_\varrho:X\to S_\varrho$ with connected general fibers, such that  for any connected Zariski closed  subset $T$ of $X$, the following properties are equivalent:
 	\begin{enumerate}[label={\rm (\alph*)}]
 		\item \label{item bounded} the image $\rho({\rm Im}[\pi_1(T)\to \pi_1(X)])$ is a bounded subgroup of $G(K)$.
 		\item \label{item normalization} For every irreducible component $T_o$ of $T$, the image $\rho({\rm Im}[\pi_1(T_o^{\rm norm})\to \pi_1(X)])$ is a bounded subgroup of $G(K)$.
 		\item \label{item contraction}The image $s_\varrho(T)$ is a point.
 	\end{enumerate} 
 \end{thmx}
It is worth noting that if $X$ is projective, the equivalence between \Cref{item bounded} and \Cref{item contraction} has been established by Katzarkov \cite{Kat97}, Eyssidieux \cite[Proposition 1.4.7]{Eys04}, and Zuo \cite{Zuo96}. 
However, the equivalence between \Cref{item bounded} and \Cref{item normalization}, not involving $s_{\varrho}$, seems to be new to the literature, even for the projective setting.
One of the building blocks of the proof of \cref{thm:KZreduction} is based on previous results by Brotbek, Daskalopoulos, Mese, and the second author \cite{BDDM,DM24} on the existence of $\varrho$-equivariant pluriharmonic mappings to Bruhat-Tits buildings (an extension of Gromov-Schoen's theorem to quasi-projective cases)  with suitable energy growth at infinity, and the construction of logarithmic symmetric differential forms via these harmonic mappings.

The following theorem is a crucial component in the proof of \cref{main2}.
 \begin{thmx}[=\cref{thm:main33}] \label{main6}
	Let $X$ be a smooth quasi-projective variety.  Let $G$ be an almost simple algebraic group defined over a non-archimedean local field $K$.  Suppose that $\varrho:\pi_1(X)\to G(K)$ is a big and unbounded Zariski dense representation. Then:
	\begin{thmenum}
		\item  \label{main:lgt}$\Spalg(X)\subsetneqq X$. 
	\item  \label{main:PPH} $X$ is pseudo Picard hyperbolic.
\end{thmenum}
\end{thmx}
%One crucial ingredient in the proof of \cref{main6} is a result on Campana's conjecture.
%\begin{thm}[=\cref{lem:abelian pi}]
%	Let $\alpha:X\to \cA$ be  a  (possibly non-proper)  morphism  from a smooth quasi-projective variety $X$ to a semi-abelian variety $\cA$ with $\dim X=\dim \alpha(X)$. If $X$ is either  $h$-special or weakly special or $\overline{\kappa}(X)=0$, then  $\pi_1(X)$ is abelian.  Moreover, $\alpha(X)$ contains a Zariski open set $\cA^\circ$ of $\cA$ with $\cA\backslash \cA^\circ$ of codimension at least two.
%\end{thm}
A significant building block in the proof of \cref{main6} is \cite[Thm A]{CDY25}  on the Big Picard type theorem for holomorphic curves into quasi-projective varieties with maximal quasi-Albanese dimension. 
%For  notions of Nevanlinna theory in the theorem we refer the readers to \cref{subsec:notion Nevanlinna}. 

% \begin{thm}\label{corx}
%	 Let  $X$ be a complex connected smooth quasi-projective variety and let $\varrho:\pi_1(X)\to {\rm GL}_N(\bC)$ be a linear representation. If $X$ admits a Zariski dense  entire curve (i.e. \(X\) is Brody special in the sense of Campana), then $\varrho(\pi_1(X))$  is virtually abelian.  
% \end{thm}
 
 \medspace
 
 \subsection{Structure of the series and further developments}
 \begin{figure}
 	\centering
 	\begin{tikzpicture}[>=stealth, node distance=1.7cm, every node/.style={rectangle, draw, align=center}]	
 		\node (A) {\cref{main6}}; 
 				\node (X) [above left of = A]  {\cref{sec:reduction}: Construction of spectral cover}; 
 				\node (E) [below left of = X] {\cref{main3}};
 				 		\node (W) [below right of = A] {\cref{main2}};
 				 		\node (F) [above right of = W] {\cref{sec:rigid}: $\bC$-VHS};
 		\node (B) [below left of=W] {\cref{main}};
 		\node (C) [below right of=W] {\cref{main:GGL}};
 		\node (D) [below of= C] {\cref{main:special}}; vv

 			 			\path[->] (X) edge (A);
 			 			\path[->] (X) edge (E);
 			 						\path[->] (A) edge (W);
 		\path[->] (W) edge (B);
 		\path[->] (W) edge (C);
 		\path[->] (F) edge (W);
 		\path[->] (C) edge (D);vv
 	\end{tikzpicture}
 	\caption{Relationships between Main Theorems}
 	\label{fig:relationships}
 \end{figure}
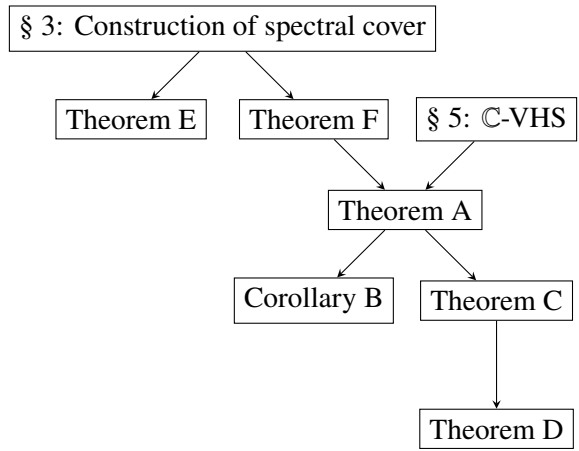

%This paper constitutes Part~II of the long preprint \cite{CDY22original} on arXiv, which has been divided into three parts for journal submission. It corresponds to Sections 2 and 5–9 of \cite{CDY22original}. In Part~I of the series, which corresponds to Sections 3 and 4 of \cite{CDY22original} and has been submitted on arxiv \cite{CDY25}, the main results for the case of abelian representations were established using Nevanlinna theory. The present paper relies mainly on non-abelian Hodge theories, combined with the results from Part~I, to prove the main theorems, so it is self-contained.
%The methods employed here differ substantially from those developed in  \cite{CDY25}, and assuming Theorem B therein, the arguments presented in this paper do not rely on any other results from that work.
%In Part III of this series, which covers the remainder of \cite{CDY22original}, we will address applications of the results from Part I and the present paper. 
%Specifically, we will focus on the fundamental groups of Campana's special varieties and a structure theorem motivated by a conjecture by Kollár \cite{Kol95}.
This paper constitutes Part~II of the long preprint~\cite{CDY22original} on arXiv, which has been divided into three parts for journal submission. It corresponds to Sections~2 and~5--9 of~\cite{CDY22original}. 
Part~I of the series, corresponding to Sections~3 and~4 of~\cite{CDY22original} and available on arXiv as~\cite{CDY25}, establishes the main results for the case of abelian representations using \emph{Nevanlinna theory}. 
The present paper is based mainly on non-abelian Hodge theory. 
The methods developed here differ substantially from those in~\cite{CDY25}; assuming Theorems~B and~C therein, the arguments of the present paper are largely independent of that work, so that the current paper can be read independently.
Part~III of this series, which covers the remaining sections of~\cite{CDY22original}, will be devoted to applications of the results from Part~I and the present paper. 
In particular, we will study the fundamental groups of Campana’s special varieties and establish a structure theorem inspired by a conjecture of Kollár~\cite{Kol95}.

Since its appearance in 2022, the preprint~\cite{CDY22original} has already found several further applications (see, for instance,~\cite{DY23,DY23b}). 
Moreover, because certain delicate points in earlier works on the projective case require careful clarification, we have aimed to make this paper as self-contained as possible. 
Our goal is therefore to present the arguments in a clear and accessible manner, even for readers not familiar with the compact case. 
For example, we provide a detailed account of the construction of unbounded representations arising from given non-rigid representations (cf.~\cref{sec:20251108}). 
 In addition, the construction of the spectral covering associated with a non-archimedean local system in the quasi-projective setting is presented in a self-contained manner as part of the proof of~\cref{main3} (cf.~\cref{sec:reduction}). 
This construction relies on the existence of equivariant pluriharmonic maps into Euclidean buildings in the quasi-projective setting—a generalization of the projective case established by Gromov–Schoen in \cite{GS92}—as developed by Brotbek, Daskalopoulos, Mese, and the second author in \cite{BDDM,DM24}. 

Finally, because the paper presents results from different perspectives, we provide in Figure~\ref{fig:relationships} a diagram summarizing the relationships among the main theorems for the reader’s convenience.

%\medspace
 
%Let us conclude this section by highlighting some recent developments following the paper.  
%  It is worth mentioning that the techniques developed in this paper have  substantial applications in these more recent works:
%\begin{itemize}
%	\item in \cite[Theorem A]{DY23}, the second and third authors constructed the Shafarevich morphism for complex reductive representations of fundamental groups of complex quasi-projective normal varieties.
%	\item In \cite[Theorem A]{DY23b}, the second and third authors  constructed the Shafarevich morphism for linear representations in positive characteristic of fundamental groups of complex quasi-projective varieties. Furthermore, in \cite{DY23b} they established analogous results to those in \cref{main2,main:GGL,main:special} for quasi-projective varieties whose fundamental groups admit linear (not necessarily reductive) representations in positive characteristic.
%	\item   It is worth emphasizing the   significance of pseudo Picard hyperbolicity explored in this paper compared to   pseudo Brody one.   In \cite[Theorem F]{DY23b}, we prove a conjecture by Claudon-H\"oring-Koll\'ar in the linear case: let $X$ be a smooth complex projective variety whose universal covering  is quasi-projective. If there is a faithful representation $\varrho:\pi_1(X)\to {\rm GL}_N(K)$ with $K$ any field, then up to some finite \'etale cover, the Albanese map of $X$ is locally isotrivial with fibers simply connected.  The pseudo Picard hyperbolicity proved in \cref{main2} played a significant role in this result. 
%\end{itemize}

\subsection*{Convention and notation.}In this paper, we use the following conventions and notations:
\begin{itemize}[noitemsep]
	\item Quasi-projective varieties and their closed subvarieties are usually assumed to be irreducible unless otherwise stated, while Zariski closed subsets might be reducible.
	\item Fundamental groups are always referred to as topological fundamental groups.
	\item If $X$ is a complex space, its normalization is denoted by $X^{\mathrm{norm}}$.
		\item For any scheme $S$, let  $S_{\rm red}$ denote its reduction.
%	\item All algebraic groups are assumed to be linear.
		%\item All non-Archimedean local fields are assumed to have  characteristic zero.
	\item A birational map $f:X\to Y$ between quasi-projective normal varieties is \emph{proper in codimension one} if there is an open set $Y^\circ\subset Y$ with $Y\backslash Y^\circ$ of codimension at least two such that $f$ is proper over $Y^\circ$.
	\item  
	For a finitely generated group $\Gamma$, a field $K$ and a representation $\varrho:\Gamma\to \mathrm{GL}_{N}(K)$, we say that $\varrho$ is a \emph{reductive} representation if the Zariski closure of $\varrho(\Gamma)$ in $\mathrm{GL}_N$ is a reductive algebraic group. 
	When the characteristic of $K$ is zero, this is equivalent to $\varrho$ being semisimple (cf. \cite[Theorem 22.42]{Mil17}). 
	This terminology emphasizes the geometric and group-theoretic properties of the image of $\varrho$, rather than the decomposition of the representation.
\end{itemize} 

\subsection*{Acknowledgment.}  
We would like to thank Michel Brion, Patrick Brosnan, Yohan Brunebarbe, Frédéric Campana, Benoît Claudon, Philippe Eyssidieux, Hisashi Kasuya, Chikako Mese, Gopal Prasad, Guy Rousseau, and Carlos Simpson for many helpful discussions. 
We are particularly grateful to Ariyan Javanpeykar for his   interest in this work and for his valuable remarks. 
B.C.\ and Y.D.\ acknowledge support from the ANR grant Karmapolis (ANR-21-CE40-0010). 
K.Y.\ acknowledges support from the JSPS Grant-in-Aid for Scientific Research (C)~22K03286.

\section{Preliminaries and notation}\label{sec:pre}

\subsection{Logarithmic forms.} Let \((\overline{X}, D)\) be a smooth log-pair, {\em i.e.\ } the data of a smooth projective variety  \(\overline{X}\), and of a simple normal crossing divisor \(D\) on it. We will sometimes denote by \(T_{1}(\overline{X}, D) := H^{0}(\overline{X}, \Omega_{\overline{X}}(\log D))\) the space of logarithmic forms. Similarly, if \(X\) is a smooth quasi-projective variety, we will write \(T_{1}(X) := T_{1}(\overline{X}, D)\), where \((\overline{X}, D)\) is any smooth log-pair compactifying \(X\). Note that \(T_{1}(X)\) depends only on \(X\), but not on the choice of \(\overline{X}\).

\subsection{Semi-abelian varieties and quasi-Albanese morphisms.} In this subsection we collect some results on quasi-Albanese morphisms of quasi-projective varieties  and semi-abelian varieties. We refer the readers to \cite{NW13,Fuj15} for further details. 

 A commutative complex algebraic  group $\cA$ is called a \emph{semi-abelian variety} if there is a short exact sequence of complex algebraic groups  
$$
1 \rightarrow H \rightarrow \cA \rightarrow \cA_0 \rightarrow 1
$$
where $\cA_0$ is an abelian variety and $H \cong (\bC^*)^\ell$.

Let \(X\) be a smooth complex quasi-projective variety. The quasi-Albanese variety of \(X\) is the semi-abelian variety
\[
	\cA_{X}=T_{1}(X)^*/\varrho(H_1(X,\bZ)) 
\]
where $\varrho(\gamma):=(\eta\mapsto \int_{\gamma}\eta)$. If we fix a base point \(\ast \in X\), and we denote by \(e \in \mathcal{A}_{X}\) the unit element, then there is a natural morphism of pointed varieties \((X, \ast) \to (\cA_{X}, e)\) given by integration along paths, exactly as in the projective setting. Also, any pointed morphism \((X, \ast) \to (B, e)\) to a semi-abelian variety factors uniquely through \(\mathcal{A}_{X}\) via a morphism of semi-abelian varieties \(\mathcal{A}_{X} \to B\).

Note that the image of the quasi-Albanese morphism is only a constructible set, which is neither closed nor open in general.
\medskip

The following lemma will permit to prove factorisation results for quasi-Albanese morphisms. 
\begin{lem} \label{lem:critpointalb}
	Let \(X\) be a smooth quasi-projective variety, and let \(q : X \to \cA_{X}\) be its quasi-Albanese morphism. Let \(S \subset T_{1}(X)\) be a set of logarithmic forms on \(X\), and let \(B \subset \mathcal{A}_{X}\) be the largest semi-abelian subvariety on which all \(\eta \in S\) vanish, using the natural identification \(T_{1}(\cA_{X}) \cong T_{1}(X)\). Let
	\[
		r : X \to \mathcal{A}_{X}/B
	\]
	be the quotient map. Then, for any morphism \(f : Y \to X\) from a   quasi-projective variety \(Y\), \(r(f(Y))\) is a point if and only for any \(\eta \in S\), one has \(f^{\ast}(\eta) = 0\).
\end{lem}
\begin{proof}
	After taking a resolution of singularities of $Y$, we assume that $Y$ is smooth. Choose base points on \(X\) and \(Y\) so that one has a diagram of pointed spaces
	\[
		\begin{tikzcd}
			X \arrow[r, "q"] & \cA_{X} \arrow[r, "p"] & \cA_{X}/B \\
			Y \arrow[u, "f"] \arrow[r, "u"] & \cA_{Y} \arrow[u, "g"]
		\end{tikzcd},
	\]
	where the semi-abelian varieties are pointed by their unit elements, and \(g, p\) are morphisms of algebraic groups.

	The previous diagram shows that \(p \circ q\) sends \(f(Y)\) to a point if and only if \(g(\cA_{Y}) \subset B\). By definition of \(B\), this is true if and only if \(g^{\ast}(\eta) = 0\) for any \(\eta \in S \subset T_{1}(\cA_{X}) = T_{1}(X)\). Since \(u^{\ast} g^{\ast} = f^{\ast} q^{\ast}\), and this \(q^{\ast}\) realizes the identification \(T_{1}(\cA_{X}) \cong T_{1}(X)\), this is equivalent to the second condition of the lemma.
\end{proof}
The above morphism $r:X\to \cA_X/B$ is called the \emph{partial quasi-Albanese morphism} induced by $S$.  

\begin{rem}
	 \cref{lem:critpointalb} enables us to define quasi-Albanese morphisms  for quasi-projective normal varieties $X$.     Let $Y$ be a desingularization of $X$. Let $S\subset T_1(Y)$ be the set of log one forms which vanish  on each fiber of $Y\to X$.  Then for the partial quasi-Albanese morphism $Y\to \cA$ induced by $S$, by \cref{lem:critpointalb} each fiber of $Y\to X$ is  contracted to one point.  Hence $Y\to \cA$ factors through $X\to \cA$. Note that this construction does not depend on the desingularization of $X$. Such $X\to \cA$ is called the quasi-Albanese morphism of $X$.
\end{rem}

We will also need two standard results permitting to evaluate the Kodaira dimension of subvarities of semi-abelian varieties (see \cite[Propositions 5.6.21 \& 5.6.22]{NW13}).  
\begin{proposition} \label{prop:Koddimabb}
	Let \(\cA\) be a semi-abelian variety. 
	\begin{enumerate}[label=(\alph*)]
		\item Let \(X \subset \cA\) be a closed subvariety. Then \(\overline{\kappa}(X) \geq 0\) with equality if and only if \(X\) is a translate of a semi-abelian subvariety.
		\item Let \(Z \subset \cA\) be a Zariski closed subset. Then \(\overline{\kappa}(\cA - Z) \geq 0\) with equality if and only if \(Z\) has no component of codimension \(1\).  \qed
	\end{enumerate}
\end{proposition}

\begin{lem}[{\cite[Lemmas 1.3 \& 1.4]{CDY25}}]
	Let $\alpha:X\to \cA$ be  a  (possibly non-proper)   morphism  from a smooth quasi-projective variety $X$ to a semi-abelian variety $\cA$ with   $\overline{\kappa}(X)=0$. 
	\begin{thmlist}
		\item  \label{lem:abelian pi0} if $\alpha$ is birational, then there exists a Zariski closed subset $Z\subset \cA$ of codimension at least two such  that $\alpha$ is   isomorphic
		over $\cA\backslash Z$. 
		\item \label{lem:abelian pi} If   $\dim X=\dim \alpha(X)$.	
		Then  $ \pi_1(X)$ is   abelian.  
		\qed
	\end{thmlist} 
\end{lem}

\subsection{Covers and Galois morphisms.} \label{sec:covGal}

If \(X\) is a complex variety, its fundamental group will be denoted by \(\pi_{1}(X)\), and its universal cover will usually be denoted by \(\pi : \widetilde{X} \to X\).

\begin{dfn}[Galois morphism]
	 	A finite map $\gamma: X \rightarrow Y$ of varieties is called \emph{Galois  with group $G$} if there exists a finite group $G \subset \operatorname{Aut}(X)$ such that $\gamma$ is isomorphic to the quotient map. 
	 \end{dfn}

	 If \(Y\) is a complex manifold and \(p : X \to Y\) is an étale cover, then there exists a Galois cover \(p' : X' \to Y\) factoring through \(p\) such that any other such Galois cover factors through \(p'\). In the case where \(X\) is connected, the Galois cover \(p' : X' \to Y\) is described as follows.
Let $F$ be the fiber of $p$ over $y\in Y$.
Then $\pi_1(Y,y)$ acts on $F$ transitively by lifting a closed path through $y$ starting with a point $x\in F$.
This defines a group morphism \(\varrho : \pi_{1}(Y,y) \to \mathfrak{S}(F)\), where $\mathfrak{S}$ stands for the permutation.
Now the covering $p':X'\to Y$ corresponds to the finite $\pi_1(Y,y)$-set $\mathrm{Im}(\varrho)$.
Since this set is a coset space of the normal subgroup $\mathrm{Ker}(\varrho)$, $p':X'\to Y$ is a Galois cover.
Note that we have a surjective map $\mathrm{Im}(\varrho)\to F$ of finite $\pi_1(Y,y)$-sets defined by $\sigma\mapsto \sigma(x_0)$, where $x_0\in F$ is a fixed point.
Hence \(p' \) factoring through \(p\).

			 Now, if \(p : X \to Y\) is a finite morphism of normal quasi-projective varieties, one can form the normalization of \(X\) in the Galois closure of \(\mathbb{C}(X)/\mathbb{C}(Y)\). This is a finite Galois morphism \(p' : X' \to Y\), factoring through \(X\); over the locus \(Y_{\circ}\) where \(p\) is \'etale on $p^{-1}(Y_{\circ})$, it identifies with the Galois closure defined previously. In particular, \(p'\) is \'etale over \(Y_{\circ}\).

\subsection{Fundamental groups of algebraic varieties}
Throughout this paper, we will make use of the following well-known result on fundamental groups from \cite[Theorem 2.1]{Ara16}:
\begin{lem}\label{lem:fun}
	Let $\mu:X'\to X$  be a bimeromorphic proper morphism between irreducible complex normal analytic variety.  Then $\mu_*:\pi_1(X')\to\pi_1(X)$ is surjective, and it is an isomorphism if both $X'$ and $X$ are smooth. Moreover, for any proper closed analytic subset $A\subset X$, $\pi_1(X\backslash A)\to \pi_1(X)$ is surjective.  \qed
\end{lem}

We also recall the following result. 
\begin{lem}[{see \cite[Proposition 1.3]{Cam91}, \cite[Proposition 2.10]{Kol95}}]\label{lem:finiteindex}
Let $f : X \to Y$ be a dominant morphism between   quasi-projective varieties, with $Y$ normal. Then the image of $f_{\ast} : \pi_{1}(X) \to \pi_{1}(Y)$ has finite index in $\pi_{1}(Y)$. \qed
\end{lem}

We now give the definition of a big representation of fundamental groups, also known as a generically large representation as introduced by Koll\'ar in \cite{Kol95}:
\begin{dfn}[Big representation]\label{def:big representation}
	Let $X$ be a quasi-projective normal variety.  Let $\varrho:\pi_1(X)\to G(K)$ be a representation, where $G$ is an algebraic group defined over some field $K$. We say that $\varrho$ is a \emph{big representation} if there are at most countably many Zariski  closed subvarieties $Z_i\subsetneqq X$ so that for every positive dimensional    closed subvariety $Y\subset X$ so that   $Y\not\subset \cup Z_i$, the image $\varrho\big({\rm Im}[\pi_1(Y^{\rm norm})\to \pi_1(X))] \big)$ is infinite.  The points in $X-\cup_{i}Z_i$ are called \emph{very general points} in $X$.
\end{dfn} 
\begin{rem}
	In a more recent work by the second and third authors \cite{DY23}, it has been established that, for a quasi-projective normal variety $X$ and a reductive representation $\varrho:\pi_1(X)\to {\rm GL}_N(\bC)$, if we consider any closed subvariety $Z$, the image $\varrho\big({\rm Im}[\pi_1(Z^{\rm norm})\to \pi_1(X))] \big)$ is infinite if and only if $\varrho\big({\rm Im}[\pi_1(Z)\to \pi_1(X))] \big)$. Consequently, when the representation is reductive, in \cref{def:big representation} we can define the big representation without taking the normalization of $Y$. 
\end{rem}

 \subsection{Notions of pseudo Picard hyperbolicity}
Let us first recall the definition of pseudo Picard hyperbolicity introduced in \cite{Denarxiv}.  
Let $f:\bD^*\to X$ be a holomorphic map from the punctured disk $\bD^*$ to a quasi-projective variety $X$.
If $f$ extends to  a holomorphic map $\mathbb D\to\overline{X}$ from the disk $\bD$ to some projective compactification $\overline{X}$ of $X$, then the same holds for any projective compactification $\overline{X}'$ of $X$, as $\overline{X}$ and $\overline{X}'$ are birational.

\begin{dfn}[pseudo Picard hyperbolicity]\label{def:Picard}
	Let $X$ be a  smooth quasi-projective variety, and let $\overline{X}$ be a smooth projective compactification. 
 $X$ is called \emph{pseudo-Picard hyperbolic} if there is a   Zariski closed proper subset $Z\subsetneq X$ so that any holomorphic map  $f:\bD^*\to X$ with $f(\bD^*)\not\subset X$ extends to a holomorphic map $\bar{f}:\bD\to \overline{X}$. %We also say that $X$ is \emph{Picard hyperbolic modulo $Z$}. 
	If $Z=\varnothing$, $X$ is simply called \emph{Picard hyperbolic}.
\end{dfn}
As shown in \cite[Proposition 1.7]{CD21}, 
pseudo Picard hyperbolic varieties exhibit the following algebraic properties. 

\begin{proposition} \label{extension theorem} 
	Let $X$ be a smooth quasi-projective variety that is pseudo Picard hyperbolic. Then any meromorphic map $f:Y\dashrightarrow X$ from another smooth quasi-projective variety $Y$ to $X$ with $f(Y)\not\subset \mathrm{Sp_p}(X)$ is \emph{rational}. 
\end{proposition} 
Although we have stated this proposition only for dominant meromorphic maps $f:Y\dashrightarrow X$ in \cite[Proposition 1.7]{CD21}, the same proof works for the proof of \cref{extension theorem}.  We provide it here for completeness. 
\begin{proof}[Proof of \cref{extension theorem}]
Let $\overline{Y}$ be a smooth projective compactification of $Y$ such that $D:=\overline{Y}\backslash Y$ is a simple normal crossing divisor.   Let $\overline{X}$ be a smooth projective compactification of $X$.  Note that  $f$ is rational if and only if $f$ extends to a meromorphic map $\bar{f}:\overline{Y}\dashrightarrow \overline{X}$. It suffices to check that this property holds in a neighborhood of any point of \(D\).  By \cite[Theorem 1]{Siu75}, any meromorphic map from a Zariski open set $W^\circ$ of a complex manifold $W$ to a compact K\"ahler manifold $\overline{X}$ extends to a meromorphic map from $W$ to $\overline{X}$ provided that the codimension of $W-W^\circ$ is at least 2.  It then suffices to consider the extensibility of $f$ around smooth points on $D$. Pick any such point $p\in D$ and choose a coordinate system $(\Omega;z_1,\ldots,z_n)$ centered at $p$ such that $\Omega\cap D=(z_1=0)$. The theorem follows if we can prove that $f:\bD^*\times \bD^{n-1}\dashrightarrow X$ extends to a meromorphic map $\bD^{n}\dashrightarrow \overline{X}$. 

  Denote by $S$  the indeterminacy locus 
 of $f|_{\bD^*\times \bD^{n-1}}:\bD^*\times \bD^{n-1}\dashrightarrow X$, which is a closed subvariety of $\bD^*\times \bD^{n-1}$  of codimension at least two.  Since we assume that $f(Y)\not\subset\Sp_p(X)$, there is thus a dense open set $W\subset \bD^{n-1}$ such that for any $z\in W$, each slice $\bD^*\times \{z\}\not\subset S$ and $f(\bD^*\times \{z\}-S)\not\subset \Sp_p(X)$. Then the restriction   $f|_{\bD^*\times \{z\}}:\bD^*\times \{z\}\dashrightarrow X$ is well-defined and  holomorphic. Then  $f:\bD^*\times \{z\}\to X$   extends  to  a holomorphic map $\bD\times   \{z\}\to \overline{X}$ for each $z\in W$.   We then apply the theorem of Siu in \cite[p.442,  ($\ast$)]{Siu75} to conclude that $f|_{\bD^*\times \bD^{n-1}}$ extends to a meromorphic map  $\bD^{n}\dashrightarrow \overline{X}$. This implies that  $f$ extends to a meromorphic map $\bar{f}:\overline{Y}\dashrightarrow \overline{X}$. By the Chow theorem, $f$ is rational.   
\end{proof}

A direct consequence of \cref{extension theorem}  is the following uniquness of algebraic structure of pseudo Picard hyperbolic varieties. 

\begin{cor}
Let $X$ and $Y$ be smooth quasi-projective varieties such that there exists an analytic isomorphism $\varphi:Y^{\rm an}\to X^{\rm an}$ of associated complex spaces.
Assume that $X$ is pseudo Picard hyperbolic. 
Then $\varphi$ is an algebraic isomorphism. \qed
\end{cor}

A classical result due to Borel \cite{Bor72} and Kobayashi-Ochiai \cite{KO71} is that quotients of bounded symmetric domains by torsion-free lattices are Picard hyperbolic. The second author has proved a similar result for algebraic varieties that admit a complex variation of Hodge structures.
\begin{thm}[{\cite[Theorem A]{Denarxiv}}]\label{thm:PicardVHS}
	Let $X$ be a smooth quasi-projective variety. Assume that there is a complex variation of  Hodge structures on $X$ whose period mapping is injective at one  point. Then $X$ is pseudo Picard hyperbolic. \qed
\end{thm} 

We conclude this subsection by recalling from \cref{subsec:20230427} the strong Green-Griffiths and Lang conjectures, which are fundamental problems in the study of hyperbolicity of algebraic varieties (cf. \cite[I, 3.5]{Lan97} and \cite[VIII, Conj. 1.3]{Lan97}).
Taking into account \Cref{main:GGL}, we  include the pseudo Picard hyperbolicity into the statement as well, and formulate the \emph{generalized  Green-Griffiths-Lang conjecture} as follows.  
\begin{conjecture}\label{conj:GGL}
Let $X$ be a smooth quasi-projective variety.  Then the following properties are equivalent:
	\begin{thmlist} 
		\item  $X$ is of log general type;
	\item  $X$ is pseudo Picard hyperbolic, i.e., $\Spp(X)\subsetneqq X$; 
	\item  $X$ is pseudo Brody hyperbolic, i.e., $\Sph(X)\subsetneqq X$; 
	\item  
	$\Spab(X)\subsetneqq X$; 
	\item  
	$X$ is strongly of log general type, i.e., $\Spalg(X)\subsetneqq X$.
	\end{thmlist}
\end{conjecture}

Note that if $X$ is of log general type, then the conjugate variety $X^\sigma:=X\times_\sigma\bC$ under  $\sigma\in {\rm Aut}(\bC/\bQ)$ is also of log general type.   Therefore, by \cref{conj:GGL}, if $X$ is pseudo Brody (Picard) hyperbolic, it is conjectured that $X^\sigma$ is also pseudo Brody (Picard) hyperbolic (cf. \cite[p. 179]{Lan97}). 
This problem remains quite open and is currently an active area of research.

 \section{Some factorisation results}\label{sec:fac}
 In this section, we prove several preliminary results concerning fibrations $X\to Y$ of quasi-projective varieties.
 In particular, we prove \cref{lem:kollar} and subsequently derive \cref{main} from \cref{main2}, using \cref{lem:kollar}.

\begin{lem}[Quasi-Stein factorisation]\label{lem:Stein}
	Let $f:X\to Y$ be a morphism between smooth quasi-projective varieties. Then $f$ factors through   morphisms $\alpha:X\to S$ and $\beta:S\to Y$ such that
	\begin{enumerate}[label={\rm (\alph*)}]
		\item $S$ is a quasi-projective normal variety;
		\item   $\alpha$ is a dominant morphism whose general fibers are connected, but not necessarily proper or surjective;
		\item $\beta$ is a finite morphism.
	\end{enumerate}
Such a factorisation is unique.
\end{lem}
\begin{proof}
	Let $\overline{X}$ be a partial smooth compactification of $X$ such that $f$ extends to a projective morphism $\bar{f}:\overline{X}\to Y$. Take the Stein factorization of $\bar{f}$ and we obtain a proper surjective morphism  $\bar{\alpha}:\overline{X}\to S$ with connected fibers   and a finite morphism $\beta:S\to Y$. Then $\alpha:X\to S$ is defined to be the restriction of $\bar{\alpha}$ to $X$, which is dominant with connected general fibers. It is easy to see that this construction does not depend on the choice of $\overline{X}$. 
\end{proof}
The previous factorisation will be called \emph{quasi-Stein factorisation} in this paper. 
  \begin{lem}\label{lem:normal}
  		Let $f:X\to Y$ be a dominant morphism between quasi-projective varieties such that $X$ is smooth.
  		Then for a connected component $F$ of a general fiber, one has $ {\rm Im}[\pi_1(F)\to \pi_1(X)]\triangleleft \pi_1(X)$. 
  \end{lem}
\begin{proof}
By \cref{lem:Stein}, we may assume that $f:X\to Y$ has connected general fibers. 
 There is a Zariski open set $Y^\circ\subset Y$ such that denoting $X^\circ:=f^{-1}(Y^\circ)$, the restriction $f|_{X^\circ}:X^\circ\to Y^\circ$ is a smooth morphism, and moreover is a topologically locally trivial fibration over $Y^\circ$.  
 Then the fibers of $f|_{X^\circ}$ consist of one irreducible component $F$ and satisfies a short exact sequence
 $$
 \pi_1(F)\to \pi_1(X^\circ)\to \pi_1(Y^\circ)\to 0.
 $$
 It follows that  $ {\rm Im}[\pi_1(F)\to \pi_1(X^\circ)]\triangleleft \pi_1(X^\circ)$. Note that $\pi_1(X^\circ)\to \pi_1(X)$ is surjective by \cref{lem:fun}.  Hence
$ {\rm Im}[\pi_1(F)\to \pi_1(X)]\triangleleft \pi_1(X)$. 
\end{proof}

\begin{lem}\label{lem:factor0}
	Let $f:X\to Y$ be a dominant morphism between smooth  quasi-projective varieties with  general fibers connected.  Let $\varrho:\pi_1(X)\to G(K)$ be a representation whose image is torsion free, where $G$ is a linear algebraic group defined on some field $K$.  If for the general fiber $F$, $\varrho({\rm Im}[\pi_1(F)\to \pi_1(X)])$ is trivial, then there is a commutative diagram
	\[
		\begin{tikzcd}
			X' \arrow[r, "\mu"] \arrow[d, "f'"] & X \arrow[d, "f"] \\
			Y' \arrow[r, "\nu"] & Y
		\end{tikzcd}
	\]
	where
	\begin{enumerate}[label=(\alph*)]
		\item \(\mu\) is a proper birational morphism,
		\item  \(\nu\) is a birational, not necessarily proper morphism ;
		\item  \(f'\) is dominant; 
	\end{enumerate} 
	 and a representation \(\tau : \pi_{1}(Y') \to G(K)\) such that $f'^*\tau=\mu^*\varrho$.  
	 \end{lem}
\begin{proof} 
	\noindent{\em Step 1. Compactifications and first reduction step.} 
	  We take a partial smooth compactification $\overline{X}$ of $X$  so that $f$ extends to a projective surjective morphism $\bar{f}:\overline{X}\to Y$ with connected fibers.   	 

	  \begin{claim} We may assume that \(\bar{f}:\overline{X} \to Y\) is equidimensional.
	  \end{claim}
	  Indeed, by Hironaka-Gruson-Raynaud's flattening theorem, there is a birational proper morphism $Y_1\to Y$ from a smooth quasi-projective variety $Y_1$ so that for the irreducible component  $T$ of $\overline{X}\times_YY_1$ which dominates $Y_1$, the induced morphism $f_{T} :=T\to Y_1$ is surjective, proper and flat.         In particular, the fibers of $f_{T}$ are equidimensional.  
	 Consider the normalization map $\nu:\overline{X}_1\to T$. Then the induced morphism $f_1:\overline{X}_1\to Y_1$ still has equidimensional fibers.  Write $\mu:\overline{X}_1\to \overline{X}$ for the induced proper birational morphism, and let $X_1:=\mu^{-1}(X)$.  Note that $\pi_1(X_1)\to \pi_1(X)$ is an isomorphism by \cref{lem:fun} below.

	 Then one has a diagram
	 \[
		 \begin{tikzcd}
		 X_{1} \arrow[r] \arrow[d] & X \arrow[d] \\
		 Y_{1} \arrow[r] & Y
		 \end{tikzcd}
         \]
	 where the horizontal maps are proper birational, and the two spaces on the left satisfy the hypotheses of the proposition if we take the representation induced on \(\pi_{1}(X_{1})\). Clearly, it suffices to show the result where \(X\) (resp. \(Y\)) is replaced by \(X_{1}\) (resp. \(Y_{1}\)). In the following, we may also replace \(\overline{X}\)  (resp. $Y$) by \(\overline{X}_{1}\) and $Y$ (resp. $Y_1$). 
	 \medskip
	
	\noindent
	{\em Step 2. Induced representation on an open subset of \(Y\).}  Consider a Zariski open set $Y^\circ\subset Y$ such that    $X^\circ:=f^{-1}(Y^\circ)$   is  a topologically locally trivial fibration over $Y^\circ$ with connected fibers $F$. Then  we have a short exact sequence
$$
\pi_1(F)\to \pi_1(X^\circ)\to \pi_1(Y^\circ)\to 0
$$
	By our assumption,  $\varrho({\rm Im}[\pi_1(F)\to \pi_1(X)])$ is trivial. Hence we can pass to the quotient, which yields a representation $\tau:\pi_1(Y^\circ)\to G(K)$ so that $\varrho|_{\pi_1(X^\circ)}= f^{\ast}\tau$. 
	\medskip

	\noindent
	{\em Step 3. Reducing \(Y\), we may assume that all divisorial components of \(Y - Y^{\circ}\) intersect \(f(X)\).} Denote by  $E$ the sum of prime divisors of $Y$ contained in the complement $Y\backslash Y^\circ$. We decompose $E=E_1+E_2$ so that $E_1$ is the sum of prime divisors of $E$ that do not intersect \(f(X)\).  We replace $Y$ by $Y\backslash E_1$.  Then for any prime divisor $P$ contained in $Y\backslash Y^\circ$, $f^{-1}(P)\cap X\neq \varnothing$.  
	\medskip

	\noindent
	{\em Step 4. Extension of the representation to the whole \(\pi_{1}(Y)\).}

	Let \(D\) be a divisorial component of \(Y - Y^{\circ}\). By what has been said above, \(f^{-1}(D) \neq \varnothing\). Since \(\bar{f} : \overline{X }\to Y\) is equidimensional, then for any prime component \(P\) of \(f^{-1}(D)\), the morphism \(f|_{P} : P \to D\) is dominant. Also, since \(X\) is normal, \(X\) is smooth at the general points  of \(P\).

This allows to find a point \(x \in P_{reg}\) (resp. \(y \in D_{reg}\)) with local coordinates \((z_{1}, \dotsc z_{m})\) (resp. \((w_{1}, \dotsc, w_{n})\)) around \(x\) (resp. \(y\)), adapted to the divisors, such that \(f^{\ast}(w_{1}) = z_{1}^{k}\) for some \(k \geq 1\).
	Hence the meridian loop $\gamma$ around the general point of \(P\)  is mapped to $\eta^k$ where $\eta$ is the meridian loop around \(D\). On the other hand, since \(\gamma\) is trivial in $\pi_1(X)$, it follows that
$$
0=\varrho(\gamma)=\tau(\eta^k).
$$    
Hence $\tau(\eta)$ is a torsion element. Since we have assumed that the image of $\varrho$ does not contain torsion element, then $\tau(\eta)$ has to be trivial. Hence $\tau$ extends to the smooth locus of $D$.

	Since this is true for any divisorial component \(D \subset Y - Y^{\circ}\), this shows that \(\tau\) extends to \(\pi_{1}(Y^{\circ\circ})\), where \(Y^{\circ} \subset Y^{\circ\circ}\) and \(Y - Y^{\circ\circ}\) has codimension \(\geq 2\) in \(Y\). However, since \(Y\) is smooth, we have \(\pi_{1}(Y) \cong \pi_{1}(Y^{\circ\circ})\), so \(\tau\) actually extends to \(\pi_{1}(Y)\).
\end{proof}
Note that the above proof is   more difficult than the compact cases, since $f(X)$ is only a constructible subset of $Y$ and $f$ is not proper. 

%\begin{lem}\label{lem:factor}
%	Under the assumptions of Lemma~\ref{lem:factor0}, if we assume that \(X\) is smooth, one may moreover assume that \(X'\) is smooth.
%\end{lem}
%\begin{proof}
%	Indeed, if \(X' \to X\) is the morphism provided by Lemma~\ref{lem:factor0}, let \(\widetilde{X}'\) be a resolution of singularities of \(X'\):
%	\[
%		\begin{tikzcd}
	%		\widetilde{X}' \arrow[r, "p"] & X' \arrow[r, "\mu"] \arrow[d, "f'"] & X \\
	%					 & Y' &
	%	\end{tikzcd}
%	\]
%	Since \(X\) is smooth, and \(\mu \circ r\) is proper, one has an isomorphism \(p_{\ast}\mu_{\ast} : \pi_{1}(\widetilde{X}') \to \pi_{1}(X)\), and
%	\[
%		p^{\ast} f'^{\ast} \tau = p^\ast \mu^{\ast} \varrho.
%	\]
%	Thus, one can replace \(X'\) with \(\widetilde{X}'\) and \(f'\) with \(f' \circ p\).
%\end{proof}

 Based on \cref{lem:factor0} we prove the following factorisation result which is important in proving \cref{main}. 
 \begin{proposition}\label{lem:kollar}
 	Let $X$ be a quasi-projective normal variety.  Let $\varrho:\pi_1(X)\to G(K)$ is a representation, where $G$ is a linear algebraic group defined over a field $K$ of zero characteristic.  Then there is a diagram
	 \[
		 \begin{tikzcd}
			 \widetilde{X} \arrow[r, "\mu"] \arrow[d, "f"] & \widehat{X} \arrow[r, "\nu"] & X\\
			 Y                                             &  &
		 \end{tikzcd}
	 \]
	 where \(Y\) and \(\widetilde{X}\) are smooth quasi-projective varieties, and
	 \begin{enumerate}[label=(\alph*)]
 		\item $\nu:\widehat{X}\to X$ is a finite étale cover;
		\item \(\mu : \widetilde{X} \to \widehat{X}\) is a birational proper morphism;
		\item $f : \widetilde{X} \to Y$ is a dominant morphism with connected general fibers;
 	\end{enumerate}
		such that there exists a big representation \(\tau : \pi_{1}(Y) \to G(K)\) with  $f^*\tau=(\nu\circ \mu)^*\varrho$.   \end{proposition}
Note that when $X$ is projective, this result is proved in \cite[Theorem 4.5]{Kol93}.   

\begin{proof}[Proof of \cref{lem:kollar}]
	{\em Step 1. We may assume that \(\varrho\) has a torsion free image.}
	Since the image $\varrho(\pi_1(X))$ is a finitely generated linear group, by a theorem of Selberg, there is a finite index normal subgroup $\Gamma\subset \varrho(\pi_1(X))$ which is torsion free.  Take an \'etale cover $\nu:\widehat{X}\to X$ with fundamental group $\pi_1(\widehat{X})=\varrho^{-1}(\Gamma)$.  Then the image of $\nu^*\varrho$ is torsion free.

	In the following, we may replace \(X\) by \(\widehat{X}\) and \(\varrho\) by \(\nu^{\ast}\varrho\) to assume that \(\mathrm{Im}(\varrho)\) is torsion free.
	\medskip

\noindent	{\em Step 2. We find a  suitable model of the Shafarevich map.}
	Denote by $H:={\ker} (\varrho)$, which is a normal subgroup of $\pi_1(X)$.  We apply \cite[Corollary 3.5 \& Remark 4.1.1]{Kol93} to conclude that there is a normal quasi-projective variety ${\rm Sh}^H(X)$ and a dominant rational map ${\rm sh}_{X}^H:X\dashrightarrow {\rm Sh}^H(X)$ so that
\begin{enumerate}[label=(\roman*)]
	\item there is a Zariski open set $X^\circ\subset X$ which does not meet the indeterminacy locus of ${\rm sh}_{X}^H|_{X^{\circ}}$ such that the fibers of ${\rm sh}_{X}^H|_{X^{\circ}}$
  	are closed in $X$; 
	\item  ${\rm sh}_{X}^H : X^{\circ} \to \mathrm{Sh}^{H}(X)$ has connected general fibers;
	\item \label{VG} there are at most countably many  closed subvarieties $Z_i\subsetneqq X$ so that for  every    closed subvariety $W\subset X$ such that   $W\not\subset \cup Z_i$, the image $\varrho\big({\rm Im}[\pi_1(W^{\rm norm})\to \pi_1(X))] \big)$ is finite (and thus trivial since the image of $ \varrho$ is torsion free by Step 1) if and only if ${\rm sh}_{X}^H(W)$ is a point. 
\end{enumerate}

Let us first take a resolution of singularities $Y_1\to {\rm Sh}^H(X)$, and then a birational proper morphism $X_1\to X$ from a smooth quasi-projective variety $X_1$ which resolves the indeterminacy of $X\dashrightarrow Y_1$. Then the induced dominant morphism
 $
f_1:X_1\to  Y_1
$ 
	fulfills the conditions of \cref{lem:factor0}. Thus, one has a commutative diagram as follows:
	\[
		\begin{tikzcd}
			X_{1}' \arrow[r, "\mu"] \arrow[d, "f_{1}'"] & X_{1} \arrow[r, "q"] \arrow[d, "f_{1}"] & X \arrow[d, dashed] \\ 
			Y_{1}' \arrow[r, "\nu"]            &  Y_{1} \arrow[r] & \mathrm{Sh}^{H}(X) 
		\end{tikzcd}
	\]
	where 
	\begin{enumerate}[label=(\alph*)]
	\item \(\mu\) is a proper birational morphism; 
	\item \(\nu\) is a birational (not necessarily proper) morphism,  
	\item  \(f_{1}'\) is a dominant morphism.
	\end{enumerate}
	Moreover, one has a representation $\tau:\pi_1(Y_{1}')\to G(K)$ so that \((f_{1}')^*\tau\) identifies with the pullback \(\mu^{\ast} q^{\ast} \varrho\).  
	
Therefore, there   are at most countably many  closed subvarieties $Z'_i\subsetneqq X_1'$ such that for  each   closed subvariety $Z\subset X_1'$ with 
\begin{itemize}
	\item $Z\not\subset \cup Z'_i$;
	\item   $f_1'(Z)$ is not a point,
\end{itemize}
 the map $Z\to (q\circ\mu)(Z)$ is proper birational, and ${\rm sh}_X^H((q\circ\mu)(Z))$ is not a point.  Write $\widetilde{Z}:=(q\circ\mu)(Z)$.  By \Cref{VG} $\varrho\big({\rm Im}[\pi_1(\widetilde{Z}^{\rm norm})\to \pi_1(X))] \big)$ is infinite.  By \cref{lem:fun} below, $\pi_1(Z^{\rm norm})\to \pi_1(\widetilde{Z}^{\rm norm})$ is surjective. Hence $(q\circ\mu)^*\varrho\big({\rm Im}[\pi_1( {Z}^{\rm norm})\to \pi_1(X_1'))] \big)$ is infinite. 
	\medskip

	\noindent
	{\em Step 3. We check that \(\tau\) is big and obtain the required diagram.} Let \(\widetilde{X} := X_{1}'\), \(Y := Y_{1}'\) and \(f := f_{1}'\). To simplify the notation, let us denote by the same letter \(\varrho\) the pullback \(\varrho : \pi_{1}(\widetilde{X}) \to G(K)\). Hence $\varrho=f^*\tau$. Recall that by Step 2 there   are at most countably many  closed subvarieties $Z'_i\subsetneqq \widetilde{X}$ such that for  every   closed subvariety $Z\subset \widetilde{X}$ with $Z\not\subset \cup Z'_i$ and $f(Z)$ not a point, the image $\varrho\big({\rm Im}[\pi_1(Z^{\rm norm})\to \pi_1(\widetilde{X}))] \big)$ is infinite. 
	
	Since $f:\widetilde{X}\to Y$ is dominant with connected general fibers,   for a subvariety $W$ of $Y$ containing a very general point, there is an irreducible component $Z$ of $f^{-1}(W)$ which dominates $W$ and $Z\not\subset \cup Z'_i$. Hence $\tau({\rm Im}[\pi_1(Z^{\rm norm})\to \pi_1(Y)])=\varrho({\rm Im}[\pi_1(Z^{\rm norm})\to \pi_1(\widetilde{X})])$ is infinite.  Since the morphism $\pi_1(Z^{\rm norm})\to \pi_1(Y)$ factors through $\pi_1(W^{\rm norm})\to \pi_1(Y)$, it follows that  $\tau({\rm Im}[\pi_1(W^{\rm norm})\to \pi_1(Y)])$ is infinite. Therefore $\tau$ is big.
	\end{proof}
	
	\begin{proof}[Derivation of \cref{main} from \cref{main2}]
	The derivation is straightforward using \cref{lem:kollar}.
	Indeed, given a Zariski dense representation $\varrho:\pi_1(X)\to G(\bC)$ in the statement of \cref{main}, we apply \cref{lem:kollar} to get the objects: a  finite \'etale cover \(\nu:\widehat{X}\to X\), a birational and proper morphism \(\mu:\widetilde{X}\to \widehat{X}\), a dominant morphism $f:\widetilde{X}\to Y$  with connected general fibers, and a big representation \(\tau : \pi_{1}(Y) \to G(\bC)\) such that \(f^{\ast} \tau = (\nu\circ\mu)^{\ast}\varrho\).
	Since $\varrho$ is Zariski dense, the property \(f^{\ast} \tau = (\nu\circ\mu)^{\ast}\varrho\) yields that $\tau$ is also Zariski dense.
	Since $G$ is semisimple, we may apply \cref{main2} for \(\tau : \pi_{1}(Y) \to G(\bC)\) to conclude that $Y$ is strongly of log-general type, i.e., $\Spalg(Y)\subsetneqq Y$, and pseudo Picard (hence Brody) hyperbolic.
	Moreover by our convention that $G$ is non-trivial, we have $\dim Y>0$.
	This, together with the definition of weakly special, shows that $X$ is not weakly special.
	Moreover $X$ cannot contain a Zariski-dense entire curve. This is because the existence of such a curve in $X$ would induce one in $\widetilde{X}$ hence in $Y$, which contradicts that $Y$ is pseudo Brody hyperbolic.
	\end{proof}
\begin{rem}
	In \cite{DY23b}, the second and third authors extended \cref{lem:kollar} to the case where the field $K$ has positive characteristic. 
	The main difficulty is that Selberg's theorem on the virtual torsion-freeness of finitely generated linear groups in characteristic zero does not hold in positive characteristic.
\end{rem}

	\section[A reduction theorem]{Non-abelian Hodge theories in the non-archimedean setting}\label{sec:reduction}
	 	In this section, we prove \cref{main3}. Our argument relies on two key results from \cite{BDDM,DM24}: 
	\begin{itemize}
		\item the existence of a $\varrho$-equivariant harmonic map 
		$u : \widetilde{X} \to \Delta(G)$ 
		from the universal cover of $X$ to the Bruhat-Tits building $\Delta(G)$ of $G$, whose energy growth at infinity is logarithmic;
		\item the construction of logarithmic symmetric differential forms on $X$ via this harmonic map $u$.
	\end{itemize} 
	When $X$ is compact, \cref{main3} was proved by Katzarkov \cite{Kat97}, Zuo \cite{Zuo96}, and Eyssidieux \cite{Eys04}.  
	In the non-compact case, several subtle and technically involved issues arise, and we provide as many details as possible in the proof of \cref{main3}.
	
As an independent interest and for potential further applications, we start from a preliminary of detailed and more general construction of spectral coverings in \cref{subsec:spectral}. 
These constructions are crucial for the proof of \cref{main6} in the next section.

\subsection{Preliminary for constructing Spectral covering}\label{subsec:spectral}

Let $X$ be a quasi-projective variety.
Let $\mathcal{E}$ be an algebraic locally free sheaf over $X$. 
We consider a graded ring $B_{X,\mathcal{E}}=\bigoplus_{i\geq 0} \Gamma(X,\Sym^i \mathcal{E})$ and the graded ring extension $B_{X,\mathcal{E}}[T]=B_{X,\mathcal{E}}\otimes_{\mathbb C}\mathbb C[T]$ equipped with a natural grading, where $T$ has weight one.
We consider a homogeneous element of the form
	\begin{align}\label{eq:char}
		P(T) = T^n+ \sigma_1 T^{n-1} + \cdots + \sigma_n\in B_{X,\mathcal{E}}[T]
	\end{align}
	where $\sigma_i \in \Gamma(X, \Sym^i \mathcal{E})$.
	Let $U\subset X$ be a non-empty open subset (in the complex topology).
	We say that $P(T)$ splits over $U$, if there exist holomorphic sections $\eta_1,\ldots,\eta_n\in \Gamma(U,\mathcal{E}^{\mathrm{an}}|_U)$ such that 
	$$
	P(T)=(T+\eta_1)\cdots(T+\eta_n)
	$$
	holds in the ring $B_{U,\mathcal{E}^{\mathrm{an}}|_U}[T]$, where $B_{U,\mathcal{E}^{\mathrm{an}}|_U}=\bigoplus_{i\geq 0} \Gamma(U,\Sym^i\mathcal{E}^{\mathrm{an}}|_U)$.
Let $\varphi:X'\to X$ be a morphism from another quasi-projective variety.
Then we define $\varphi^*P(T)$ by
$$
\varphi^*P(T)=T^n+ (\varphi^*\sigma_1) T^{\ell-1} + \cdots + (\varphi^*\sigma_n),
$$
where $\varphi^*\sigma_i \in \Gamma(X', \Sym^i (\varphi^*\mathcal{E}))$ is a pull-back of $\sigma_i$.
Then $\varphi^*P(T)\in B_{X',\varphi^*\mathcal{E}}[T]$, where $B_{X',\varphi^*\mathcal{E}}=\bigoplus_{i\geq 0} \Gamma(X',\Sym^i(\varphi^* \mathcal{E}))$.

\begin{proposition}\label{lem:spectral}
Let $X$ be a smooth quasi-projective variety.
Let $\mathcal{E}$ be an algebraic locally free sheaf over $X$. 
Let $P(T)$ be a homogeneous polynomial of degree $n$ as given in \eqref{eq:char}.
	Suppose there exists a non-empty open subset $U\subset X$ such that $P(T)$ splits over $U$.
	Then there exists a finite Galois cover $\pi:\Sigma\to X$ from a quasi-projective normal variety $\Sigma$, satisfying the following conditions.
	\begin{thmlist}
	\item
   $\pi^*P(T)$ splits over $\Sigma$, i.e., there exists $\eta_1,\ldots,\eta_n\in\Gamma(\Sigma, \pi^*\mathcal{E})$ such that 
   $$
	\pi^*P(T)=(T+\eta_1)\cdots(T+\eta_n).
	$$
           \item\label{item:extension}
The subset $\{\eta_1,\ldots,\eta_n\}\subset \Gamma(\Sigma, \pi^*\mathcal{E})$ is invariant under the natural action of $\mathrm{Aut}(\Sigma/X)$ on $\Gamma(\Sigma, \pi^*\mathcal{E})$, and the induced action of $\mathrm{Aut}(\Sigma/X)$ on this subset is faithful.
\item\label{item:20251211}
If $\varphi:Y\to X$ is a dominant morphism from a quasi-projective normal variety $Y$ such that $\varphi^*P(T)$ splits over $Y$, then the morphism $Y\to X$ factors as $Y\to \Sigma\overset{\pi}{\to} X$.
\item\label{item:ramified}
	The Galois morphism $\pi:\Sigma \to X$ is \'etale outside
	\[
		R := \{ z \in \Sigma \mid \exists\, i \neq j \text{ such that } \eta_i\not=\eta_j \text{ and } (\eta_i - \eta_j)(z) = 0 \},
		\]
		and it satisfies $\pi^{-1}(\pi(R)) = R$.
		\end{thmlist}
\end{proposition}

We remark that the assertion (iii) determines the covering $\Sigma\to X$ uniquely.
Before proving this proposition, we begin with three preliminary lemmas on polynomials whose coefficients lie in various fields.
The main object is a homogeneous polynomial $f$ in variables $T,x_1,\ldots,x_l$ given by:
\begin{equation}\label{eqn:2025111410}
f=T^n+a_1T^{n-1}+\cdots+a_n,
\end{equation}
where each $a_i$ is a homogeneous polynomial in $x_1,\ldots,x_l$ of degree $i$.
We are interested in complete factorization into linear forms of the following type:
\begin{equation}\label{eqn:2025111411}
f=(T+g_1)\cdots(T+g_n),
\end{equation}
where each $g_j$ is of the form 
\begin{equation}\label{eqn:2025111412}
g_j=c_{j1}x_1+\cdots+c_{jl}x_l.
\end{equation}
We say that $f$ completely factorizes into linear forms over a field $L$ if $c_{ji}\in L$ for all $c_{ji}$.

\begin{lem}\label{lem:20251113}
Let $F$ be a non-archimedean valuation field, and let $R$ be the valuation ring of $F$.
Let $f\in F[T,x_1,\ldots,x_l]$ be a homogeneous polynomial of degree $n$ as in \eqref{eqn:2025111410}.
Suppose that $f$ completely factorizes into linear forms as in \eqref{eqn:2025111411} over $F$.
If $f\in R[T,x_1,\ldots,x_l]$, then $g_j\in R[x_1,\ldots,x_l]$ for all $j=1,\ldots,n$.
\end{lem}

\begin{proof}
We denote by $|\cdot|$ the norm on $F$, and by $||\cdot||$ the Gauss  norm on the polynomial ring $F[x_1,\ldots,x_l]$.
For each $g_j$ as in \eqref{eqn:2025111412}, by changing the indexes $j$, we may assume that $||g_1||\geq ||g_2||\geq \cdots\geq ||g_n||$.
It is enough to show $||g_1||\leq 1$.

Assume, for the sake of contradiction, that $||g_1||>1$. Let $m$ be the maximum index such that $||g_m||>1$. 
Since the Gauss norm is a multiplicative norm, we have $||g_1\cdots g_m||>||g_{j_1}\cdots g_{j_m}||$ for any sequence $1\leq j_1<\ldots <j_m\leq n$ other than $j_i=i$.
Hence the strong triangle inequality of Gauss norm implies 
$||a_m||=||g_1||\cdots||g_m||>1$.
This contradicts the assumption $a_m\in R[x_1,\ldots,x_l]$.
Therefore, we have $||g_j||\leq 1$ for all $j=1,\ldots,n$, which means $g_j\in R[x_1,\ldots,x_l]$.
\end{proof}

\begin{lem}\label{lem:202511131}
Let $K$ be a field and let $L$ be a field extension of $K$.
Let $f\in K[T,x_1,\ldots,x_l]$ be a homogeneous polynomial of degree $n$ as in \eqref{eqn:2025111410}.
Suppose that $f$ completely factorizes into linear forms as in \eqref{eqn:2025111411} over $L$.
Then $f$ completely factorizes into linear forms over the field $\overline{K}\cap L$.
In particular, $f$ completely factorizes into linear forms over $\overline{K}$.
\end{lem}

\begin{proof}
We may assume that $L$ is generated over $K$ by the $c_{ji}$ from \eqref{eqn:2025111412}. 
Suppose, contrary to our claim, that $L$ is transcendental over $K$. 
Without loss of generality, by extending $K$ finitely, we may assume $\mathrm{tr.deg}_K L=1$. 
Let $c_{ji} \in L$ be transcendental over $K$. 
We can then construct a discrete valuation on $L$ which restricts to the trivial valuation on $K$ and satisfies $|c_{ji}|>1$. 
This provides a contradiction to \cref{lem:20251113}. 
Therefore, $L$ is algebraic over $K$.
\end{proof}

\begin{lem}\label{lem:202511141}
Let $F$ be a complete non-archimedian discrete valuation field and let $R\subset F$ be the valuation ring of $F$.
Assume that the residue field of $F$ has characteristic zero.
Let $f\in R[T,x_1,\ldots,x_l]$ be a homogeneous polynomial of degree $n$ as in \eqref{eqn:2025111410} such that $f$ completely factorizes into linear forms as in \eqref{eqn:2025111411} over $\overline{F}$.
Suppose that distinct $g_i$ and $g_j$ have distinct reductions $\overline{g_i}\not=\overline{g_j}$ in $k[x_1,\ldots,x_l]$ under the reduction map, where $k$ is the residue field of the non-archimedian valuation field $\overline{F}$.
Then $f$ completely factorizes into linear forms over $F^{ur}$, where $F^{ur}\subset \overline{F}$ is the maximal unramified extension of $F$.
\end{lem}

\begin{proof}
We first assume that $f$ is irreducible in $F[T,x_1,\ldots,x_l]$. 
In this case, we have $g_i\not=g_j$ for $i\not=j$. 
Hence $\overline{g_i}\not=\overline{g_j}$ for $i\not=j$. 
Let $\sigma\in \mathrm{Gal}(\overline{F}/F^{ur})$. 
Then $\sigma(g_i)\in \{g_1,\ldots,g_n\}$, and the reduction is fixed: $\overline{\sigma(g_i)}=\overline{g_i}$ as $\sigma$ acts trivially on $k$.
Hence $\sigma(g_i)=g_i$. 
Thus, $g_i\in F^{ur}[x_1,\ldots,x_l]$, which proves the lemma when $f$ is irreducible.

For the general case, let $f=f_1^{\nu_1}\cdots f_k^{\nu_k}$ be the irreducible decomposition of $f$ in $F[T,x_1,\ldots,x_l]$. 
By the argument for the irreducible case, each factor $f_i$ completely factorizes into linear forms over $F^{ur}$. 
This completes the proof.
\end{proof}

Now we return to the geometric situation.
Let $X$ be a quasi-projective variety and let $\mathcal{E}$ be an algebraic locally free sheaf on $X$.
Assume that we are given a polynomial $P(T)$ as in \eqref{eq:char}.
Let $K$ be the rational function field of $X$ and let $\xi:\mathrm{Spec}(K)\to X$ be the schematic generic point.
Set $E_K=\Gamma(\mathrm{Spec}(K),\xi^*\mathcal{E})$, which is a vector space over the field $K$.
If we choose a basis $x_1,\ldots,x_l\in E_K$, then $B_{K,\xi^*\mathcal{E}}=\bigoplus_{i\geq 0} \Sym^i E_K$ is identified with the polynomial ring $K[x_1,\ldots,x_l]$ over $K$, where $l=\mathrm{rank}(\mathcal{E})$.
We have the pull-back $\xi^*P(T)\in B_{K,\xi^*\mathcal{E}}[T]$, which is a homogeneous polynomial of degree $n$.

Now if $P(T)$ splits over $X$, then $\xi^*P(T)\in K[T,x_1,\ldots,x_l]$ completely factorizes into linear forms over $K$.
The following lemma shows that the converse is also true.

\begin{lem}\label{lem:202511132}
Let $X$ be a quasi-projective normal variety.
Let $\mathcal{E}$ be an algebraic locally free sheaf over $X$. 
Let $P(T)=\sum_i\sigma_iT^i$ be the polynomial given in \eqref{eq:char}.
Suppose $\xi^*P(T)$ completely factorizes into linear forms over $K$, where $\xi:\mathrm{Spec}(K)\to X$ is the schematic generic point of $X$.
Then $P(T)$ splits over $X$.
\end{lem}

\begin{proof}
Since $\xi^*P(T)$ completely factorizes into linear forms over $K$, there exist $\eta_1,\ldots,\eta_n\in E_K$ such that
$$\xi^*P(T)=(T+\eta_1)\cdots(T+\eta_n)\in B_{K,\xi^*\mathcal{E}}[T].$$
Each $\eta_i$ is a rational section of $\mathcal{E}$. 
Let $U\subset X$ be the maximal Zariski open set over which all $\eta_i$ are holomorphic sections. 
We aim to show that $X\setminus U$ contains no divisor.
To show this, let $D \subset X$ be an arbitrary divisor. 
Let $\zeta$ be the generic point of $D$, and let $R=\mathcal{O}_{X,\zeta}$ be the associated local ring. 
Since $X$ is normal, $R$ is a discrete valuation ring of the rational function field $K$.
Let $\mathcal{E}_{\zeta}$ be the stalk at $\zeta$, which is a free $R$-module.
Hence we may write as $\mathcal{E}_{\zeta}=Rx_1+\cdots+Rx_l$ so that
\begin{equation}\label{eqn:20251117}
\bigoplus_{i\geq 0} \Sym^i \mathcal{E}_{\zeta}=R[x_1,\ldots,x_l].
\end{equation}
This yields the identification $B_{K,\xi^*\mathcal{E}}=K[x_1,\ldots,x_l]$.
Since all $\xi^*\sigma_i \in K[x_1,\ldots,x_l]$ are contained in $R[x_1,\ldots,x_l]$, \cref{lem:20251113} implies that every $\eta_i$ is also contained in $R[x_1,\ldots,x_l]$. 
This means that $\eta_i$ is regular at $\zeta$, i.e., $\zeta \in U$.
Since this holds for every divisor $D$, we conclude that $X \setminus U$ contains no divisor. 
Consequently, the codimension of $X \setminus U$ in $X$ is greater than one.

Now for each Zariski open $\mathrm{Spec}(A)\subset X$ on which $\mathcal{E}$ is trivial, we note that $\eta_i$ is a section of the locally free sheaf $\mathcal{E}|_{\mathrm{Spec}(A)}$ over $\mathrm{Spec}(A)\setminus U$.
Since $A$ is normal, such section extends over $\mathrm{Spec}(A)$ (cf. \cite[Theorem 6.45]{GW10}).
This shows that  $\eta_i$ is a section of $\mathcal{E}$ over $X$.
As the natural ring homomorphism $B_{X,\mathcal{E}}[T]\to B_{K,\xi^*\mathcal{E}}[T]$ is injective, we deduce that $P(T)=\prod_{i=1}^n (T+\eta_i)$ in $B_{X,\mathcal{E}}[T]$.
\end{proof}

\begin{proof}[Proof of \cref{lem:spectral}]
We first construct $\pi:\Sigma\to X$.
By shrinking $U$ if necessary, we may assume that $U$ is contained in some affine open subset over which $\mathcal{E}$ is a free sheaf.
Let $K$ be the rational function field of $X$ and let $L$ be the field of all meromorphic functions on $U$.
Then we have $K\subset L$.
Let $\xi:\mathrm{Spec}(K)\to X$ be the schematic generic point of $X$.
Since $P(T)$ splits over $U$, the homogeneous polynomial $\xi^*P(T)\in B_{K,\xi^*\mathcal{E}}[T]$ completely factors into linear forms over the field extension $L$. Thus, by \cref{lem:202511131}, $\xi^*P(T)$ completely factors into linear forms over $\overline{K}$. Let $F\subset \overline{K}$ be the field generated over $K$ by the coefficients of the linear forms appearing in the factorization.
Then $F$ is a finite Galois extension of $K$, over which $\xi^*P(T)$ completely factors into linear forms as in \eqref{eqn:2025111411}.
These linear forms $g_1,\ldots,g_n$, using the notation in \eqref{eqn:2025111411}, are invariant under the Galois action of $\mathrm{Gal}(F/K)$, and the induced action of $\mathrm{Gal}(F/K)$ on the set $\{g_1,\ldots,g_n\}$ is faithful.
Let $\pi:\Sigma\to X$ be the normalization of $X$ in $F$.
Then $\Sigma$ is normal and $\pi:\Sigma\to X$ is finite Galois with $\mathrm{Aut}(\Sigma/X)=\mathrm{Gal}(F/K)$.
We consider $\pi^*P(T)\in B_{\Sigma,\pi^*\mathcal{E}}[T]$.
Let $\tau:\mathrm{Spec}(F)\to \Sigma$ be the schematic generic point of $\Sigma$.
Since $\tau^*\pi^*P(T)$ completely factors into linear forms over $F$, \cref{lem:202511132} yields sections $\eta_i\in\Gamma(\Sigma,\pi^*\mathcal{E})$ such that $\pi^*P(T)=\prod_{i=1}^n (T+\eta_i)$.
This is the construction of $\Sigma$ and $\eta_1,\ldots,\eta_n$.
By the construction, the set $\{\eta_1,\ldots,\eta_n\}$ is invariant under the action $\mathrm{Aut}(\Sigma/X)$, and the induced action of $\mathrm{Aut}(\Sigma/X)$ on this set is faithful.

Next suppose $\varphi:Y\to X$ is a dominant morphism from a normal variety such that $\varphi^*\mathcal{E}$ splits over $Y$.
Then its rational function field $F'$ is a field extension of $K$, and $\xi^*P(T)$ completely factors into linear forms over $F'$.
By the definition of $F$, we have $F\subset F'\cap \overline{K}$ (cf. \cref{lem:202511131}).
In particular, $F\subset F'$.
Hence $Y\to X$ factors $\Sigma\to X$.

Let $D\subset X$ be a prime divisor not contained in $\pi(R)$.
We take the schematic generic point $\zeta$ of $D$, and show that $\pi:\Sigma\to X$ is unramified over $\zeta$.
Let $D'$ be a prime divisor contained in $\pi^*D$, and let $\zeta'$ be the schematic generic point of $D'$.
Let $\mathcal{O}_{X,\zeta}\subset K$ be the local ring associated to $\zeta$, and let $\mathcal{O}_{\Sigma,\zeta'}\subset F$ be that of $D'$.
Then these are discrete valuation rings.
We normalize the valuation on $F$ so that it extends the valuation on $K$.
Let $\hat{K}\subset \hat{F}$ be the associated completions.
By $D\not\subset \pi(R)$, we have $\zeta'\in \Sigma\setminus R$.
Hence distinct $\eta_i$ and $\eta_j$ have distinct reductions under $\pi^*\mathcal{E}_{\zeta'}\to \pi^*\mathcal{E}_{\zeta'}\otimes k$, where $k$ is the residue field of $\mathcal{O}_{\Sigma,\zeta'}$.
By considering the identification as in \eqref{eqn:20251117}, we apply \cref{lem:202511141} to $\xi^*P(T)$ and deduce that $\hat{F}\subset \hat{K}^{ur}$.
This shows that the local homomorphism $\pi^*:\mathcal{O}_{X,\zeta}\to \mathcal{O}_{\Sigma,\zeta'}$ is unramified.
Thus $\pi:\Sigma\to X$ is unramified over the generic points of all prime divisors on $X$ not contained in $\pi(R)$.
Therefore by the purity of the branch locus \cite[Tag 0BMB]{stacks-project}, $\pi:\Sigma\to X$ is \'etale over $X\setminus \pi(R)$.
Since the set $\{\eta_1,\ldots,\eta_n\}$ is invariant under the action $\mathrm{Aut}(\Sigma/X)$, we have $\pi^{-1}(\pi(R))=R$.
Hence $\pi:\Sigma\to X$ is \'etale outside $R$.
\end{proof}

\begin{rem} 
We remark that the covering constructed in \cref{lem:spectral} satisfies the following functorial property generalizing \cref{item:20251211}, although it will not be used in this paper.
Let $X$, $\mathcal{E}$ and $P(T)$ be the same as in \cref{lem:spectral}, in particular $P(T)$ splits over some non-empty open subset $U\subset X$.
Then we have the covering $\pi:\Sigma\to X$ and the Zariski closed subset $R\subset \Sigma$ as in \cref{lem:spectral}.
Let $\varphi:Y\to X$ be a morphism from a smooth quasi-projective variety $Y$ such that $\varphi(Y)\not\subset \pi(R)$.
Then $\varphi^*P(T)$ splits over every simply connected domain of $Y\setminus \pi(R)$.
Therefore applying  \cref{lem:spectral}, we may construct the covering $\pi':\Sigma'\to Y$.

\begin{claim}
The map $\varphi\circ\pi':\Sigma'\to X$ factors as $\Sigma'\to \Sigma\to X$.
\end{claim}

\begin{proof}
We first assume that $\varphi:Y\to X$ is a locally closed immersion and $\varphi(Y)\cap \pi(R)=\emptyset$.
Let $Y'$ be an irreducible component of $Y\times_{X}\Sigma$.
Then by \cref{lem:spectral}, the induced map $p:Y'\to Y$ is finite \'etale and Galois.
Since there exists a natural locally closed immersion $\varphi':Y'\to \Sigma$, the pull-back $\varphi^*P(T)$ splits over $Y'$ as
$$p^*\varphi^*P(T)=(T+(\varphi')^*\eta_1)\cdots(T+(\varphi')^*\eta_n).$$
Hence $p:Y'\to Y$ factors as $Y'\to \Sigma'\to Y$ (cf. \cref{item:20251211}).
Since $\varphi^*P(T)$ splits over $\Sigma'$, the action of $\mathrm{Aut}(Y'/\Sigma') \subset \mathrm{Aut}(Y'/Y)$ on the set $\{(\varphi')^*\eta_1,\ldots,(\varphi')^*\eta_n\}$ is trivial. On the other hand, the condition $\varphi'(Y') \not\subset R$ ensures the injectivity of the natural map $\{\eta_1,\ldots,\eta_n\} \to \{(\varphi')^*\eta_1,\ldots,(\varphi')^*\eta_n\}$, which implies that $\mathrm{Aut}(Y'/Y) \subset \mathrm{Aut}(\Sigma/X)$ acts faithfully on the set  $\{(\varphi')^*\eta_1,\ldots,(\varphi')^*\eta_n\}$ (cf. \cref{item:extension}).
Consequently, $\mathrm{Aut}(Y'/\Sigma')$ is trivial, which implies $Y' = \Sigma'$ by Galois theory. This establishes our claim under the assumption that $\varphi:Y \to X$ is a locally closed immersion and $\varphi(Y) \cap \pi(R) = \emptyset$.

Next we consider the general case.
We may take a non-empty Zariski open subset $W\subset Y$ such that $\varphi|_W:W\to X$ factors as $W\to Z\to X$, where $W\to Z$ is dominant and $Z\to X$ is a locally closed immersion such that $Z\cap \pi(R)=\emptyset$.
By the argument above, we obtain a finite \'etale Galois covering $Z' \to Z$. 
By \cref{item:20251211}, the morphism $W \to Z$ induces a morphism $(\pi')^{-1}(W) \to Z'$; by composing this with $Z' \to \Sigma$, we obtain a morphism $(\pi')^{-1}(W) \to \Sigma$.
Let $T\subset Y\times_{X}\Sigma$ be the Zariski closure of the natural map $(\pi')^{-1}(W)\to  Y\times_{X}\Sigma$.
Then we get a finite surjective map $T\to Y$.
Since $\Sigma'$ is normal, the map $(\pi')^{-1}(W)\to T$ extends to $\Sigma'\to T$.
Consequently, we obtain the desired map $\Sigma'\to \Sigma$ by composing $\Sigma' \to T$ with the projection map $T \to \Sigma$.
\end{proof}

On the other hand, functoriality fails in general if we use the map $\varphi^*\mathcal{E} \to \mathcal{E}'$ to construct a covering $\Sigma'' \to Y$ instead of $\Sigma'$; this is illustrated by the following example.

\begin{example}
Set $X=\mathrm{Spec}\, \mathbb C[x,y]$, $Y=\mathrm{Spec}\, \mathbb C[y]$ and $\mathcal{E}=\Omega_X$.
Consider a closed immersion $f:Y\to X$ defined by $x=0$. 
Take a polynomial $P(T)=T^2+2dyT+(dy^2-ydx^2)$ so that $P(T)=(T+dy-\sqrt{y}dx)(T+dy+\sqrt{y}dx)$.
The associated covering is given by $\Sigma = \operatorname{Spec} \mathbb{C}[x, \sqrt{y}]$.
Then $\pi(R)=(y=0)$, hence $Y\not\subset\pi(R)$.
We denote by $[f^*P(T)]$ the polynomial obtained from $f^*P(T)$ by applying the morphism $f^*\Omega_{X} \to \Omega_{Y}$ to its coefficients.
Then $[f^*P(T)]=T^2+2dyT+dy^2$.
The associated covering $\Sigma''$ is nothing but $Y = \operatorname{Spec} \mathbb{C}[y]$.
Hence  we do not have the functorial lifting $f':\Sigma''\to\Sigma$.
\end{example}
\end{rem}

\medskip

\subsection{Some finiteness criterion for subgroups of almost simple algebraic groups}

Let \(K\) be a non-archimedian local field, and let \(G\) be a reductive group defined over \(K\). One can attach to \(G\) 
 its {\em Bruhat-Tits building} \(\Delta(G)\); this is a simplicial complex obtained by glueing affine spaces isometric to \(\mathbb{R}^{N}\) (called the {\em appartments}), where \(N = \mathrm{rk}_{K}(G)\). It is a ${\rm CAT}(0)$ space.  
We refer the readers to \cite{AB08,Guy23} for more details. 

There is a natural continuous action of \(G(K)\) on \(\Delta(G)\) which acts transitively on the appartments, and which is such that the stabilizer of any point in \(\Delta(G)\) is a bounded subgroup of \(G(K)\)  by \cref{lem:AB} below.

We begin with the following definition. 
\begin{dfn}[Bounded subgroup]
	Let $G$ be a  reductive algebraic group over the non-archimedean local field $K$. Fix an embedding $G\to {\rm GL}_N$. A subgroup $H$ of $G(K)$ is \emph{bounded} if  there is an upper bound on the absolute values of the matrix entries in ${\rm GL}_N(K)$ of the elements of $H$, otherwise it is called \emph{unbounded}. 
\end{dfn}
\begin{lem}[ {{\cite[Corollary 7.2]{Rou09}}}]\label{lem:AB}
	Let $G$ be a  reductive algebraic group over a non-archimedean local field $K$.   Let $H$ be a  subgroup  of $G(K)$. Then the following properties are equivalent. 
	\begin{itemize}
		\item $H$ is bounded.
		\item $H$ is contained in a compact subgroup of $G(K)$.
		\item $H$ fixes a point in $\Delta(G)$.
	\end{itemize}
\end{lem}
 	A representation $\varrho:\pi_1(X)\to G(K)$ is \emph{(un)bounded} if its image $\varrho(\pi_1(X))$ is  a (un)bounded subgroup of $G(K)$.

The following finiteness criterion will be used to prove \cref{genericallyfinite}, which is the cornerstone of \cref{main6}.
\begin{lem}\label{lem:BT}
	Let $G$ be an almost simple algebraic group over the non-archimedean local field $K$.  Let $\Gamma\subset G(K)$ be a finitely generated subgroup so that
	\begin{itemize}
		\item  it is a Zariski dense subgroup in $G$,
		\item it is not contained in any bounded subgroup of $G(K)$. 
	\end{itemize} 
	Let $\Upsilon$ be a normal subgroup of $\Gamma$ which is \emph{bounded}. Then $\Upsilon$ must be finite. 
\end{lem}
\begin{proof} 
To prove the lemma, we may assume that $G$ is connected.  	As $G(K)$ acts on $\Delta(G)$ transitively, we denote by  $R\subset \Delta(G)$ the set of fixed points of $\Upsilon$. Since $\Upsilon$ is bounded, $R$ is not empty. $R$ is moreover closed and convex. 	 
	We will prove that $R$ is moreover invariant under $\Gamma$, i.e. $\gamma R\subset R$ for $\gamma \in \Gamma$. 
	
	For every $u\in \Upsilon$ and $\gamma\in \Gamma$, one has  $\gamma^{-1}u\gamma\in \Upsilon$ since $\Upsilon\triangleleft \Gamma$.  Then for every $x\in R$, one has $u(\gamma x)=\gamma ((\gamma^{-1}u\gamma)x)=\gamma x$.  Hence $R$ is  invariant under $\Gamma$.
	
	If $R$ is bounded, by Bruhat-Tits' fixed point theorem (see \cite[Theorem 11.23]{AB08}), $\Gamma$ fixes a point in $R$. Then $\Gamma$ is a bounded subgroup of $G(K)$ by \cref{lem:AB}, which contradicts with our assumption. Hence $R$ is unbounded. 
	
Consider the   compactification $\overline{\Delta(G)}$ of the building $\Delta(G)$ by adding points at infinity.  A point at infinity in $\overline{\Delta(G)}$ is the equivalent class of geodesic rays in $\Delta(G)$ (see \cite[\S 11.8.1]{AB08}). Write $\partial \Delta(G):=\overline{\Delta(G)}-\Delta(G)$.  Note that the action of $G(K)$ on $\Delta(G)$ induces a natural action on $\overline{\Delta(G)}$.   Let $\widetilde{R}= \overline{\Delta(G)}^\Upsilon$ be the set of fixed points  of $\Upsilon$ in the compactification $\overline{\Delta(G)}$.  Then $\widetilde R$ is closed convex. If $\widetilde{R} \cap \partial\Delta(G)$ is empty, then $\partial\Delta(G)$ is covered by some cones (in finite number by compactness of $\partial\Delta(G)$) not intersecting $\widetilde R$. 
	These cones define the topology on $\overline{\Delta(G)}$ (the cone topology). Let O be an origin in $\Delta(G)$, such a cone is the set $C(y,\epsilon)$ of all $\xi$ in $\overline{\Delta(G)}$ such that the line segment (or geodesic ray) $[O,\xi]$ contains a point $x$ at distance less than $\epsilon$ from a fixed point $y \in \Delta(G)$. To get a basis of the topology, one has also to consider open balls in $\Delta$ as bounded cones.
	So if $\widetilde{R} \cap \partial\Delta(G)$ is empty, $\partial\Delta(G)$ is in the finite union of some cones $C(y,1/n)$ with $d(O,y)=n$   for some (great) $n$, and these cones do not intersect $\widetilde R$. Therefore $R=\widetilde{R}$ is in the ball $B(O,n)$. This contradicts with the unboundness of $R$. Hence $\widetilde{R}\cap \partial\Delta(G)\neq\varnothing$.  In other words, $\Upsilon$ fixes a point at infinity.   
	$\Upsilon$ is thus contained in $P(K)$ where $P\subset G$ is a proper parabolic subgroup of $G$. 
Write $H\subset G$ to be the Zariski closure of $\Upsilon$ in $G$. Then $H\subsetneq G$.   Since   $\Upsilon\triangleleft \Gamma$ and $\Gamma$ is Zariski dense in $G$, it follows that $H$ is a normal subgroup of $G$. 
%	Consider the map 
%	\begin{align*}
%		p: G\times H \to G\\
%		(g,h)\mapsto ghg^{-1}
%	\end{align*}
%	which is an algebraic morphism.  Since $g\Upsilon g^{-1}\subset \Upsilon$ for $g\in \Gamma$, it follows that $gHg^{-1}\subset H$. Hence $p(\Gamma\times H)\subset H$. As $\Gamma$ is Zariski dense in $G$, one thus has
% $
%	p(G\times H)\subset H, 
%	$ which shows $H\triangleleft G$. 
 	Since 
	$G$ is assumed to be almost simple, we conclude that $H$, hence $\Upsilon$ is finite. This proves the lemma.   
\end{proof} 
We mention that \cref{lem:BT} was also obtained in \cite{Bru22} by similar arguments. 
\iffalse
\begin{rem}
It is worth acknowledging that  \cref{lem:BT} presented here was inspired by a discussion between the second author and Brunebarbe at a non-abelian Hodge Theory conference held in Saint-Jacut, France in June 2022.  After the second author sent the proof of \cref{lem:BT} to Brunebarbe, he informed the second author that his proof is similar to ours, which was later published in \cite{Bru22}. We are grateful for this fruitful discussion and the contributions made by Brunebarbe to this area.
\end{rem}
\fi

\subsection{Harmonic mappings to Euclidean buildings} \label{sec:harmmapp}
The construction of 
\(s_\varrho : X \to S_\varrho\) 
in \cref{main3} is based on an existence theorem for 
\(\varrho\)-equivariant pluriharmonic maps into 
Bruhat-Tits buildings.   
\begin{thm}[\cite{BDDM,DM24}] \label{thm:BDDMex} 
	Let $X$ be a smooth quasi-projective variety and let $G$ be a reductive group defined over a non-archidemean local field $K$. Let $\Delta(G)$ be  the  Bruhat-Tits building of $G$. Denote by  $\pi_X:\widetilde{X}\to X$ the universal covering map.   If $\varrho:\pi_1(X)\to G(K)$ is a Zariski dense representation, then the following statements hold:
	\begin{thmenum}
		\item  \label{item:existence}  There exists a $\varrho$-equivariant pluriharmonic map $u:\widetilde{X}\to \Delta(G)$ with  logarithmic energy growth. 
		\item \label{item:harmonic}$ {u}$ is harmonic with respect to any K\"ahler metric on $\widetilde{X}$.  	
		\item  \label{item:pullback}Let $f: Y \to X$ be a morphism from a smooth quasi-projective variety $Y$. Denote by $\tilde{f}: \widetilde{Y} \to \widetilde{X}$   the lift of $f$ between the universal covers of $Y$ and $X$.  Then the $f^*\varrho$-equivariant map $ {u }\circ \tilde{f}: \widetilde{Y} \to \Delta(G)$ is   pluriharmonic and has logarithmic energy growth.  
	\end{thmenum} 
\end{thm} 
\begin{rem}
For the definitions appearing in the theorem, 
we refer the reader to \cite{BDDM} or to \cite[Section~2]{DM24} for more details. 
The above theorem was established in \cite[Theorem~A]{BDDM} by 
Brotbek, Daskalopoulos, Mese, and the second author in the case where $G$ is semisimple. 
Subsequently, building on \cite{BDDM}, 
the second author and Mese extended the result to the general reductive case 
in \cite[Theorem~A]{DM24}.

In the general case, the (enlarged) Bruhat--Tits building $\Delta(G)$ 
is the product of the Bruhat--Tits building $\Delta(\mathcal{D}G)$ of the derived group $\mathcal{D}G$  (that is semisimple)
with a real Euclidean space $V := \mathbb{R}^N$ on which $G(K)$ acts by translations 
(cf.~\cite{KP23}). 
The stabilizer of any point in $\Delta(G)$ is an open and bounded subgroup of $G(K)$. 
The pluriharmonic map $u$ of logarithmic energy constructed in~\cite{DM24} then decomposes as
\[
(u_1, u_2) \colon \widetilde{X} \longrightarrow \Delta(\mathcal{D}G) \times V,
\]
where $u_1$ is the $\varrho_1$-equivariant pluriharmonic map of logarithmic energy constructed in~\cite[Theorem~A]{BDDM}, where $\varrho_1$ is the composite of $\varrho$ with the natural projection $G \to \mathcal{D}(G)$.

In~\cite{DM24}, the component $u_2$ is defined by integrating suitable global logarithmic $1$-forms $\eta_1, \ldots, \eta_\ell$ on $X$ with purely imaginary residues along the boundary.
\end{rem}

\medspace

With this notation, recall that   \(x \in \widetilde{X}\) is called a {\em regular point}  of $u$  if it admits an open neighborhood \(U\) such that \(u(U) \subset A\) for some appartment \(A\). We say that a point \(x \in X\) is regular if some (equivalently, any) of its preimages in \(\widetilde{X}\) is regular. The regular points of $u$ in \(X\) form a non-empty open subset \(\cR(u)\); denote by \(\mathcal{S}(u)\) its complementary. By the deep theorem of Gromov-Schoen, $\cS(u)$ has small Hausdorff dimension. 
\begin{thm}[{\cite[Theorem 6.4]{GS92}}]\label{thm:GS}
	$\cS(u)$ is a closed subset of $X$ with  \emph{Hausdorff  dimension}    at most $2\dim_{\bC}X-2$. \qed
\end{thm} 
Let us also mention that, more recently, the second author and Mese proved in \cite{DM24} that $\cS(u)$ is contained in a proper Zariski closed subset of $X$, although we will not use this result in the present paper.

We consider a pluriharmonic map $u:M\to\Delta(G)$ from a K\"ahler manifold $(M,\omega)$.
We recall that the energy density $|\nabla u|^2$ is defined in \cite{KS}  
(see also \cite[Section~2]{BDDM}), and is an $L^1_{\rm loc}$-function on $M$.  
%Since $u$ is $\varrho$-equivariant, $|\nabla u|^2$ descends to an $L^1_{\rm loc}$-function on $X$.  
%It is shown in \cite[Lemma 4.27]{DMW24} that this function is moreover continuous, though we will not use this fact here.
Since the target is a  Euclidean building, the energy density admits a more explicit expression than in the general case of harmonic maps into NPC spaces.
Indeed, for any regular point $x\in M$, by definition there exists an open neighborhood $\Omega$ of $x$ and an apartment $A$ such that $u(\Omega)\subset A$.  
%Indeed, for any regular point $x\in \pi^{-1}(\cR(u))\subset \widetilde{X}$, by definition there exists an open neighborhood $\Omega$ of $x$ and an apartment $A$ such that $u(\Omega)\subset A$.  
Fix an isometry $i:A\to\bR^N$, and write
\[
i\circ u|_{\Omega}=(u_1,\ldots,u_N).
\]
Then the energy density of $u|_{\Omega}$ coincides with the usual energy density of harmonic maps into Riemannian manifolds. In particular, for any $x\in \Omega$, we have
\begin{equation}\label{eq:density}
	|\nabla u(x)|^2=\sum_{i=1}^N |du_i(x)|_\omega^2
	= 2\sum_{i=1}^N |\partial u_i(x)|_\omega^2,
\end{equation}
where the second equality follows from the fact that $u$ is pluriharmonic, so $\db\d u_i\equiv 0$ for each $i$.
In particular, the energy density $|\nabla u|^2$ is continuous on $\Omega$, hence on $\cR(u)$.
We remark that it is shown in \cite[Lemma 4.27]{DMW24} that the energy density  extends to a continuous function on $M$, though we will not use this fact here.

\begin{lem}\label{lem:density}
For a pluriharmonic map $u:M\to\Delta(G)$ from a K\"ahler manifold $(M,\omega)$, the following two assertions are equivalent.
\begin{enumerate}[label*=(\arabic*)]
\item
$|\nabla u|^2=0$ almost everywhere on $M$.
\item
$u$ is a constant map.
\end{enumerate}
\end{lem}
\begin{proof}
Let $V\subset M$ be the set of all regular points of $u$.
Then by \cite[Theorem 6.4]{GS92}, $M\setminus V$ is a closed subset of $M$ with Hausdorff  dimension  at most $2\dim_{\bC}M-2$. 
Thus $V$ is connected.
By \eqref{eq:density}, $|\nabla u|^2$ is continuous on $V$.
Thus, the condition that $|\nabla u|^2$ vanishes almost everywhere is equivalent to $|\nabla u|^2$ vanishing everywhere on $V$.
This means that $u$ is locally constant on $V$, hence constant on $V$ as $V$ is connected. Because $u$ is locally Lipschitz, the latter condition is equivalent to $u$ being constant everywhere on $M$. 
\end{proof}

We prove a result on the continuity of energy of harmonic mappings with respect to boundary values.  We first make a convention.

Let $P,Q\in \Delta(G)$ and let $\gamma:[0,1]\to \Delta(G)$ be the  unique
 geodesic from $P$ to $Q$. Let  $$(1-s)P+s Q:=\gamma(s)$$ denote the
point which is a fraction $s$ of the way from $P$ to $Q$ along $\gamma$. 
\begin{lem}\label{lem:energycontinuity}
Let $\mathbb{B}_r\subset \bR^n$ be the Euclidean ball of radius $r$ centered at the origin, endowed with the Euclidean metric, and write $\mathbb{B}:=\mathbb{B}_1$.   
Let $\Sigma=\partial\mathbb{B}$, equipped with the standard round metric.  Assume that $n\ge 2$.  
Consider a family of harmonic maps
\[
u_t : \mathbb{B}_{1+\delta} \longrightarrow \Delta(G),
\]
parametrized by $t\in(-\ep_0,\ep_0)$, where $\delta>0$ is a fixed constant.   Assume that there exists a constant $C>0$ such that, for all $t\in(-\ep_0,\ep_0)$, 
the restriction $u_t|_{\mathbb{B}}$ is $C$-Lipschitz, that is,
\begin{align}\label{eq:lipschitz}
	 d\bigl(u_t(x),u_t(y)\bigr) \le C\, d_{\mathbb{B}}(x,y)
	 \quad\text{for all }x,y\in\mathbb{B}.
\end{align} 
Denote by $E(t):=E^{u_t}[\mathbb{B}]$ the energy of $u_t$ on $\mathbb{B}$, and by $E_\Sigma(t)$ the energy of the restriction $u_t|_{\Sigma}$. 
 Assume that %there exists $K\geq 1$ such that 
%\[
%E(t),\; E_\Sigma(t) \leq K \qquad\text{for all } t,
%\]
  that the function 
\[
t \mapsto \int_{\Sigma} d^{\,2}\!\bigl(u_t, u_0\bigr)\, d\sigma
\]
is continuous at $t=0$, 
where $d\sigma$  is the round measure.
Then the function 
\[
t \mapsto E(t)
\]
is continuous at $t=0$.  
Here $d(-,-)$ denotes the distance function on $\Delta(G)$. 
\end{lem}

\begin{proof}
 By the uniform Lipschitz continuity of \(u_t\) in \eqref{eq:lipschitz}, together with \cite[Remark~1.5]{BDDM}, 
the energy density of \(u_t\)  and $u_t|_{\Sigma}$ satisfies
\begin{align*}
	 |\nabla u_t(x)|^2 \le C   \qquad \text{for all } x\in \mathbb{B} \text{ and } t\in (-\varepsilon_0,\varepsilon_0)\\
	  |\nabla (u_t|_{\Sigma})(x)|^2 \le C  \qquad \text{for all } x\in \Sigma \text{ and } t\in (-\varepsilon_0,\varepsilon_0).
\end{align*} 
Consequently, there exists a constant \(K>0\) such that 
\begin{align}\label{eq:Kbound}
E(t) \leq K,\quad E_\Sigma(t)\leq K, \qquad \text{for all } t\in (-\varepsilon_0,\varepsilon_0).
\end{align} 
For any  $\varrho\in (0,1)$,  we
define diffeomorphisms
\begin{align*}
	\Phi&:\mathbb B_{1-\rho} \rightarrow \mathbb B, \ \ x \mapsto \frac{x}{1-\rho}\\
	\Psi_1&:\mathbb B \backslash  \mathbb B_{1-\rho} \rightarrow \Sigma \times [0,\rho], \ \ x \mapsto (\frac{x}{|x|}, 1-|x| )\\ 
	\Psi_2&:\mathbb B \backslash  \mathbb B_{1-\rho} \rightarrow \Sigma \times [0,\rho], \ \ x \mapsto (\frac{x}{|x|},  |x|-1+\varrho )
\end{align*}

 For any Lipschitz  map $f:\bB \rightarrow \Delta(G)$, note  that the energy of $f\circ \Phi$ on $\bB$ is 
\begin{align}\label{eq:dec}
	E^{f  \circ \Phi}[ \bB_{1-\varrho}] =
\int_{\bB_{1-\varrho}}	|\nabla(f\circ \Phi)|^2 d\mu 
	=\int_{\bB_{1-\varrho}} (1-\rho)^{n-2} \Phi^*(|\nabla f|^2 d\mu )\leq E^{f}[ \bB].
 \end{align}
	Here $d\mu$ is the Euclidean volume form. 
%Fix any $\ep \in (0,1)$. %By the above equality, we  can fix some $\rho>0$ sufficiently small such that  \begin{align}\label{eq:apro}
%	K\rho<\frac{\ep}{8}
%\end{align} 	and 	 
% \begin{align}\label{eq:apro2}
% 	E^{f  \circ \Phi}[ \bB_{1-\varrho}] \leq \left(1+\frac{\ep}{4K}	\right)E^f[ \bB].
% \end{align}  

Let 
$
f : \Sigma \times [0,\varrho] \longrightarrow \Delta(G)
$
be a Lipschitz map. We endow \(\Sigma \times [0,\varrho]\) with the product metric of the round metric on \(\Sigma\) and the standard Euclidean metric on the interval.
  We have 
\begin{align}\label{eq:dec2}
E^{f\circ \Psi_i}[\bB\backslash \bB_{1-\varrho}]&:=\int_{\bB\backslash \bB_{1-\varrho}}|\nabla (f\circ \Psi_i)|^2d\mu \\\nonumber
 &\leq \begin{cases}
 	 \int_{\Sigma\times[0,\varrho]}|\nabla f|^2 d\sigma dr=:E^f[{\Sigma\times[0,\varrho]}],\quad \mbox{if } n\geq 3\\
 		(1-\varrho)^{-1} \int_{\Sigma\times[0,\varrho]}|\nabla f|^2 d\sigma dr= 	(1-\varrho)^{-1} E^f[{\Sigma\times[0,\varrho]}],\quad \mbox{if } n=2 
 \end{cases} 
 \end{align} 
 Therefore, by \eqref{eq:Kbound} and \eqref{eq:dec2},  for any $\ep\in (0,1)$,  we  can   fix some sufficiently small  $\varrho\in (0,1)$ such that  \begin{align}\label{eq:apro}
 	K\rho <\frac{\ep}{4}, \quad
 	E^{f\circ \Psi_i}[\bB\backslash \bB_{1-\varrho}] \leq 2 E^f[{\Sigma\times[0,\varrho]}].
 \end{align}

 Since  the function 
\[
t \longmapsto \int_{\Sigma} d^{\,2}\!\bigl(u_t, u_0\bigr)\, d\sigma
\]
is continuous at $t=0$, it follows that    
\begin{align}\label{eq:conti}
 \frac{1}{\rho} \int_\Sigma d^2(u_t,u_0)d\sigma<\frac{\ep}{4}
\end{align}  for $|t|$ sufficiently small.
	
	Define $W_t:\Sigma \times [0,\rho] \rightarrow \Delta(G)$ to be 
	\[
	W_t(x,s)=(1-\frac{s}{\rho}) u_t(x)+\frac{s}{\rho} u_0(x). 
	\] 
	By \cite[Lemma 3.12]{KS2}, for $|t|$ sufficiently small, the energy of $W_t$    is bounded by 
\begin{align}\label{eq:dec3}
	E^{W_t}[{\Sigma\times[0,\varrho]}]\leq \frac{\rho}{2} (E_\Sigma(0)+E_{\Sigma}(t))+ \frac{1}{\rho} \int_\Sigma d^2(u_t,u_0) \, d\sigma < \frac{\ep}{2}. 
\end{align}
	Here we apply   \eqref{eq:apro} and \eqref{eq:conti}.  
	Define $w_t:\mathbb B \rightarrow \Delta(G)$ by
	\[
	w_t(x)=
	\left\{
	\begin{array}{ll}
		u_0(\frac{x}{1-\rho})& \mbox{for } x \in \mathbb B_{1-\rho}\\
		W_t(\frac{x}{|x|},1- |x| )& \mbox{for }x \in \mathbb B \backslash \mathbb B_{1-\rho}.
	\end{array}
	\right.
	\]
	Since $u_t$ is energy minimizing and $w_t|_{\d \bB}=u_t|_{\d \bB}$, for $|t|$ sufficiently small, we have 
	\begin{equation*}  
		E(t) \leq E^{w_t}[\bB] =E^{u_0 \circ \Phi}[ \bB_{1-\varrho}] +E^{W_t \circ \Psi_1}[\bB\backslash\bB_{1-\varrho}] \leq  E(0) +2E^{W_t}[{\Sigma\times[0,\varrho]}]  < E(0) +\ep.
	\end{equation*} Here we apply \eqref{eq:dec}  and \eqref{eq:dec3}.

Similarly, we define another function
 $v_t:\mathbb B \rightarrow \Delta(G)$ by
\[
v_t(x)=
\left\{
\begin{array}{ll}
	u_t(\frac{x}{1-\rho})& \mbox{for } x \in \mathbb B_{1-\rho}\\
	W_t(\frac{x}{|x|},|x|-1+\varrho )& \mbox{for }x \in \mathbb B \backslash \mathbb B_{1-\rho}.
\end{array}
\right.
\]
	Since $u_0$ is energy minimizing and $v_t|_{\d \bB}=u_0|_{\d \bB}$, for $|t|$ sufficiently small, we have 
\begin{equation*}  
	E(0) \leq E^{v_t}[\bB] =E^{u_t \circ \Phi}[ \bB_{1-\varrho}]  +E^{W_t \circ \Psi_2}[\bB\backslash\bB_{1-\varrho}] \leq  E(t) +2E^{W_t}[{\Sigma\times[0,\varrho]}]   < E(t) +\ep.
\end{equation*}
	Here we apply\eqref{eq:dec} and \eqref{eq:dec3} once again.    The lemma is thus proved. 
\end{proof}

\begin{cor}\label{cor:20251127}
Let $v_n:\mathbb D\to\Delta(G)$ be a sequence of  $C$-Lipschitz harmonic maps which converges uniformly on compact subsets on $\mathbb D$ to a $C$-Lipschitz harmonic map $v:\mathbb D\to\Delta(G)$.
Then $v$ is constant if and only if for every $\delta\in (0,1)$, $E^{v_n}(\mathbb D_{\delta})\to 0$ as $n\to\infty$.
\end{cor}

\begin{proof}
Since $v_n$ converges to $v$ uniformly on $\partial\mathbb D_{\delta}$, we have 
$$\int_{\partial\mathbb D_{\delta}} d^{\,2}\!\bigl(v_n, v\bigr)\, d\sigma\to0$$
as $n\to\infty$.
Hence, by applying \cref{lem:energycontinuity}, where $u_t=v_n$ for $\frac{1}{n+1}\leq |t|<\frac{1}{n}$ and $u_0=v$, we obtain $E^{v_n}(\mathbb D_{\delta})\to E^{v}(\mathbb D_{\delta})$. 
Thus, it suffices to show that $v$ is constant if and only if $E^{v}(\mathbb D_{\delta})=0$ for all $\delta\in (0,1)$.
The latter condition is equivalent to the almost everywhere vanishing of the energy density $|\nabla v|^2=0$ on $\mathbb D$.
Thus our corollary follows from \cref{lem:density}.
	\end{proof}

\subsection{Proof of \Cref{main3}} \label{sec:KZthm}
\begin{thm}[=\Cref{main3}] \label{thm:KZreduction}
	Let $X$ be a complex quasi-projective normal variety,  and let $G$ be a reductive algebraic group defined over a non-archimedean local field $K$. Let $\varrho:\pi_1(X)\to G(K)$ be a Zariski-dense representation. Then there exist  a quasi-projective normal variety $S_\varrho$ and a dominant morphism $s_\varrho:X\to S_\varrho$ with connected general fibers, such that for any connected Zariski closed  subset $T$ of $X$, the following properties are equivalent:
	\begin{enumerate}[label*={\rm (\alph*)}]
		\item  the image $\rho({\rm Im}[\pi_1(T)\to \pi_1(X)])$ is a bounded subgroup of $G(K)$.
		\item  For every irreducible component $T_o$ of $T$, the image $\rho({\rm Im}[\pi_1(T_o^{\rm norm})\to \pi_1(X)])$ is a bounded subgroup of $G(K)$.
		\item The image $s_\varrho(T)$ is a point.
	\end{enumerate} 
\end{thm}
We divide its proof into 5 steps. Throughout Steps 1-4, we assume that $X$ is smooth, and in Step 5, we address the case when $X$ is singular.    Steps 4-5 are independent of the other part of the paper and the readers may skip them.    It is worth highlighting that \cref{thm:KZreduction} will play a crucial role as a cornerstone in the   paper  \cite{DY23} by the second and third authors on the Shafarevich conjecture for reductive representations over quasi-projective varieties (see also \cite{Bru23} for a similar result).

\begin{proof}[Proof of \cref{thm:KZreduction}]
In this proof, by a \emph{multiset} we mean a pair $(S, m)$, where $S$ is the underlying set and 
$m : S \to \mathbb{Z}_{> 0}$ is a function assigning to each element $s \in S$ its multiplicity $m(s)$, 
that is, the number of times $s$ occurs in the multiset.  
We often write a multiset in the form $\{s_1, s_2, \ldots, s_k\}_{m}$, or simply $\{s_1, s_2, \ldots, s_k\}$, where each element $s_i$ appears $m(s_i)$ times in this notation.
The union of two multisets  $(S,m)$ and $(S',m')$ is defined as the multiset $(S\cup S', m+m')$.
We say that a group $G$ acts on a multiset $(S,m)$ if $G$ acts on $S$ and $m$ is invariant under this action.

\medspace

	\noindent
	{\em Step 1. Constructing spectral covering and spectral one forms}. 
Let $\overline{X}$ be a smooth projective compactification of $X$ such that $D:=\overline{X}\backslash X$ is a simple normal crossing divisor.

	For a reductive algebraic group $G$ over a non-archimedean local field $K$,  
it induces   a real Euclidean space $V$ endowed with a Euclidean metric and  an affine Weyl group $W$ acting on $V$ isometrically.  Such group $W$ is a semidirect product   $T\rtimes W^v$,  where $W^v$ is the \emph{vectorial Weyl group}, which is a finite group generated by reflections on $V$, and $T$ is a translation group of $V$ (see \cite{Guy23} for more details.).

	For any apartment $A$ in $ \Delta(G)$,  there exists an isomorphism $i_A:A\to V$, which is  called a chart. For two charts $i_{A_1}:A_1\to V$ and $i_{A_2}:A_2\to V$, if $A_1\cap A_2\neq\varnothing$,   it satisfies the following properties:
	\begin{enumerate}[label=(\alph*)]
		\item $Y:=i_{A_2}(i_{A_1}^{-1}(V))$ is convex.
		\item \label{item:weyl}There is an element $w\in W$ such that $w\circ i_{A_1}|_{A_1\cap A_2}=i_{A_2}|_{A_1\cap A_2}$.
	\end{enumerate} 
	Let us fix  orthonormal coordinates $(x_1,\ldots,x_N)$ 
	for $V$.  Since $W^v\subset {\rm GL}(V)$ acts on $V$ isometrically,  for any $w\in W^v$, $(w^*x_1,\ldots,w^*x_N)$ are orthonormal coordinates  for $V$. We define a multiset $\Phi$ by 
	\begin{align}\label{eq:Phi}
		\Phi:=\bigcup_{w\in W^v}\{w^*x_i \}_{i\in  \{1,\ldots,N\}},
	\end{align}
	where the union is taken as multiset.
	Since $W^v$ is a finite group, this is a finite union, and consequently  the underlying set of $\Phi$ is a finite subset of $V^*$.
	Note that  $W^v$ acts on $\Phi$.  We   write 
	$ \Phi=\{\beta_1,\ldots,\beta_m\}$ as a multiset. 
	
%	We define  real affine functions 
%	\begin{align}\label{eq:beta}
	%	\beta_{A,i}:=\beta_i\circ i_{A}(x) 
	%\end{align}
%	on $A$ for each $i$.  
	%\begin{lem}\label{lem:indepedent}
%		If $A_1\cap A_2\neq\varnothing$, then we have $$\{d\beta_{A_1,1},\ldots,d\beta_{A_1,m}\}|_{A_1\cap A_2}=\{d\beta_{A_2,1},\ldots,d\beta_{A_2,m}\}|_{A_1\cap A_2}$$
	%	as multisets.
%	\end{lem}
%	\begin{proof}
	%	By \Cref{item:weyl}, there exists an element $w\in W$ such that  
	%	$
	%	\beta_k\circ i_{A_2}|_{A_1\cap A_2}=  \beta_k\circ w\circ i_{A_1}|_{A_1\cap A_2} 
	%	$ 
	%	for any $k=1,\ldots,m$. 
	%	Recall that $W^v$ permutes $\Phi$. It follows that there exist $a_1,\ldots,a_m\in \bR$ 
	%	and  a permutation $\sigma$ of $m$-elements   such that
	%	\begin{align}\label{eq:compatible} 
	%		\beta_k\circ i_{A_2}|_{A_1\cap A_2}=  \beta_k\circ w\circ i_{A_1}|_{A_1\cap A_2} = \beta_{\sigma(k)}\circ i_{A_1}|_{A_1\cap A_2} -a_k
	%	\end{align} 
	%	for any $k=1,\ldots,m$.   This implies the lemma. 
%	\end{proof}
By \cref{thm:BDDMex}, there exists a $\varrho$-equivariant pluriharmonic map $ {u}:\widetilde{X}\to \Delta(G)$ with  logarithmic energy growth. 	For any regular point $x\in \cR(u)$ of $u$,   one can choose a   simply-connected open neighborhood  $U_x$ of $x$   such that
	\begin{enumerate}[label=(\arabic*)]
		\item \label{item:connect} the inverse image $\pi_X^{-1}(U_x)=\coprod_{\alpha\in I}U_\alpha$ is a union of disjoint open sets in $\widetilde{X}$,  each of which is mapped isomorphically onto 
		$U$ by $\pi_X:\widetilde{X}\to X$.
		\item For some $\alpha\in I$,  there is an apartment $A_\alpha$ of $\Delta(G)$ such that $u({U}_\alpha)\subset A_\alpha$.
	\end{enumerate}

 Let $\Phi = \{\beta_1, \ldots, \beta_m\}$ be the multiset defined in \eqref{eq:Phi}.  
% For each apartment $A$ of $\Delta(G)$, let $\{\beta_{A,1}, \ldots, \beta_{A,m}\}$ be the affine functions on $A$ defined in \eqref{eq:beta}.  
 For each $j \in \{1, \ldots, m\}$, we define a real function
 \begin{align}\label{eq:u}
 	u_{\alpha, j} = \beta_j\circ i_{A_\alpha} \circ  {u} \circ (\pi_X|_{U_\alpha})^{-1} : U_x \to \bR.
 \end{align} 
 By the pluriharmonicity of $u$, we have
 \[
 \db\d u_{\alpha, j} = 0
 \]
 for each $j\in \{1, \ldots, m\}$.  
 Hence, $\partial u_{\alpha, j}$ is a holomorphic $1$-form on $U_x$.   
 \begin{claim} 
 The multiset of holomorphic $1$-forms
 \begin{align*}
 	\eta_{U_x} := \{\partial u_{\alpha,1}, \ldots, \partial u_{\alpha,m}\}
 \end{align*} 
 over $U_x$ does not depend on the choice of $U_\alpha$ or of the apartment $A_\alpha$ containing $u(U_\alpha)$.  Consequently,  for any $x,y\in \cR(u)$, if $U_x\cap U_y\neq \varnothing$, then   we have
 \[
 \eta_{U_x} = \eta_{U_y} \quad \text{on } U_x \cap U_y
 \]
 as multisets.
 \end{claim}
\begin{proof} For any $\alpha \in I$, we denote by
	$$
u_\alpha:=i_{A_\alpha}\circ {u} \circ (\pi_X|_{U_\alpha})^{-1} : U_x \to V.
	$$ Pick any $\alpha,\beta\in I$. Then there exists  $\gamma\in \pi_1(X)$ such that $\gamma$ maps $U_\alpha$ to $U_\beta$ isomorphically.  Since $ {u}$ is $\varrho$-equivariant,  one has $ \varrho(\gamma) {u}\circ (\pi_X|_{U_\alpha})^{-1}= {u}\circ (\pi_X|_{U_\beta})^{-1}$,  and thus
	\begin{align*}
		u_\beta= i_{A_\beta}\varrho(\gamma) i_{A_\alpha}^{-1}u_\alpha. 
	\end{align*} 
It implies that     $$u_\beta(U_x)\subset V\cap i_{A_\beta}\varrho(\gamma) i_{A_\alpha}^{-1}(V)\neq\varnothing.$$  By \cite[Corollary 4.2.25]{KP23} and \cite[Axiom 4.1.4 (A 1)]{KP23}, there exists $w\in  W$  such that $wx=i_{A_\beta}\varrho(\gamma) i_{A_\alpha}^{-1}(x)$ for any $x\in V\cap i_{A_\beta}\varrho(\gamma) i_{A_\alpha}^{-1}(V)$.   This implies that    $u_\beta=wu_\alpha$.  Since $\Phi$ is invariant under the  action by $W^v$, by \eqref{eq:u},    we conclude that 
	\begin{align*}
		\{\partial u_{\alpha,1},\ldots,\partial u_{\alpha,m}\} = \{\partial u_{\beta,1},\ldots,\partial u_{\beta,m}\}.
	\end{align*} 
This shows that the construction of multi-valued 1-form is canonical.The claim is proved. 
\end{proof}
 
 Consequently, if we define $\sigma_{i,\eta_{U_x}}\in  \Gamma(U_x, \Sym^i \Omega_{U_x})$ by
 $$
T^m + \sigma_1 T^{m-1} + \cdots + \sigma_m=(T-\partial u_{\alpha,1})\cdots(T-\partial u_{\alpha,m}),
 $$
 where $T$ is a formal variable,
 then we have $\sigma_{i,\eta_{U_x}}=\sigma_{i,\eta_{U_y}}$ on $U_x\cap U_y$.
 Hence these local sections glue to define $\sigma_i \in \Gamma(\cR(u), \Sym^i \Omega_{\cR(u)})$.  
 We refer the reader to \cite[Lemma~4.3]{BDDM} for further details.
As $u$ is locally Lipschitz and has logarithmic energy growth,   by the same arguments as in \cite[Proposition~4.4]{BDDM}, each $\sigma_i$ extends to a logarithmic symmetric form in $H^0(\overline{X}, \Sym^i\Omega_{\overline{X}}(\log D))$. Indeed, the locally Lipschitz property of $u$, together with the fact that $\cS(u)$ has Hausdorff codimension at least two, implies that each $\sigma_i$ extends across $\cS(u)$, and the logarithmic energy growth of $u$ shows that $\sigma_i$ has at most logarithmic poles along $D$. 
 Hence, the assumptions of \cref{lem:spectral} are satisfied with $\mathcal{E} = \Omega_{\overline{X}}(\log D)$ and $P(T)=T^m + \sigma_1 T^{m-1} + \cdots + \sigma_m$.
Let 
 \[
 \pi : \overline{\xsp} \to \overline{X}
 \]
 be the covering $\Sigma\to \overline{X}$ defined in \cref{lem:spectral}. We then have 
 \[
 \eta_1, \ldots, \eta_m \in H^0(\overline{\xsp}, \pi^*\Omega_{\overline{X}}(\log D)) 
 \]
 such that $\pi^*P(T)=(T-\eta_1)\cdots(T-\eta_m)$.
	 By \cref{item:extension,item:ramified}, we have following two assertions.
	\begin{claim}\label{claim:extension}
		Any $g\in  \widetilde{G}:={\rm Aut}(\overline{\xsp}/\overline{X})$ acts on the (underlying) set $\{\eta_1,\ldots,\eta_m\}$ by a permutation.     
	\end{claim}   
		\begin{claim}\label{claim:ramified}
		The Galois morphism $\pi:\overline{\xsp}\to \overline{X}$ is \'etale outside $$R:=\{z\in \overline{\xsp} \mid \exists \eta_i\neq\eta_j \ \mbox{with}\ (\eta_i-\eta_j)(z)=0 \},$$
		which satisfies $\pi^{-1}(\pi(R))=R$.  
	\end{claim}

\begin{dfn}[Spectral covering and one forms]\label{def:spectral}
	 The   Galois morphism $\pi:\xsp\to X$  is called the \emph{spectral covering} of $X$ associated to $\varrho$.
	 The multiset  $\{\eta_1, \ldots, \eta_m\}$ of sections in $H^0(\overline{\xsp}, \pi^*\Omega_{\overline{X}}(\log D))$ is called the \emph{spectral one-forms} induced by $\varrho$.
\end{dfn}

	\medspace

\noindent 	{\em Step 2. Pluriharmonic map and  spectral one forms.}  
	Let $Y\to \xsp$ be a resolution of singularities. 
	For each $i=1,\ldots,m$, we define $\tau_i\in H^0(Y, \Omega_{Y})$ by the pull-back of $\eta_i \in H^0(\overline{\xsp}, \pi^*\Omega_{\overline{X}}(\log D))$ to $Y$. 
	Furthermore,  we set $p:Y\to X$ to be the composition $Y\to \xsp \to X$.

	\begin{claim}\label{claim:202511272}
	Let $f: Z \to Y$ be a morphism from a smooth quasi-projective variety $Z$.
	The following three assertions are equivalent:
	\begin{enumerate}[label*=(\arabic*)]
	\item \label{item:20251201}
	For all $\tau_1, \ldots, \tau_m\in H^0(Y, \Omega_{Y})$, we have $f^*\tau_1=\cdots=f^*\tau_m=0$ in $H^0(Z,\Omega_{Z})$.
	\item \label{item:202512011}
	The pluriharmonic map $u\circ \widetilde{p\circ f}:\widetilde{Z}\to\Delta(G)$ is constant, where $\widetilde{p\circ f}: \widetilde{Z} \to \widetilde{X}$ is a lift of $p\circ f:Z\to X$ between the universal covers of $Z$ and $X$.  
	\item \label{item:202512012}
	The image $p^*\rho({\rm Im}[\pi_1(Z)\to \pi_1(Y)])$ is a bounded subgroup of $G(K)$.
	\end{enumerate}
	\end{claim}

	\begin{proof}
	In this proof, we set $u_f=u\circ \widetilde{p\circ f}$ for short.
	We first show the equivalence of \ref{item:202512011} and \ref{item:202512012}.
	Note that the $(p\circ f)^*\varrho$-equivariant map $u_f: \widetilde{Z} \to \Delta(G)$ is pluriharmonic and has logarithmic energy growth (cf. \cref{thm:BDDMex}).
	Suppose that $u_f$ is constant.
	Let $\{P\}$ be the image of $u_f$.  
	Then $p^*\varrho({\rm Im}[\pi_1(Z)\to \pi_1(Y)])$ fixes $P$ as $u_f$ is $(p\circ f)^*\varrho$-equivariant.  
	By \cref{lem:AB}, the stabilizer of $P$ is a bounded subgroup of $G(K)$, hence $(p\circ f)^*\varrho(\pi_1(Z))$ is bounded.
	Next suppose $p^*\rho({\rm Im}[\pi_1(Z)\to \pi_1(Y)])$ is bounded.
	Let $P$ be the fixed point of $p^*\varrho(({\rm Im}[\pi_1(Z)\to \pi_1(Y)]))$, and define the constant map  
	$u':\widetilde{Z}\to\{P\}$.  
	Then $u'$ is $(p\circ f)^*\varrho$-equivariant and harmonic.  
	By the unicity result in \cite[Theorem B]{DM24}, we have
	\[
	 |\nabla_{u_f}|^2=|\nabla_{u'}|^2 =0
	\]
	almost everywhere on $\tilde{Z}$. 
	Thus by \cref{lem:density},  $u_f$ is constant.
	We have proved the equivalence of \ref{item:202512011} and \ref{item:202512012}.

	Next we prove the equivalence of \ref{item:20251201} and \ref{item:202512011}.
	For this purpose, we introduce the following two assertions for a holomorphic map $g:\mathbb D\to Y$:
	\begin{enumerate}[label*=(\roman*)]
	\item\label{item:202512015}
	 For all $\tau_1, \ldots, \tau_m\in H^0(Y, \Omega_{Y})$, we have $g^*\tau_1=\cdots=g^*\tau_m=0$ in $H^0(\mathbb D,\Omega_{\mathbb D})$.
	\item \label{item:202512016}
	$u_g=u\circ \widetilde{p\circ g}:\mathbb D\to\Delta(G)$ is constant, where $\widetilde{p\circ g}: \mathbb D\to \widetilde{X}$ is a lift of $p\circ g:\mathbb D\to X$ to the universal cover of $X$.  
	\end{enumerate}
Condition \ref{item:20251201} is equivalent to the requirement that \ref{item:202512015} holds for all $g:\mathbb D\to Y$ that factor through $Z\to Y$.
Similarly \ref{item:202512011} is equivalent to the requirement that \ref{item:202512016} holds for all $g:\mathbb D\to Y$ that factor through $Z\to Y$.
Thus, it suffices to prove the equivalence of \ref{item:202512015} and \ref{item:202512016} for each holomorphic map $g:\mathbb D\to Y$. 
	Furthermore, to establish the equivalence, it suffices to show that it holds locally on $\mathbb D$ where the map $g:\mathbb D\to Y$ is an immersion.
	Therefore it is enough to show the equivalence under the assumption that we can extend $g:\mathbb D\to Y$ to a map $G:\mathbb D\times \mathbb D^k\hookrightarrow Y$ such that the restriction $G|_{\mathbb D\times \{0\}}$ equals $g$, and $G$ is a biholomorphic map onto an open subset of $Y$.
	We denote by $u_G:\mathbb D\times \mathbb D^k\to \Delta(G)$ the composite of $\widetilde{p\circ G}:\mathbb D\times \mathbb D^k\to\tilde{X}$ and $u:\tilde{X}\to \Delta(G)$.
	Then $u_G:\mathbb D\times \mathbb D^k\to \Delta(G)$ is locally Lipschitz and pluriharmonic.

	Set  $S:=(p\circ G)^{-1}(\cS(u))$.
Let $V$ be the set of all $w\in \mathbb D^k$ such that $(\mathbb D\times\{w\})\cap S$ has measure zero in $\mathbb D$.
Since $\cS(u)$ has Hausdorff codimension at least two (indeed, it is contained in a proper Zariski closed subset of $X$ by \cite{DM24}), Fubini's theorem implies that $\mathbb D^k\setminus V$ has measure zero.
In particular, $V$ is dense in $\mathbb D^k$.
 Denote by $\iota_w:\mathbb D\hookrightarrow \mathbb D\times\mathbb D^k$ the embedding
\[
\mathbb D  \longrightarrow \mathbb D \times\{w\} \longrightarrow \mathbb D\times\mathbb D^k.
\]
 We fix a K\"ahler  metric $\omega:=\sn dz\wedge d\bar{z}$ on $\mathbb D$. 
Since $u_G$ is pluriharmonic, it follows that for \emph{each} $w\in \mathbb D^k$, the map
\[
u_{G}\circ \iota_w : \mathbb D \longrightarrow \Delta(G)
\]
is harmonic.  
By construction, for each $w\in V$, the holomorphic $1$-forms
\[
\iota_w^*G^*\tau_1,\ldots,  \iota_w^*G^*\tau_m\in \Gamma(\bD, \Omega^1_{\bD})
\]
correspond to the complex differentials of $u_{G}\circ \iota_w$ on the non-empty open dense set $\mathbb D\setminus \iota_w^{-1}(S)$.
Thus by \eqref{eq:density} and \eqref{eq:Phi}, 
the energy density of $u_{G}\circ \iota_w$ for $w\in V$ satisfies (cf. \cite[eq. (3.7)]{DM24})
\begin{align*}
	|\nabla_{u_{G}\circ \iota_w}|^2  
	= \frac{2}{|W^v|} \sum_{i=1}^{m} \left|\iota_w^*G^*\tau_i\right|^2_{\omega} \quad \mbox{ on } \mathbb D\setminus  \iota_w^{-1}(S).
\end{align*}
Hence,  for every $\delta\in (0,1)$, we have 
\begin{align*}
	E^{u_{G}\circ \iota_w}[\bD_{\delta}]
	=
	\frac{1}{|W^v|} \int_{\bD_{\delta}}
	\iota_w^*G^*
	\left(\sn \sum_{i=1}^{m}\tau_i\wedge \bar{\tau}_i\right)
	 \quad \mbox{ for each } w\in V.
\end{align*} 
 Since $u_G$ is locally Lipschitz continuous,  we have $u_{G}\circ \iota_w\to u_{G}\circ \iota_0$ uniformly on compact subsets on $\mathbb D$ as $w\to 0$.
Hence by \cref{cor:20251127}, $u_{G}\circ \iota_0$ is constant if and only if $E^{u_{G}\circ \iota_w}[\bD_{\delta}]\to 0$ as $w\to 0$ along $w\in V$.

Note that, 
\begin{align*} 
	w\mapsto 
	 \int_{\bD_{\delta}}
	 \iota_w^*G^*
	 \left(\sn \sum_{i=1}^{m}\tau_i\wedge \bar{\tau}_i\right)
\end{align*} 
is a continuous function on the whole $w\in\bD^k$.  
Therefore $E^{u_{G}\circ \iota_w}[\bD_{\delta}]\to 0$ as $w\to 0$ along $w\in V$ is equivalent to 
 $$
  \int_{\bD_{\delta}}
 \iota_0^*G^*
 \left(\sn \sum_{i=1}^{m}\tau_i\wedge \bar{\tau}_i\right)=0.
 $$
 Noting that $u_g=u_{G}\circ \iota_0$, this implies that 
 \[
\text{$u_g$ is constant}
\quad \Longleftrightarrow \quad
\iota_0^*G^*\tau_i = 0 \text{ for all } i \quad\Longleftrightarrow\quad g^*\tau_i = 0 \text{ for all } i.
\]
The equivalence of \ref{item:202512015} and \ref{item:202512016} is proved.
\end{proof}

	\medspace

\noindent 	{\em Step 3. Partial quasi-Albanese morphism via spectral one forms.}  
Let $\mu: \overline{\xsp}'\to \overline{\xsp}$   be a	$\widetilde{G}$-equivariant resolution of singularities such that $E:=({\pi}\circ \mu)^{-1}(D)$  is a simple normal crossing divisor.
	Write $Y:=({\pi}\circ \mu)^{-1}(X)$, which is a smooth quasi-projective variety. 
	We consider the pull-back of the spectral one forms $\{\mu^*\eta_1,\ldots,\mu^*\eta_m\}$ as one forms in $H^0(\overline{\xsp}', \Omega_{\overline{\xsp}'}(\log E))$.
	Define $A\subset \cA_Y$ to be   the largest semi-abelian variety contained in the quasi-Albanese variety $\cA_Y$ of $Y$ on which each $\eta_i$ vanishes.  
	We denote by $\cA:=\cA_Y/A$ the quotient, which is also a semi-abelian variety. 
	By \cref{claim:extension}, the natural $\widetilde{G}$-action on $\cA_Y$ induced by the $\widetilde{G}$-action on $Y$ gives rise to a  $\widetilde{G}$-action on $\cA$.

	\begin{claim}\label{claim:factor}
		The composition $Y\to \cA_Y\to \cA$ factors through $\alpha:X^{\! \rm sp}\to \cA$.   Moreover, $\alpha$ is $\widetilde{G}$-equivariant. 
	\end{claim} 
	\begin{proof}[Proof of \cref{claim:factor}]
		Since $\xsp$ is normal,  each  fiber $F$ of $Y\to \xsp$ is connected. 
		We have $\eta|_{F}=0$ for any $\eta\in \{\mu^*\eta_1,\ldots,\mu^*\eta_m\}$.  It follows from \cref{lem:critpointalb} that $F$ is mapped to one point under the morphism $Y\to \cA$.  Hence it factors through $\xsp\to \cA$.   The second one is easy to show since all our previous constructions are \(\widetilde{G}\)-equivariant.
	\end{proof}

	\begin{dfn}[Partial quasi-Albanese morphism]\label{def:partial}
		The above morphism $\alpha: \xsp\to \cA$ is called the partial quasi-Albanese morphism  induced by the spectral one forms $\{\eta_1,\ldots,\eta_m\}$.  
	\end{dfn}

	We denote by $p:Y\to X$ the composite of $\mu|_Y:Y\to \xsp$ and $\pi:\xsp\to X$.

	\begin{claim}\label{claim:trivial}
Let $f : Z \to Y$ be a morphism from a smooth quasi-projective variety $Z$. 
	Then the following  assertions are equivalent:
	\begin{enumerate}[label*=(\arabic*)]
	\item \label{item:20251202}
	the image $p^*\rho({\rm Im}[\pi_1(Z)\to \pi_1(Y)])$ is a bounded subgroup of $G(K)$.
	\item \label{item:202512021}
	The composite $\alpha\circ \mu\circ f:Z\to \cA$ is constant.
	\item \label{item:202512022}
	The pluriharmonic map $u\circ \widetilde{p\circ f}:\widetilde{Z}\to\Delta(G)$ is constant, where $\widetilde{p\circ f}: \widetilde{Z} \to \widetilde{X}$ is a lift of $p\circ f:Z\to X$ between the universal covers of $Z$ and $X$.  
	\end{enumerate}
	\end{claim}
	\begin{proof}
	This directly follows from \cref{claim:202511272} and Lemma~\ref{lem:critpointalb}.
	Indeed, according to \cref{claim:202511272}, it is enough to prove the equivalence of condition \ref{item:20251201} in \cref{claim:202511272} and condition \ref{item:202512021} in \cref{claim:trivial}.
	This equivalence follows from Lemma~\ref{lem:critpointalb}.
	\end{proof}

	\medspace

\noindent {\em Step 4. Construction of the reduction map $s_\rho:X\to S_\rho$.}   
The map $\alpha:X^{\! \rm sp}\to \cA$ is $\widetilde{G}$-equivariant with respect to the $\widetilde{G}$-actions on $X^{\! \rm sp}$ and $\cA$ (cf. \cref{claim:factor}).
	As $X=X^{\! \rm sp}/\widetilde{G}$, this gives rise to a morphism $\beta:X\to \cA/\widetilde{G}$ satisfying the following commutative diagram 
\begin{equation*}
	 		\begin{tikzcd}
	 Y\arrow[r,"\mu"]\arrow[dr,"p"'] &	\xsp\arrow[rr, "\alpha"] \arrow[d, "\pi"] & &  \cA\arrow[d]\\
&	 	X\arrow[rr, "\beta"] \arrow[dr, "s_\varrho"'] && \cA/\widetilde{G}\\
&	 	&S_\varrho\arrow[ur]&
	 \end{tikzcd} 
	\end{equation*}
	where  $s_\varrho:X\rightarrow S_\varrho$ is the quasi-Stein factorisation of $\beta$.

The following claim establishes our theorem (\cref{thm:KZreduction}) under the assumption that $X$ is smooth.
 \begin{claim}\label{claim:bounded subgroup}
	Let $T$ be any connected Zariski closed subset of $X$. Then the following properties are equivalent.
	\begin{enumerate}[label*=(\arabic*)]
	\item The image $s_\varrho(T)$ is a point in $S_\varrho$.			
	\item  The image $\rho({\rm Im}[\pi_1(T)\to \pi_1(X)])$ is a bounded subgroup of $G(K)$.
		\item  For every irreducible component $T_o$ of $T$, the image $\rho({\rm Im}[\pi_1(T_o^{\rm norm})\to \pi_1(X)])$ is a bounded subgroup of $G(K)$. 
		\end{enumerate} 	
\end{claim} 
\begin{proof}
The implication (2) $\Rightarrow$ (3) is obvious from ${\rm Im}[\pi_1(T_o^{\rm norm})\to \pi_1(X)]\subset {\rm Im}[\pi_1(T)\to \pi_1(X)]$.

Next we show (3) $\Rightarrow$ (1).
Let $T_0$ be an irreducible component of $T$.
 Set \( p = \mu \circ \pi : Y \to X \). 
 There exists an irreducible component \( W \) of \( p^{-1}(T_o) \) such that \( p(W) = T_o \). 
 Let \( Z \to W \) be a desingularization and let $f:Z\to Y$ be the composite of $Z\to W$ and $W\hookrightarrow Y$.
 Then the natural map $Z\to T_o$ factors $Z\to T_o^{\rm norm}$.
 Hence our condition (3) implies that $\rho({\rm Im}[\pi_1(Z)\to \pi_1(X)])$ is bounded.
 Therefore $\alpha\circ\mu\circ f:Z\to\cA$ is constant (cf. \cref{claim:trivial}).
 Hence the image $s_\varrho(T_o)$ is a point in $S_\varrho$.
 Since $T$ is connected,  the image $s_\varrho(T)$ is a point in $S_\varrho$.

 We prove (1) $\Rightarrow$ (2).
 We denote by \( \widetilde{T} \) the universal cover of $T$.
 Let $u_T:\widetilde{T}\to\Delta(G)$ be the composite of the induced map $\widetilde{T}\to\widetilde{X}$ and $u:\widetilde{X}\to\Delta(G)$.
 We show that \(u(\widetilde{T})\) is a single point \(P \in \Delta(G)\). 
 Let $ \widetilde{T}^{\rm reg}$ be the regular part of  $\widetilde{T}$.
 It is enough to show that $u$ is constant over each connected component $\mathcal{W}$ of $ \widetilde{T}^{\rm reg}$.
 Note that $\widetilde{T}^{\rm reg}\to T^{\rm reg}$ is a local homeomorphism, where $T^{\rm reg}$ is the regular part of $T$. Let $W$ be the image of $\mathcal{W}$. Then $W$ is a connected component of $T^{\rm reg}$, and the induced map $\mathcal{W}\to W$ is a local homeomorphism.
 Let $Z$ be a desingularization of an irreducible component $C$ of the locally closed subset $p^{-1}(W)\subset Y$ such that $C$ dominates $W$.
 Then the induced map $Z\to W$ is a proper surjective morphism of smooth quasi-projective varieties.
 Therefore the induced map $\widetilde{Z}\to\widetilde{W}$ between universal covers is surjective.
 Thus the composite $\widetilde{Z}\to\widetilde{W}\to\mathcal{W}$ is also surjective.
  The first item (1) shows that the map $Z\to \cA$ is constant.
  By \cref{claim:trivial}, the map $\widetilde{Z}\to \Delta(G)$ is constant, hence $\mathcal{W}\to\Delta(G)$ is constant.
 This shows that \(u_T(\widetilde{T})\) is a single point \(P \in \Delta(G)\). 
 This implies that 
\[
\varrho\bigl({\rm Im}[\pi_1(T) \to \pi_1(X)]\bigr)
\]
fixes the point \(P\), and therefore is bounded.  
The implication is thus proved.
 \end{proof} 
 
 Hence we prove the theorem when $X$ is smooth. 
 
 \medspace
 
 \noindent {\it Step 5. We do not assume that $X$ is smooth.}
 Let $\mu:Y\to X$ be a resolution of singularities. Then since $X$ is normal, $\mu^\varrho:\pi_1(Y)\to G(K)$ is also a Zariski dense representation (cf. \cref{lem:fun}).
 By \cref{claim:bounded subgroup}, the desired reduction map $s_{\mu^*\varrho}:Y\to S_{\mu^*\varrho}$ exists.  
 
 Let $F$ be any fiber of $\mu$, which is connected and compact as $X$ is normal. Obviously, $\mu^*\varrho({\rm Im}[\pi_1(F)\to \pi_1(Y)])=\{e\}$. Therefore, $s_{\mu^*\varrho}(F)$ is a point. Hence there exists a morphism $s_\varrho:X\to S_{\mu^*\varrho}$ such that $s_\varrho\circ\mu=s_{\mu^*\varrho}$. We will prove that it satisfies the required property in the theorem. 
 
   Let $T':=\mu^{-1}(T)$. Since $X$ is normal, it follows that each fiber of $\mu$ is connected. Hence the natural morphism $T'\to T$ has connected fibers. By \cite[Lemma 3.47]{DY23} (cf. \cref{rem:20250910} below), we know that  $\pi_1(T')\to \pi_1(T)$ is surjective.   
   
 \medspace
 
 	\noindent {\em Proof of (3) $\Rightarrow$ (1)}.  Since $s_{\mu^*\varrho}(T')=s_\varrho(T)$ is a point, it follows that
 	$$
 	 \varrho({\rm Im}[\pi_1(T)\to \pi_1(X)]) = \mu^*\varrho({\rm Im}[\pi_1(T')\to \pi_1(Y)]) 	 
 	$$
 	is bounded. 
 	
 	\medspace
 	
 	 	\noindent {\em Proof of (1) $\Rightarrow$ (2)}. This is obvious.
 	 		
 	 		\medspace
 	 		
 	 		 	\noindent {\em Proof of (2) $\Rightarrow$ (3)}.  We take  an irreducible component, denoted as $T_o'$, of $T'$ that dominates $T_o$. Considering
 	 		 	$$\mu^*\varrho({\rm Im}[\pi_1(T_o')\to \pi_1(Y)])\subset  	 \varrho({\rm Im}[\pi_1(T_o)\to \pi_1(X)]),$$ 
 	 		 we can conclude that $\mu^*\varrho({\rm Im}[\pi_1(T_o')\to \pi_1(Y)])$ is   bounded.   By \cref{claim:bounded subgroup}, $s_{\mu^*\varrho}(T_o')$ is a point. This leads  to  the conclusion that $s_\varrho(T_o)$ is a  point as well. Given the connectedness of $T$, we  conclude that $s_\varrho(T)$ as a point.

 We complete the proof of the theorem.   
\end{proof}

We finish this section with some remarks  on the above long proof. 
\begin{rem}
	For the purpose of the proof of \cref{main2}, we  only need to study the properties of the morphism $\alpha:\xsp\to \cA$ together with the precise information on the ramification locus $\pi: {\xsp}\to  {X}$ in \cref{claim:ramified}. The reduction $s_\varrho:X\to S_\varrho$ will be an important tool in studying the linear Shafarevich conjecture for quasi-projective varieties (see \cite{Eys04,EKPR12}, the very recent work by Green-Griffiths-Katzarkov \cite{GGK22}, and \cite{DY23,DY23b}). 
\end{rem}

\begin{rem}\label{rem:20250910}
We recall and prove \cite[Lemma 3.47]{DY23} for completeness.

\begin{lem}\label{lem:surjective}
	Let $f:X\to Y$ be a proper surjective morphism between connected (possible reducible) quasi-projective varieties $X$ and $Y$.    Assume that each fiber of $f$ is connected. Then $f_*: \pi_1(X)\to \pi_1(Y)$ is surjective. 
\end{lem}
\begin{proof}
	Let $\widetilde{Y}\to  Y$ be the universal covering of $Y$. Consider the fiber product $X':=X\times_Y\widetilde{Y}$.  
	\begin{equation*}
		\begin{tikzcd}
			X' \arrow[r, "\pi_1"] \arrow[d, "f'"]& X\arrow[d, "f"]\\
			\widetilde{	Y} \arrow[r, "\pi_2"] & Y
		\end{tikzcd}
	\end{equation*}Since $f$ is proper and each fiber of $f$ is connected, it follows that $f':X'\to \widetilde{Y}$ is proper, surjective and each fiber of $f'$ is  connected.  
	\begin{claim}
		$X'$ is connected.
	\end{claim}
	\begin{proof}
		Assume by contradiction that  	$X'$ is not connected. Then $X'=\sqcup_{\alpha\in I}X_\alpha$, where $X_\alpha$  are  connected components   of $X'$.  Since each fiber of $f'$ is connected, it follows that any fiber of $f'$ is contained in  some $X_\alpha$. This implies that $f'(X_\alpha)\cap f'(X_\beta)=\varnothing$ if $\alpha\neq \beta$. Since $f'$ is surjective, it follows  $\sqcup_{\alpha\in I}f'(X_\alpha)=\widetilde{	Y}$.  Note that $f'(X_\alpha)$ and $f'(\sqcup_{\beta\in I, \beta\neq\alpha}X_\beta)$ are both closed since  $f'$ is proper.  This contradicts with the connectedness of  $\widetilde{	Y}$. Hence  	$X'$ is connected. 
	\end{proof}
	We choose a base point $x\in X$ and $y\in Y$  such that $y=f(x)$.  Let $x'\in X'$ and $y'\in \widetilde{	Y}$ be such that $x=\pi_1(x')$, $y=\pi_2(y')$ and $y'=f'(x')$. Then for any element $\gamma\in \pi_1(Y,y)$, the  lift of $\gamma$ in $\widetilde{Y}$ starting at $y'$ will end at some $y''$ such that $y=\pi_2(y'')$. 
	Since $X'=X\times_Y\widetilde{Y}$, it follows that there exists a unique $x''\in f'^{-1}(y'')$ such that $x=\pi_1(x'')$. 
	Since $X'$ is connected, there exists a continuous path $\sigma:[0,1]\to X'$ such that $x'=\sigma(0)$ and $x''=\sigma(1)$. Consider the continuous path $f'\circ\sigma:[0,1]\to \widetilde{	Y}$. Then $y'=f'\circ\sigma(0)$ and $y''=f'\circ\sigma(1)$. It follows that $\gamma=[\pi_2\circ f'\circ \sigma]=[f\circ \pi_1\circ \sigma]$.  Note that $x=\pi_1\circ \sigma(0)=\pi_1\circ \sigma(1)$.    This proves that $f_*([\pi_1\circ \sigma])=\gamma$. Therefore, $f_*:\pi_1(X,x)\to \pi_1(Y,y)$ is surjective. The lemma is proved. 
\end{proof}

\end{rem}

\section[Hyperbolicity and non-Archimedean local systems]{Hyperbolicity and non-archimedean local systems}
\label{sec:spectral}

This section is devoted to the proof of \cref{main6}.  
In \cref{sub:spectral}, we show that for the spectral cover $\xsp \to X$ associated to $\varrho$ in \cref{main6}, defined in \cref{def:spectral}, the variety $\xsp$ is of log general type, and the partial quasi-Albanese morphism induced by its spectral $1$-forms (cf. \cref{def:spectral}) is generically finite onto its image.  
In \cref{sec:extract}, based on \cref{sub:spectral}, we extract some essential properties of varieties admitting unbounded non-Archimedean representations of $\pi_1$ in a more general framework (see \cref{property:20250914}).  
In \cref{sec:GGL}, we recall some results on hyperbolicity from Part~I of our series \cite{CDY25}.  
Finally, in \cref{subsection:spread1} we spread positivity from suitable ramified coverings, and in \cref{subsection:spread2} we complete the proof of \cref{main6}.

 \subsection{Positivity  of log canonical bundle of  spectral cover} \label{sub:spectral}
 \begin{thm}\label{thm:spectral cover}
	 Let $X$ be a smooth quasi-projective variety.  Assume that there is a Zariski dense representation $\varrho:\pi_1(X)\to G(K)$ where $G$ is an almost simple algebraic group defined over a non-archimedean local field $K$.   When $\varrho$ is big and unbounded, then \begin{thmlist}
	 	\item \label{spectral general type}the spectral cover $\xsp$ of $X$ defined in  \cref{def:spectral} is of log general type.  
	 	\item  \label{genericallyfinite}  Let $\alpha:\xsp\to \cA$  be the partial quasi-Albanese map induced by   the spectral one forms $\{\eta_1,\ldots,\eta_m\}\subset  H^0( {\xsp}, \pi^*\Omega_{ {X}})  $  defined in \cref{def:partial}. Then $\dim \xsp=\dim \alpha(\xsp)$. 
	 	\end{thmlist}
  \end{thm}
\begin{proof}
	\noindent
	{\em Step 1. We replace \(\xsp\) by a smooth model.} 
	We will use the notations in \cref{sec:KZthm}. Let \(\mu:Y \to \xsp\) be a resolution of singularities, and let \(p : Y \to X\) be the composite map. 
Since the image of \(p_{\ast} : \pi_{1}(Y) \to \pi_{1}(X)\) has finite index by \cref{lem:finiteindex}, the Zariski closure $H$ of   \(p^{\ast}\varrho(\pi_1(Y))\) contains the identity component $G^o$ of $G$, hence is also almost simple.  We have also:
	\begin{claim}\label{claim:20251203}
		The representation \(p^{\ast}\varrho\) is big.
	\end{claim}
\begin{proof}
	For any closed   subvariety $Z\subset Y$ containing a  very general point in $Y$, its image $p(Z)$ is   closed subvariety  passing to a very general point in $X$.  Since $ Z\to p(Z)$ is surjective, ${\rm Im}[\pi_1(Z^{\rm norm})\to \pi_1(p(Z)^{\rm norm})]$  has   finite index in  $\pi_1(p(Z)^{\rm norm})$ by Lemma~\ref{lem:finiteindex}.  
	 Hence  $p^*\varrho({\rm Im}[\pi_1(Z^{\rm norm})\to \pi_1(Y)])$ has finite index in $\varrho({\rm Im}[\pi_1(p(Z)^{\rm norm})\to \pi_1(X)])$.  
	 Since \(p(Z)\) contains a very general point of $X$, this latter group is infinite, which   implies that $p^*\varrho({\rm Im}[\pi_1(Z^{\rm norm})\to \pi_1(Y)])$ is infinite. This proves our claim.
\end{proof}
	 
	 \medskip

	 \noindent
	 {\em Step 2. We show that \(\overline{\kappa}(Y) \geq 0\).} 
	 Let \(q : Y \to \cA\) be the composite map of  $\mu:Y\to \xsp$ and $\alpha:\xsp\to \cA$. 
	 Let \(F\subset Y\) be a connected component of a general fiber of \(q : Y \to q(Y)\).
	 Then $F$ is a smooth quasi-projective variety.
	 \cref{claim:trivial} implies that $p^*\varrho({\rm Im}[\pi_1(F)\to \pi_1(Y)])$  is a bounded subgroup of $H(K)$.  
	 However, $p^*\varrho({\rm Im}[\pi_1(F)\to \pi_1(Y)])$ is a normal subgroup of $p^*\varrho(\pi_1(Y))$ by \cref{lem:normal}. 
	  Since $p^*\varrho:\pi_1(Y)\to H(K)$ is Zariski dense and unbounded,   \cref{lem:BT} yields  the finiteness of $p^*\varrho({\rm Im}[\pi_1(F)\to \pi_1(Y)])$. 
	   However,  by \cref{claim:20251203}, \(p^{\ast}\varrho\) is big. 
	   Then \(F\) must be a point. This implies that  $q:Y\to \cA$, hence  $\alpha:\xsp\to \cA$ is generically finite onto its image.  Let $Z$ to be the Zariski closure of $q(Y)$.  By \cref{prop:Koddimabb}, $\bar{\kappa}(Z)\geq 0$.   Hence $\overline{\kappa}(Y)\geq 0$ for $q:Y\to Z$ is dominant and generically finite. 
	 \medskip

	\noindent
	{\em Step 3. We show that the log--Iitaka fibration of \(Y\) is trivial.}
	We may replace \(Y\) with a birational modification so that the log-Iitaka fibration of \(Y\) is well-defined as a dominant morphism $f:Y\to B$ with connected general fibers. 
	Note that a very general fiber $F$ of $f$ is a connected smooth quasi-projective variety with $\overline{\kappa}(F)=0$. 
	Moreover, since \(q:Y\to \cA\) is generically finite onto its image, so is its restriction $q|_{F}:F\to \cA$. 
	By \cref{lem:abelian pi}, it follows that $\pi_1(F)$ is abelian.  
	Hence $p^*\varrho({\rm Im}[\pi_1(F)\to \pi_1(Y)])$   is an abelian subgroup of $G(K)$.   Note that $\pi_1(F)\triangleleft \pi_1(Y)$ by \cref{lem:normal}. 
	Since $p^*\varrho: \pi_1(Y)\to H(K)$ is Zariski dense and $H$ is   almost simple, it follows that the Zariski closure of $p^*\varrho({\rm Im}[\pi_1(F)\to \pi_1(Y)])$  is either  finite or a finite index subgroup of $H$. 
	The second case cannot happen since  $p^*\varrho({\rm Im}[\pi_1(F)\to \pi_1(Y)])$   is abelian. 
	Therefore, $p^*\varrho({\rm Im}[\pi_1(F)\to \pi_1(Y)])$  must be finite. 
	Now, since \(p^*\varrho\) is big (cf. \cref{claim:20251203}), this implies that $F$ is a point. 
	We conclude that $Y$, hence $X^\sp$  is of log general type. 
\end{proof}

\subsection{The property of varieties with non-archimedean representations of fundamental groups}\label{sec:extract}
For later applications, as well as for its independent interest, we formulate the following property for a smooth quasi-projective variety $X$. 
\begin{property}\label{property:20250914}
	Let $X$ be a smooth quasi-projective variety. Assume that there exist a smooth projective compactification $\overline{X}$ with simple normal crossing boundary divisor $D=\overline{X}\backslash X$, and a finite surjective morphism $\pi:\overline{\Sigma}\to \overline{X}$ from a normal projective variety $\overline{\Sigma}$ together with nonzero sections $\tau_1,\ldots,\tau_l \in H^0(\overline{\Sigma}, \pi^*\Omega_{\overline{X}}(\log D))$ satisfying the following two conditions:
	\begin{thmlist} 
		\item \label{cond:i1}
		setting $\Sigma=\pi^{-1}(X)$, the variety $\Sigma$ is of log-general type and admits a morphism $a:\Sigma\to A$ into a semi-abelian variety $A$ with $\dim \Sigma=\dim a(\Sigma)$;
		\item \label{cond:iii1}
		defining
		\[ 
		R:=\{s\in \overline{\Sigma} \mid \exists\, i \ \text{such that } \tau_i(s)=0 \}, 
		\]
		then $R\subsetneqq \overline{\Sigma}$ is a proper Zariski closed subset, and $\pi:\overline{\Sigma}\to \overline{X}$ is \'etale outside $R$.
	\end{thmlist}
\end{property} 
\begin{lem}\label{lem:property}
	 	Let $X$ be a smooth quasi-projective variety.  Let $G$ be an almost simple algebraic group defined over a non-archimedean local field $K$.  Suppose that $\varrho:\pi_1(X)\to G(K)$ is a big and unbounded Zariski dense representation. Then $X$ satisfies \Cref{property:20250914}.
\end{lem}
 \begin{proof}
 	 Let $\pi:\overline{\xsp}\to \overline{X}$ be the spectral cover associated to $\varrho$ defined in \cref{def:spectral}, and $\{\eta_1,\ldots,\eta_m\}\in H^0(\overline{\xsp}, \pi^*\Omega_{ \overline{X}}(\log D))$ the resulting spectral one forms defined in \cref{def:spectral}.  
 	 By \cref{genericallyfinite} the partial quasi-Albanese morphism $\alpha:\xsp\to \cA$ associated to  $\{\eta_1,\ldots,\eta_m\}$ satisfies that $\dim \xsp=\dim \alpha(\xsp)$.
 	 By \cref{spectral general type}, $\xsp$ is of log-general type.
 	 Set $\tau_{ij}=\eta_i-\eta_j$ for $\eta_i\neq\eta_j$.
 	 The by \cref{claim:ramified}, $\pi$ is \'etale outside $\{\tau_{ij}=0\}$.
 	 Hence $X$ satisfies \Cref{property:20250914}.
 \end{proof}
In the remainder of this section, we will see that if a quasi-projective variety $X$ satisfies \Cref{property:20250914}, then $X$ is of log general type, and any Zariski dense holomorphic map $f:\bD^* \to X$ has no essential singularity at the origin.  
In \cref{claim:202509132}, we establish the existence of a proper Zariski closed subset $E \subsetneqq X$ such that, for any closed subvariety $V \not\subset E$, the variety $V$ also satisfies \Cref{property:20250914}.  
Combined with \cref{lem:property}, this yields the proof of \cref{main6}.

 \subsection{Two results on hyperbolicity}\label{sec:GGL} 
 We recall two results from Part I of our paper series, which are applied for the proof of \cref{main6}.
 These results are proved by application of Nevanlinna theory.

 \begin{thm}[{\cite[Theorem C]{CDY25}}]\label{cor:GGL}
 	Let $X$ be a smooth quasi-projective variety. 
 	Assume that there is a morphism $a:X\to A$ to a semi-Abelian variety $A$ such that $\dim X=\dim a(X)$.
 	Then the following properties are equivalent:  
 	\begin{enumerate}[wide = 0pt,  noitemsep,  font=\normalfont, label=(\alph*)] 
 		\item \label{being general type1} $X$ is of log general type; 
 		\item \label{strong LGT1} 
 		$X$ is strongly of log general type;  
 		\item  \label{pseudo Picard1} $X$ is pseudo Picard hyperbolic;
 		\item \label{pseudo Brody} $X$ is pseudo Brody hyperbolic;
 		\item \label{spab}
 		$\Spab(X)\subsetneqq X$.  \qed
 	\end{enumerate}
 \end{thm}

 \begin{thm}[{\cite[Theorem B]{CDY25}}]\label{thm:20250911}
 	Let $X$ be a smooth quasi-projective variety. 
 	Assume that  $X$ satisfies \Cref{property:20250914}, then every holomorphic map $f:\mathbb D^*\to X$ with Zariski dense image has a holomorphic extension $\bar{f}:\mathbb D\to\overline{X}$, where $\overline{X}$ is any smooth projective compactification of $X$. \qed
 \end{thm}

  \subsection{Spread positivity from suitable ramified covering}\label{subsection:spread1} 
\begin{proposition}\label{prop:log general type}
 	 	Let $X$ be a smooth quasi-projective variety.  If $X$ satisfies \Cref{property:20250914}, then  it   is   of log general type.    
\end{proposition} 
 \begin{proof} 
 Let us use the following objects described in \Cref{property:20250914}: $\pi:\overline{\Sigma}\to\overline{X}$, $\tau_1,\ldots,\tau_l \in H^0(\overline{\Sigma}, \pi^*\Omega_{\overline{X}}(\log D))$, and $R\subset \overline{\Sigma}$.
Note that any finite and surjective morphism between
 normal varieties can be enlarged to become Galois without changing the branch loci (see e.g. \cite[Theorem 3.7]{GKP16}).  
 Namely, there exists a finite and surjective morphism from a projective normal variety $\nu:\overline{Y}\to \overline{\Sigma}$ such that the composition $p:=\pi\circ\nu$ is a Galois cover with Galois group $G$, and the branch loci of $p$ and $\pi$ are the same, which we denote by $B\subset \overline{X}$. 
 Then $p:\overline{Y}\to\overline{X}$ is \'etale outside $p^{-1}(B)$.
By \Cref{cond:iii1}, we have $B\subset \pi(R)$.
 Hence $p$ is \'etale outside $R':=p^{-1}(\pi(R))$.

 We define a section
 	 	$$
 	 	\sigma:=\prod_{g\in G}\prod_{i\in\{1,\ldots,l\}}g^*\nu^*\tau_i\in H^0(\overline{Y}, \Sym^{N}p^*\Omega_{\overline{X} }(\log D)), 
 	 	$$
 	 	which is non-zero and vanishes at $\nu^{-1}(R)$ by the definition of $R$ (cf. \Cref{cond:iii1}). 
 	 	Since $\sigma$ is invariant under $G$-action, it follows that $\sigma$ descends to a section
 	 	$$
 	 	\sigma^{G}\in H^0(\overline{X}, \Sym^{N}\Omega_{\overline{X}}(\log D))
 	 	$$
 such that $p^*\sigma^{G}=\sigma$.      
 Since $\sigma$ vanishes on $\nu^{-1}(R)$, it follows that $\sigma^G$ vanishes at $p(\nu^{-1}(R))=\pi(R)$, which is a proper Zariski closed subset of $\overline{X}$.
  Let $E\subset \pi(R)$ be the sum of prime divisors contained in $\pi(R)$.    This implies that there is a non-trivial morphism
 	 	\begin{align}\label{eq:CP1}
 	 		\mathcal{O}_{\overline{X}}(E)\to \Sym^{N}\Omega_{\overline{X}}(\log D).
 	 	\end{align} 
By \cref{cond:i1}, $\Sigma$ is of log general type. 
Since $\pi|_{\Sigma\backslash R}:\Sigma\backslash R\to X\backslash \pi(R)$ is finite \'etale, by \cite[Theorem 3]{Iitaka1977} or \cref{lem:KodairaDim} below,   
$X\backslash \pi(R)$ is also of log general type.  By our construction, $\pi(R)\backslash E$ is of codimension at least two in $\overline{X}$.     By \cite[Lemma 3]{NWY13}, $K_{\overline{X}}+E+D$ is big.  
  	 	 %	Let   $\overline{Y}\to \overline{S}$ be a strict desingularization so that for the induced morphism $f:\overline{Y}\to \overline{W}$,  $D_Y:=f^{-1}(D)$ is a simple normal crossing divisor.
% Hence
 	 %	$
 	 %	f^*D-D_Y
 	 %	$ is effective.  Note that $f$ is \'etale outside $R':=f^{-1}(\pi(R_0))$. It follows that
 	 %	$
 	 %	K_{\overline{Y}}-f^*K_{\overline{W}}=R_1
 	 	%$ 
 	 %	where $R_1$ is an effective divisor supported at $R'$. 
 	 %	Hence 		 $
 	 %	f^*(K_{\overline{W}}+m\pi(R_0))-K_{\overline{Y}}
 	 %	$ is effective for some integer $m>0$. Therefore, 
 	 %	$
 	 %	f^*(K_{\overline{W}}+m\pi(R_0)+D)-(K_{\overline{Y}}+D_Y)
 	 %	$
 	 %	is effective.  Recall that $K_{\overline{Y}}+D_Y$ is big by $V\not\subset E_2$.
	%	Hence $f^*(K_{\overline{W}}+m\pi(R_0)+D)$ is big. By \cite[Proposition 4.12]{Bou02}, $K_{\overline{W}}+m\pi(R_0)+D$ is also big. 
	Together with \eqref{eq:CP1} we can apply \cite[Corollary 8.7]{CP19} to conclude that $K_{\overline{X}}+ D$ is big.
		Hence  $X$ is of log general type.  
 	 \end{proof}

\begin{lem}\label{lem:KodairaDim}
	Let $f':U\to V$ be a finite \'etale morphism between smooth quasi-projective varieties.   
	Then the logarithmic Kodaira dimension $\bar{\kappa}(U)=\bar{\kappa}(V)$. 
\end{lem}
\begin{proof}
	We first take a smooth projective compactification $Y$ of $V$  so that $D_Y:=Y-V$ is a simple normal crossing divisor.  By \cref{sec:covGal}, there is a normal projective variety $X$ compactifying $U$ so that $f'$ extends to a finite morphism $f:X\to Y$.  Let $\mu:Z\to X$ be a strict desingularization so that $\mu^{-1}(U)\simeq U$ and $D_Z:=Z-\mu^{-1}(U)$ is   a simple normal crossing divisor.  Write $g=f\circ\mu$.  
	\begin{claim}\label{claim:exceptional}
		$E:=K_{Z}+D_Z-g^*(K_Y+D_Y)$ is an effective \emph{exceptional} divisor.  
	\end{claim}
	\begin{proof}[Proof of \cref{claim:exceptional}]
		Since $g:(Z,D_Z)\to (Y,D_Y) $ is a log morphism, $E$ is effective. 	Let $D_Y^{\rm sing}$ be the singularity of $D_Y$  which is a Zariski closed subset of $Y$ of codimension at least two.  Write $Y^\circ:=Y-D_Y^{\rm sing}$, and $X^\circ:=f^{-1}(Y^\circ)$. Note that $X^\circ$ is smooth, and $D_X^\circ:=X^\circ-U$ is a smooth divisor in $X^\circ$.  Moreover, it follows from the proof of \cite[Lemma A.12]{Den22} that at any $x\in D_X^\circ$ we can take a holomorphic coordinate  $(\Omega;x_1,\ldots,x_n)$ around $x$ with $D^\circ_X\cap \Omega=(x_1=0)$ and a holomorphic coordinate  $(\Omega';y_1,\ldots,y_n)$ around $f(x)$ with $D_Y\cap \Omega'=(y_1=0)$  so that 
		\begin{align*}
			f(x_1,\ldots,x_n)=(x_1^k,x_2,\ldots,x_n).
		\end{align*}
		This implies that $K_{X^\circ}+D_X^\circ=f^*(K_{Y^\circ}+D_Y^\circ)$. Since $\mu$ is a strict desingularization, it follows that $\mu^{-1}(X^\circ)\simeq X^\circ$.   Therefore,   $\mu (K_{Z}+D_Z-g^*(K_Y+D_Y))$ is contained in $f^{-1}(D_Y^{\rm sing})$  which is of codimension at least two. 
	\end{proof}
Therefore, one has
	$$
	\kappa(K_Z+D_Z)=	\kappa(g^*(K_Y+D_Y)+E)=\kappa(g^*(K_Y+D_Y))=	\kappa(K_Y+D_Y). 
	$$ 
	where the first equality follows from the above claim, the second one is due to  \cite[Example 2.1.16]{Laz04} and the last one follows from \cite[Lemma 5.13]{Uen75}.   This concludes that $\bar{\kappa}(U)=\bar{\kappa}(V)$. 
\end{proof}

  \subsection{Proof of \cref{main6}}\label{subsection:spread2} 
\begin{thm}[=\Cref{main6}]\label{thm:main33}
	Let $X$ be a complex  smooth quasi-projective variety and let $G$ be an almost simple algebraic group over some non-archimedean local field $K$. If  $\varrho:\pi_1(X)\to G(K)$ is  a big  representation which is   Zariski dense and unbounded, then $X$ is both strongly of log general type, and pseudo Picard hyperbolic. 
	\end{thm}
	 The proof is based on \cref{thm:20250911,cor:GGL,prop:log general type} combined with \cref{lem:property}.
	We first prove the following lemma.

\begin{lem}\label{claim:202509132}
Let $X$ be a smooth quasi-projective variety satisfying \Cref{property:20250914}.
Then there exists a proper Zariski closed subset $E\subsetneqq X$ such that if $V\subset X$ is a closed subvariety with $V\not\subset E$ and of positive dimension, then every smooth modification $W\to V$ satisfies \Cref{property:20250914}.
\end{lem}

\begin{proof}
Let us use the following objects described in \Cref{property:20250914}: $\pi:\overline{\Sigma}\to\overline{X}$, $\tau_1,\ldots,\tau_l \in H^0(\overline{\Sigma}, \pi^*\Omega_{\overline{X}}(\log D))$, $a:\Sigma\to A$, and $R\subset \overline{\Sigma}$.
Moreover we set $Z_{i}:=\{z\in\overline{\Sigma}\mid \tau_i(z)=0\}$ so that $R=\cup_iZ_i$.

\noindent {\bf Step 1.}
{\em We construct    $E\subsetneqq X$.}
Given $i$, we first construct a proper Zariski closed subset $F_i\subsetneqq \Sigma$ with the following property: 
If $g:T\to \Sigma$ is a morphism from a smooth quasi-projective variety $T$ such that $g^*\tau_i$ vanishes in $H^0(T,\Omega_T)$ and $g(T)\not\subset F_i$, then $g(T)\cap Z_i=\emptyset$.

The construction of $F_i\subsetneqq \Sigma$ goes as follows.
Let $\mu:  \overline{Y}\to \overline{\Sigma}$ be a resolution of singularities with $Y:=\mu^{-1}(\Sigma)$.  
Let $\alpha_Y:Y\to \cA_Y$ be the quasi-Albanese morphism.
By the universal property of the quasi-Albanese morphism, the composite $a\circ\mu:Y\to A$ factors $\alpha_Y$.
Set $\omega_i:=\mu^*\tau_i$. 
We know that there exist a log one form $\widetilde{\omega}_i\in T_1(\cA_Y)$ so that $\omega_i=\alpha_Y^*\widetilde{\omega}_i$. 
Consider the maximal semi-abelian subvariety $B_i$ of $\cA_Y$  on which $\widetilde{\omega}_i$ vanishes. 
 Denote by $q_i:\cA_Y\to \cA_Y/B_i$ the quotient map and by $\beta_i:Y\to \cA_Y/B_i$ the composed morphism.
 Write $Z_{i}':=(\omega_i=0)$, which is a proper Zariski closed subset  of $Y$.
By \cref{lem:critpointalb}, for any irreducible component  $Z$ of $Z_{i}'$,  $\beta_i(Z)$ is a point in $\cA_Y/B_i$.   
	Since $Z_{i}'$ has   finitely many irreducible components, it follows that $\beta_i(Z_{i}')$ is a set of finite points in $\cA_Y/B_i$. 
	Denote by $F_{i}':=\beta_i^{-1}(\beta_i(Z_i'))$, which is a proper Zariski closed subset of $Y$.  
Enlarging $F_i'$, we may assume that $\mu:Y-F_i'\to \Sigma$ is an open immersion.
We set $F_i=\mu(F_i')$.
Then $F_i\subsetneqq \Sigma$ is a proper  Zariski closed set.

To show $F_i$ has the desired property, let $g:T\to \Sigma$ be a morphism from a smooth quasi-projective variety $T$ such that $g^*\tau_i$ vanishes in $H^0(T,\Omega_T)$ and $g(T)\not\subset F_i$.
Since $\mu:Y\to \Sigma$ is an isomorphism over $F_i$, we may take a smooth modification $T'\to T$ such that $g:T\to\Sigma$ lifts to $g':T'\to Y$.
By $g^*\tau_i=0$, we have $(g')^*\omega_i=0$.
By \cref{lem:critpointalb}, we conclude that $\beta_i\circ g'(T')$ is a point in $\cA_Y/B_i$.
By the construction of $F_i'$, we have $\beta_i\circ g'(T')\not\subset \beta_i(Z_{i}')$.
Hence $g'(T)\cap Z_i'=\emptyset$.
By $\mu^{-1}(Z_i)\subset Z_i'$, we have $g(T)\cap Z_i=\emptyset$. 
Thus we have constructed $F_i\subsetneqq \Sigma$ with the desired property.

By \cref{cor:GGL}, we have $\Spalg(Y)\subsetneqq Y$.
By   \Cref{cond:i1}, we take a proper Zariski closed set $\Xi\subsetneqq \Sigma $such that $a:\Sigma-\Xi\to A$ is quasi-finite.
We set $E=\pi\circ\mu(\Spalg(Y))\cup \pi (\Xi)\cup\cup_i\pi(F_i)$.
Then $E\subsetneqq X$ is a proper Zariski closed subset.

\noindent  {\bf Step 2.}
Let $V\subset X$ be a positive dimensional closed subvariety such that $V\not\subset E$.
 We denote by $W\to V$ a smooth modification.
Let $\overline{W}$ be a smooth projective compactification of $W$ so that $D_{\overline{W}}:=\overline{W}-W$ is a simple normal crossing divisor and $(\overline{W},D_{\overline{W}})\to (\overline{X}, D)$ is a log morphism.  
 Let $\overline{S}$ be a normalization of an irreducible component of $\overline{W}\times_{\overline{X}}\overline{\Sigma}$. 
 \begin{equation*}
 	\begin{tikzcd}
 		\overline{S}  \arrow[r, "g"] \arrow[d, "p"]&  \overline{\Sigma}\arrow[d, "\pi"]\\
 		\overline{W} \arrow[r] & \overline{X}
 	\end{tikzcd}
 \end{equation*}

Set $S:=p^{-1}(W)$. 
 Then $g(S)\not\subset \Xi\cup\mu(\Spalg(Y))$.
 Hence $S$ is of log-general type and $\dim S=\dim a\circ g(S)$.
 Therefore  \Cref{cond:i1} of \cref{property:20250914} is satisfied.

 Let $\psi_i$ be the pull back of $\tau_i\in H^0(\overline{\Sigma}, \pi^*\Omega_{\overline{X}}(\log D))$ by $\overline{S}\to \overline{\Sigma}$. 
 Then   we can consider  $\psi_i$ as a section in $H^0(\overline{S}, p^*\Omega_{\overline{W}}(\log D_{\overline{W}}))$. 
 Let $I$ be the set of all $i$ such that
 \begin{itemize}
 	\item   the image of $S\to \Sigma$ intersects with $Z_{i}:=\{z\in\overline{\Sigma}\mid \tau_i(z)=0\}$. 
 \end{itemize}

 \begin{claim}
 For $i\in I$, $\psi_i\not=0$ in $H^0(\overline{S}, p^*\Omega_{\overline{W}}(\log D_{\overline{W}}))$.
 \end{claim}
	\begin{proof}
		Assume by contradiction that $\psi_i=0$ in $H^0(\overline{S}, p^*\Omega_{\overline{W}}(\log D_{\overline{W}}))$.
		Let $\varphi:T\to S$ be a smooth modification.
		Then $\varphi^*\psi_i=0$.
		Hence by the choice of $F_i$, we have $g\circ\varphi(T)\cap Z_i=\emptyset$, hence $g(S)\cap Z_i=\emptyset$.
		This contradicts to the assumption $i\in I$.
	\end{proof}

 We set
 \begin{align*} 
 	 		R':=\{z\in \overline{S} \mid \exists   i\in I  \mbox{ with } \psi_i(z)=0 \}.
 	 	\end{align*}  
By the claim above, $R'$ is a proper Zariski closed subset of $\overline{S}$.  	
Recall that  $\pi:\overline{\Sigma}\to \overline{X}$  is \'etale over $\overline{\Sigma}-R$, where $R:=\cup_{i}Z_{i}'.$   
 Since the base change of an \'etale morphism   is also étale, it follows that 
$p$ is \'etale over $\overline{S}-g^{-1}(R)$.  
We have $g^{-1}(R)\subset R'$, since for $i\not\in I$, the image of $S\to \Sigma$ does not intersect with $Z_{i}$, so such $Z_{i}$ does not contribute to $g^{-1}(R)\cap S$. 
Hence $p$ is \'etale over $\overline{S}-R'$.
The proof of \cref{claim:202509132} is completed.
\end{proof}

\begin{proof}[Proof of \cref{thm:main33}]
	By \cref{lem:property}, the variety $X$ satisfies \Cref{property:20250914}.
	Let $E \subsetneqq X$ be a proper Zariski closed subset as in \cref{claim:202509132}. 
	
	Let $Z \subset X$ be a positive-dimensional closed subvariety that is not contained in $E$, and let $Z' \to Z$ be a smooth modification. 
	By \cref{claim:202509132}, the variety $Z'$ satisfies \Cref{property:20250914}. 
	Hence, by \cref{prop:log general type}, $Z'$, and therefore $Z$, is of log general type.  
	It follows that 
	\[
	\Spalg \subset E \subsetneq X,
	\] 
	and thus $X$ is strongly of log general type. 
	
	\medskip
	
	Now let $f:\mathbb{D}^* \to X$ be a holomorphic map with $f(\mathbb{D}^*) \not\subset E$. 
	Let $V \subset X$ denote the Zariski closure of $f(\mathbb{D}^*)$. 
	Then $V \not\subset E$. 
	Let $W \to V$ be a smooth modification. 
	Then, by \cref{claim:202509132}, $W$ satisfies \Cref{property:20250914}. 
	Applying \cref{thm:20250911}, we conclude that $f$ extends to a holomorphic map 
	\[
	\bar{f}:\mathbb{D} \to \overline{X}.
	\] 
	Therefore, 
	\[
	\Spp \subset E \subsetneq X,
	\] 
	and $X$ is pseudo-Picard hyperbolic. 
	This completes the proof of the theorem. 
\end{proof}

\section{Rigid representation and $\bC$-VHS} \label{sec:rigid}
In this section we show that rigid representations of the fundamental groups of smooth quasi-projective varieties underlie  complex variations of Hodge structures (abbreviated $\bC$-VHS), which is the content of \cref{thm:vhs}.
This result was first established by Mochizuki in his long and technically involved paper \cite{Moc06}. For the sake of completeness, we include a proof here. Our argument relies on Mochizuki's Kobayashi--Hitchin correspondence in the non-compact case and on his correspondence between reductive representations of $\pi_1(X)$ into ${\rm GL}_N(\bC)$ and tame pure imaginary harmonic bundles on $X$, as developed in \cite{Moc06,Moc07b}.

\subsection{Tame and pure imaginary harmonic bundles}\label{sec:tame}

\begin{dfn}[Higgs bundle]\label{Higgs}
	A \emph{Higgs bundle} on a complex manifold $X$  is a pair $(E,\theta)$ where $E$ is a holomorphic vector bundle  and $\theta:E\to E\otimes \Omega^1_X$ is a holomorphic one form with value in $\End(E)$, called the \emph{Higgs field},  satisfying $\theta\wedge\theta=0$.
	%	\begin{eqnarray}\label{higgs triple}
		%	(\db +\theta)^2=0.
		%	\end{eqnarray} 
\end{dfn}
Let  $(E,\theta)$ be a Higgs bundle  over a complex manifold $X$.  
Suppose that $h$ is a smooth hermitian metric of $E$.  Denote by $\nabla_h$  the Chern connection with respect to $h$, and by $\theta^\dagger_h$  the adjoint of $\theta$ with respect to $h$.  We write $\theta^\dagger$ for $\theta^\dagger_h$ for short.   The metric $h$ is  \emph{harmonic} if the connection  $\bD_1:=\nabla_h+\theta+\theta^\dagger$
is flat, i.e., if  $\bD_1^2=0$. 
\begin{dfn}[Harmonic bundle] A \emph{harmonic bundle} on  $X$ is
	a Higgs bundle $(E,\theta)$ endowed with a  harmonic metric $h$.
\end{dfn}

Let $\overline{X}$ be a compact complex manifold, $D=\sum_{i=1}^{\ell}D_i$ be a   simple normal crossing divisor of $\overline{X}$ and $X=\overline{X}\backslash D$ be the complement of $D$.
Let $(E,\theta,h)$ be a  harmonic bundle on $X$.	 Let $p$ be any point of $D$, and $(U;z_1,\ldots,z_n)$  be a coordinate system centered at  $p$ such that $D\cap U=(z_1\cdots z_\ell=0)$ . On $U$, we have the description:
\begin{align*}
	\theta=\sum_{j=1}^{\ell}f_jd\log z_j+\sum_{k=\ell+1}^{n}f_kdz_k.
\end{align*}
%\begin{dfn}[Tame nilpotent harmonic bundle]
%The tame harmonic bundle $(E,\theta,h)$
%	is  \emph{ nilpotent } if   each $f_i(0)$ is nipotent.
%\end{dfn} 

\begin{dfn}[Tameness]\label{def:tameness}
	Let $t$ be a formal variable. For any $j=1,\dots, \ell$, the characteristic polynomial $\det (f_j-t)\in \mathcal{O}(U\backslash D)[t]$, is a polynomial in $t$   whose coefficients are holomorphic functions. 
	If those functions can be extended
	to the holomorphic functions over $U$ 
	for all $j$, then the harmonic bundle is called \emph{tame} at $p$.  A harmonic bundle is \emph{tame} if it is tame at each point.
\end{dfn} 

For a tame harmonic bundle  $(E,\theta,h)$ over $X$,  we prolong $E$ over $\overline{X}$ by a sheaf of $\cO_{\overline{X}}$-module $\diae_h$  as follows:  for any open set $U$ of $X$, 
\begin{align*} 
	\diae_h(U):=\{\sigma\in\Gamma(U\backslash D,E|_{U\backslash D})\mid |\sigma|_h\lesssim {\prod_{i=1}^{\ell}|z_i|^{-\ep}}\  \ \mbox{for all}\ \ep>0\}. 
\end{align*}  
In \cite{Moc07} Mochizuki proved that $\diae_h$ is locally free and that  $\theta$ extends to a morphism
$$
\diae_h\to \diae_h\otimes \Omega_{\overline{X}}(\log D),
$$
which we still denote by $\theta$.

\begin{dfn}[Pure imaginary, nilpotent residue]\label{def:nilpotency}
	Let $(E, h,\theta)$ be a tame harmonic bundle on $\overline{X}\backslash D$. The residue $\Res_{D_i}\theta$ induces an endomorphism of $\diae_h|_{D_i}$. Its characteristic polynomial has constant coefficients, and thus the eigenvalues are constant. We say that $(E,\theta,h)$ is \emph{pure imaginary} (resp. has  \emph{nilpotent residue}) if for each component $D_i$ of $D$, the   eigenvalues of $\Res_{D_i}\theta$ are all pure imaginary (resp. all zero).
\end{dfn}   
One can verify that \cref{def:nilpotency} does not depend on the compactification $\overline{X}$ of $\overline{X}\backslash D$.

\begin{thm}[{Mochizuki \cite[Theorem 25.21]{Moc07b}}]\label{moc}
	Let $X$ be a smooth quasi-projective variety and let $(E,\theta,h)$ be a tame pure imaginary harmonic bundle on $X$. Then the flat bundle $(E, \nabla_h+\theta+\theta^\dagger)$ is semi-simple. Conversely, if $(V,\nabla)$ is a semisimple flat bundle on $X$, then there is a tame pure imaginary harmonic bundle $(E,\theta,h)$ on $X$ so that $(E, \nabla_h+\theta+\theta^\dagger)\simeq (V,\nabla)$. Moreover, when $\nabla$ is \emph{simple}, then any such harmonic metric $h$ is unique up to positive multiplication. 
\end{thm}

\subsection{Regular filtered Higgs bundles}  
In this subsection, we recall the notions of regular filtered Higgs bundles. For more details refer to \cite{Moc06}. Let $X$ be a complex manifold with a reduced simple normal crossing divisor $D=\sum_{i=1}^{\ell}D_i$, and let $U=X-D$ be the complement of $D$. We denote the inclusion map of $U$ into $X$ by $j$.

\begin{dfn}\label{dfn:parab-higgs}
	A \emph{regular filtered Higgs bundle} $(\bm{E}_*,\theta)$ on $(X, D)$ is holomorphic vector bundle $E$ on $X-D$, together with an $\mathbb{R}^\ell$-indexed
	filtration ${}_{ \bm{a}}E$ (so-called {\em parabolic structure}) by coherent subsheaves of $j_*E$ such that
	\begin{enumerate}[leftmargin=0.7cm]
		\item $\bm{a}\in \mathbb{R}^\ell$ and ${}_{\bm{a}}E|_U=E$. 
		\item  For $1\leq i\leq \ell$, ${}_{\bm{a}+\bm{1}_i}E = {}_{\bm{a}}E\otimes \cO_X(D_i)$, where $\bm{1}_i=(0,\ldots, 0, 1, 0, \ldots, 0)$ with $1$ in the $i$-th component.
		\item $_{\bm{a}+\bm{\ep}}E = {}_{\bm{a}}E$ for any vector $\bm{\ep}=(\ep, \ldots, \ep)$ with $0<\ep\ll 1$.
		\item  The set of {\em weights} \{$\bm{a}$\ |\  $_{\bm{a}}E/_{\bm{a}-\bm{\ep}}E\not= 0$  for any vector $\bm{\ep}=(\ep, \ldots, \ep)$ with $0<\ep\ll 1$\}  is  discrete in $\mathbb{R}^\ell$.
		\item There is a $\cO_X$-linear map, so-called Higgs field, 
		$$\theta:\diae\to \Omega_X^1(\log D)\otimes \diae$$
		such that
	 $\theta\wedge \theta=0$, 
		and
		$$\theta(_{\bm{a}}E)\subseteq \Omega_X^1(\log D)\otimes {}_{\bm{a}}E.$$ 
	\end{enumerate}
\end{dfn}
 Denote $_{\bm{0}}E$ by $\diae$, where $\bm{0}=(0, \ldots, 0)$.  By the work of Borne-Vistoli the parabolic structure of a parabolic bundle  is  \emph{locally abelian}, \emph{i.e.} it admits a local frame compatible with the filtration (see e.g. \cite{IS07} and \cite{BV12}).

A natural class of regular filtered   Higgs bundles comes from prolongations of tame harmonic bundles. We first   recall some notions in \cite[\S 2.2.1]{Moc07}.  Let $E$ be a holomorphic vector bundle with a $\cC^\infty$ hermitian metric $h$ over $X-D$.
Pick any $x\in D$. Let  $(\Omega; z_1, \ldots, z_n)$  be a coordinate system of $X$ centered at  $x$. For any section $\sigma\in \Gamma(\Omega-D,E|_{\Omega-D})$, let $|\sigma|_h$ denote the norm function of $\sigma$ with respect to the metric $h$. We denote $|\sigma|_h\in \cO(\prod_{i=1}^{\ell}|z_i|^{-b_i})$ if there exists a positive number $C$ such that $|\sigma|_h\leq C\cdot\prod_{i=1}^{\ell}|z_i|^{-b_i}$. For any $\bm{b}\in \bR^\ell$, say $-\mbox{ord}(\sigma)\leq \bm{b}$ means the following:
$$
|\sigma|_h=\cO(\prod_{i=1}^{\ell}|z_i|^{-b_i-\varepsilon})
$$
for any real number  $\varepsilon>0$ and $0<|z_i|\ll1$. For any $\bm{b}$, the sheaf ${}_{\bm{b}} E$ is defined as follows: 
\begin{align*}
	\Gamma(\Omega, {}_{\bm{b}} E):=\{\sigma\in\Gamma(\Omega-D,E|_{\Omega-D})\mid -\mbox{ord}(\sigma)\leq \bm{b} \}. 
\end{align*}
The sheaf ${}_{\bm{b}} E$ is called the prolongment of $E$ by an increasing order $\bm{b}$. In particular, we use the notation ${}^\diamond E$ in the case $\bm{b}=(0,\ldots,0)$.

According to  Simpson \cite[Theorem 2]{Sim90} and Mochizuki \cite[Theorem 8.58]{Moc07}, the above prolongation gives a regular filtered Higgs bundle.
\begin{thm}[Simpson, Mochizuki] \label{thm:SM} Let $(X, D)$ be a complex manifold $X$ with a simple normal crossing divisor $D$. If $(E , \theta, h)$ is a tame harmonic bundle on $X-D$, then the corresponding filtration $_{\bm{b}}E$ defined above defines a regular filtered Higgs bundle $(\bm{E}_*, \theta)$ on $(X,D)$.  	\qed
\end{thm}

\subsection{Character variety and rigid representation}  
In this subsection, we discuss the character variety and rigid representations. Let $X$ be a quasi-projective manifold, and $G$ be a reductive algebraic group defined over a number field $k$.   
We have an affine scheme $\Hom(\pi_1(X), G)$ defined over $\overline{\bQ}$ that represents the functor
$$
S\mapsto \Hom(\pi_1(X), G(S))
$$
for any ring $S$. 
The \emph{character variety} $M_B(\pi_1(X),G):=\Hom(\pi_1(X), G)\slash\!\!\slash G$  is the GIT quotient of $\Hom(\pi_1(X), G)$ by $G$, where $G$ acts by conjugation. Note that it might be reducible while it is called variety.  Thus, $\Hom(\pi_1(X), G)\to M_B(\pi_1(X),G)$ is surjective.  For any representation $\varrho:\pi_1(X)\to G(\bC)$, we write $[\varrho]$ to denote  its image in $M_B(\pi_1(X),G)$.  We list some properties of  character varieties which will be used in this paper, and we refer the readers to the comprehensive paper \cite{Sik12} for more details.

Let $K$ be any algebraically closed field of   characteristic zero containing $k$.  A representation $\varrho:\pi_1(X)\to G(K)$ is called  \emph{reductive} if the  Zariski closure of $\varrho(\pi_1(X))$ is a   reductive group. 
In particular,  Zariski dense representations in  $\Hom(\pi_1(X), G)(K)$ are reductive for $G$ is reductive.

By \cite[Theorem 30]{Sik12}, the orbit of  any representation $\varrho$ in $ \Hom(\pi_1(X), G)(K)$ is closed if and only if  $\varrho$ is   reductive. Two reductive representations $\varrho,\varrho'$ are conjugate under $G(K)$ if and only if $[\varrho]=[\varrho']$.

A reductive representation   
$\varrho:\pi_1(X)\to G(K)$  is   \emph{rigid} if  the irreducible component  of $M_B(\pi_1(X),G)$ containing $[\varrho]$  is zero dimensional,  otherwise it is called \emph{non-rigid}.  For a  rigid reductive  representation $\varrho$, by  above arguments any continuous deformation $\varrho':\pi_1(X)\to G(K)$ of $\varrho$ which is reductive is conjugate to $\varrho$  under $G(K)$. 

\subsection{Rigid representations underlie $\bC$-VHS}
We will prove that any rigid representation underlies a $\bC$-VHS. We refer the readers to    \cite[\S 3.1.3]{Moc06} for  the definition of  $\mu_L$-stability of regular filtered Higgs bundles with respect to some ample polarization $L$.
\begin{thm}\label{thm:vhs}
	Let $X$ be a quasi-projective manifold.  Let  $\sigma:\pi_1(X)\to G(\bC)$ be a Zariski dense representation where $G$ is a reductive  algebraic group over  $\bC$. If  $\sigma$ is  rigid, then $\sigma$ underlies a $\bC$-VHS.   
\end{thm}

\begin{proof}
	Fixing a faithful representation \(G \to {\rm GL}_{N}\), \(\sigma\) induces a representation \(\varrho : \pi_{1}(X) \to {\rm GL}_{N}(\mathbb{C})\) whose Zariski closure is \(G\).
	\medskip
	
\noindent	{\em Step 1. We may assume that $\varrho$ is a simple representation.} Since the Zariski closure of $\varrho$ is $G$, it follows that $\varrho$ is semisimple and thus $\varrho=\oplus_{i=0}^{\ell}\varrho_i$ where   $\varrho_i:\pi_1(X)\to GL(V_i)$ is a simple representation with $\oplus_{i=0}^{\ell}V_i=\bC^N$. Let $G_i\subset GL(V_i)$ be the Zariski closure of $\varrho_i$ which is reductive. Then $G:=\prod_{i=0}^{\ell}G_i$, and it follows that	$M_B(X,G)=\prod_{i=1}^{\ell}M_B(X,G_i)$. Since $[\varrho]$ is an isolated point in $M_B(X,G)$, it implies that each $[\varrho_i]$ is an isolated point hence rigid.   We thus can assume that $\varrho:\pi_1(X)\to {\rm GL}_N(\bC)$ is a simple representation to prove the theorem. 
	
\noindent	{\em Step 2. Higgs field and scaling by a real number.}
	Take a projective compactification $\overline{X}$ of $X$ so that $D:=\overline{X}\backslash X$  is a simple normal crossing divisor.  Fix an ample polarization $L$ 
on $\overline{X}$. By \cref{moc}, we can find a tame purely imaginary harmonic bundle \((E, \theta, h)\) with monodromy \(\varrho\) by \cref{moc}. Then, we can let $(\mathbf{E}_*, \theta)$ be the associated regular filtered Higgs bundle  
of $(E,\theta,h)$  on the log pair $(\overline{X},D)$ defined in \cref{thm:SM}.  This filtered Higgs bundle is $\mu_L$-stable  regular filtered  Higgs bundle with trivial characteristic numbers.   
	
	Let $t$ be a real number in $]0,1]$.  Then   $(\mathbf{E}_*, t\theta)$ is also $\mu_L$-stable   with trivial characteristic numbers. By the Kobayashi-Hitchin correspondence \cite[Theorem 9.4]{Moc06}, there is a pluriharmonic metric $h_t$ for $(E,t\theta)$ which is adapted to the parabolic structure of $(\mathbf{E}_*, t\theta)$. In particular, $(E,t\theta,h_t)$ is a tame harmonic bundle. It is moreover pure imaginary since $t$ is real. Hence the associated monodromy representation $\varrho_t:\pi_1(X)\to {\rm GL}_N(\bC)$ of the flat connection $\bD_t:=\nabla_{h_t}+t\theta+(t\theta)_{h_t}^\dagger$ is  semisimple by \cref{moc}. Here $\nabla_t$ is the Chern connection of $(E,h_t)$.   %It is moreover simple since the corresponding  filtered bundle $(\mathbf{E}_*, t\theta)$ is stable.  
	By \cite[Lemma 10.10]{Moc06}, the Zariski closure of $\varrho_t$ coincides with that of $\varrho$, hence is also $G$.  
Hence $[\varrho_t]\in M_B(\pi_1(X),G)(\bC)$.  %Therefore, by \cref{moc}, the metrics \(h_{t}\) are unique up to positive multiplication. 

	%{\em Step 2. Uniqueness of harmonic metrics.} Since $G$ is reductive, $\varrho_t$ is thus semisimple for any $t\in ]0,1]$.    $\varrho_{t}=\varrho_1\oplus \varrho_2$ with $\varrho_i:\pi_1(X)\to {\rm GL}_{m_i}(\bC)$, let $G_i$ be the Zariski closure of $\varrho_i$ which are reductive. Then $G=G_1\oplus G_2$ which contradicts with the assumption that $G$ is simple. Therefore, by \cref{moc}, the metrics \(h_{t}\) are unique up to positive multiplication. 
	
\noindent	{\em Step 3. Existence of isometries between the deformations.} By the proofs of \cite[Theorem 10.1 \&  Lemma 10.11]{Moc06}, we know that the map
	\begin{align*}
		]0,1]&\to M_B(\pi_1(X),G)(\bC)\\
		t&\mapsto [\varrho_t]
	\end{align*}
	is \emph{continuous}, where we endow the complex affine variety $M_B(\pi_1(X),G)(\bC)$ with the analytic topology. Since $ [\varrho] \in M_B(X,G)(\bC)$ is isolated and $\varrho_t$ is reductive, it follows that $\varrho_t$ is conjugate to $\varrho$. This implies that $\varrho_t$ is also simple for each $t\in ]0,1]$.  Let  $V$ be the underlying smooth vector bundle of $E$.  Fix some $t\in ]0,1[$. One can thus construct a  smooth automorphism $\varphi:V\to V$ so that $\bD_t=\varphi^*\bD_1:=\varphi^{-1}\bD_1\varphi$.  Hence  $\varphi^*h$  defined by
	$
	\varphi^*h(u,v):=h(\varphi(u),\varphi(v)) 
	$  is the harmonic metric for the flat bundle $(V, \varphi^*\bD_1)=(V, \bD_t)$.   By the unicity of harmonic metric of  simple flat bundle in  \cref{moc},  there is a constant $c>0$ so that $c\varphi^*h=h_t$.  Let us  replace the automorphism $\varphi$ of $V$ by $\sqrt{c}\varphi$ so that   we will have $\varphi^*h=h_t$.
	\medskip
	
\noindent	{\em Step 4. \(\theta\) is nilpotent.} Since the decomposition of $\bD_t$ with respect to the metric $h_t$ into a sum of unitary connection   and self-adjoint operator of $V$ is unique, it implies that 
	\begin{align*}
		\nabla_{h_t}&=\nabla_{\varphi^*h}=\varphi^*\nabla_h\\
		t\theta+ (t\theta)^\dagger_{h_t}&=\varphi^*(\theta+\theta_{h}^\dagger).
	\end{align*} 
	Since $E=(V, \nabla_h^{0,1})=(V,\nabla_{h_t}^{(0,1)})$, the first equality means that $\varphi$ is a holomorphic isomorphism of $E$. The second one implies that  
	\begin{align}\label{eq:same}
		t\theta=\varphi^*\theta:=\varphi^{-1}\theta \varphi.
	\end{align} 
	Let $T$ be a formal variable. Consider the characteristic polynomial 
	$$
	\det(T-\varphi^{-1}\theta \varphi)=\det(T-\theta)=T^N+a_1T^{N-1}+\cdots+a_0
	$$
	with $a_i\in H^0(\overline{X}, \Omega^i_{\overline{X}}(\log D))$. Note that
	$$
	\det(T-t\theta)=T^N+ta_1T^{N-1}+\cdots+t^Na_0.
	$$
	By \eqref{eq:same} and $t\in (0,1)$, it follows that $a_i\equiv 0$ for each $i=1,\ldots,N$. Hence $\theta$ is \emph{nilpotent}.
	\medskip
	
\noindent	{\em Step 5. Scaling of $\theta$ by an element of \({\rm U}(1)\).} By the previous step,  $(E,\theta,h)$ has nilpotent residues in the sense of \cref{def:nilpotency}.  Then for any $\lambda\in  {\rm U}(1)$, $(E,\lambda\theta,h)$ also has nilpotent residues, hence is tame pure imaginary 
harmonic bundle.  We apply \cite[Lemmas 10.9 \&  10.10]{Moc06} to conclude that for the  monodromy representation $\varrho_\lambda$ of the flat connection  $\bD_\lambda:=\nabla_h+\lambda\theta+\overline{\lambda}\theta^\dagger_h$, its Zariski closure is also $G$, which implies that $\varrho_\lambda$ is reductive.  Hence $[\varrho_\lambda]\in M_B(\pi_1(X),G)(\bC)$.   
As $\varrho_\lambda$ is a continuous deformation of $\varrho$ in $\Hom(\pi_1(X), G)(\bC)$ and $ [\varrho] \in M_B(X,G)(\bC)$ is isolated, it follows that $\varrho_\lambda$ is conjugate to $\varrho$. 
	\medskip
	
\noindent	{\em Step 6. End of proof.} The rest of the proof is exactly the same as \cite[Proposition 4.8]{BDDM}; we provide it for completeness sake.  Now fix $\lambda\in {\rm U}(1)$ which is not root of unity. By the same argument as above, there is a  a  smooth automorphism $\phi:V\to V$ so that $\bD_\lambda=\phi^*\bD_1:=\phi^{-1}\bD_1\phi$ and $\phi^*h=h$.  Moreover, 
	\begin{align*}
		\nabla_{h}&=\nabla_{\phi^*h}=\phi^*\nabla_h\\
		\lambda\theta+ \overline{\lambda}\theta^\dagger_{h}&=\phi^*(\theta+\theta_{h}^\dagger).
	\end{align*} 
	In other words, $\phi:(E,h)\to (E,h)$ is a holomorphic automorphism which is moreover an isometry.  Moreover, $\phi^*\theta=\lambda\theta$. 
	Consider the prolongation $\diae_h$  over $\overline{X}$ via norm growth defined in \cref{sec:tame}. Since $\phi: (E,h)\to (E,h)$ is   an isometry, it thus extends to a holomorphic isomorphism $\diae_{h}\to \diae_h$, which we still denote by $\phi$. Recall that $(E,\theta,h)$ and $(E,\lambda\theta,h)$ 
is tame,  we thus have the following commutative diagram
	\begin{equation*}
		\begin{tikzcd}
			\diae_{h} \arrow[r, "\lambda\theta"] \arrow[d, "\phi"] & \diae_{h}\otimes \Omega_{\overline{X}}(\log D) \arrow[d, "\phi\otimes {\rm id}"] \\
			\diae_{h} \arrow[r, "\theta"]  & \diae_{h}\otimes \Omega_{\overline{X}}(\log D).
		\end{tikzcd}
	\end{equation*} 
Let $T$ be a formal variable. Consider the polynomial  $\det (\phi-T)=0$ of $T$. Since its coefficients are holomorphic functions defined over  $\overline{X}$, they are all constant.  Let $\eta$ be an eigenvalue of $\det (\phi-T)=0$.  Consider the generalized eigenspace $\diae_{h,\eta}$ defined by $\ker (\phi-\eta)^\ell=0$ for some sufficiently big $\ell$. One can check that $\theta:  \diae_{h,\eta} \to \diae_{h,\lambda^{-1}\eta}\otimes \Omega_{\overline{X}}(\log D)$. Since $\lambda$ is not root of unity,  the eigenvalues of $\phi$ break up into a finite number of  chains of the form $\lambda^{i}\eta,\ldots, \lambda^{-j}\eta$  so that $\lambda^{i+1}\eta$ and $\lambda^{-j-1}\eta$
	are not eigenvalues of $\phi$.  Therefore, there is decomposition $\diae_{h}=\oplus_{i=0}^{m} \diae_i$ so that 
	$$
	\theta: \diae_i\to \diae_{i+1}\otimes \Omega_{\overline{X}}(\log D).
	$$
	In the language of \cite[Definition 2.10]{Den22} $(\diae_{h}=\oplus_{i=0}^{m} E_i,\theta)$ is a \emph{system of log Hodge bundles}. One can apply \cite[Proposition 2.12]{Den22} to show that the harmonic metric  $h$ is moreover a \emph{Hodge metric} in the sense of \cite[Definition 2.11]{Den22}, \emph{i.e.}  $h(u,v)=0$ if $u\in E_i|_{X}$ and $v\in E_j|_{X}$ with $i\neq j$. Hence by \cite[\S 8]{Sim88}, $(E,\theta,h)$ corresponds to a complex variation of Hodge structures in $X$.     This proves the theorem. 
\end{proof}
\begin{rem}
The above theorem has been known to us for the following cases:	when $X$ is projective, it is proved by Simpson \cite{Sim92}; when $\sigma$ has quasi-unipotent monodromies at infinity and $G={\rm SL}_n$, it is proved by Corlette-Simpson \cite{CS08}; when $G={\rm GL}_N$ and the corresponding harmonic bundle of $\sigma$ has nilpotent residues at infinity, it is proved in \cite{BDDM}. In Steps 1-4 of the above proof we utilise Mochizuki's work to prove that $\theta$ is nilpotent. 
\end{rem}

\section{Proof of  Theorem A} \label{sec:proof1}
In this section, we  prove our first main result, \cref{main2}, of this paper.
The outline of the proof is as follows: The proof relies on dividing the argument into two essential cases depending on whether the representation $\varrho:\pi_1(X)\to G(\bC)$ is rigid or not.
This strategy is now standard as originated from \cite{Zuo96}.

\noindent
Case 1: Non-rigid. When the representation is non-rigid, we first construct an unbounded representation $\varrho':\pi_1(X)\to G(F)$, where $F$ is some non-archimedean local field of zero characteristic (cf. \cref{lem:20220819}), and then apply \cref{main6}.

\noindent
Case 2: Rigid. When $\varrho$ is rigid, we may assume that it is defined over some number field $L$. If there exists some non-archimedean place $\nu$ of $L$ such that the representation is unbounded with respect to $\nu$, we apply \cref{main6}. Otherwise, the representation becomes a complex direct factor of a $\bZ$-variation of Hodge structures, by \cref{thm:vhs}. 
We then apply \cref{thm:PicardVHS}.

In our proof, this strategy is first carried out under the stronger assumption that $G$ is almost simple, and then we reduce the general semisimple case to the almost simple case.
In addition, we need the stability of the existence of a Zariski dense and big representation $\varrho:\pi_1(X)\to \mathrm{GL}_n(\bC)$ under the Galois conjugates (cf. \cref{thm:conjugate}).

\subsection{Existence of unbounded representations}\label{sec:20251108}
 We first prove a theorem on constructing unbounded representation in almost simple algebraic groups over non-archimedean local field. This refines a previous result in \cite[\S 4.2]{Yam10}.

\begin{proposition}\label{lem:20220819}
	Let $X$ be a smooth quasi-projective variety and  let $G$ be an almost simple algebraic group defined over $\bC$. 
Assume that there is a Zariski dense, non-rigid representation $\varrho:\pi_1(X)\to G(\bC)$, then there is a Zariski dense, unbounded representation $\varrho':\pi_1(X)\to G(F)$ where $F$ is some non-archimedean local field of zero characteristic. 
	If moreover $\varrho$ is big, then $\varrho'$ is taken to be big.%\footnote{Shall we prove a more general theorem, including the case when the monodromy of $\varrho$ at infinity is quasi-unipotent (then the conclusion is $\varrho'$ also has quasi-unipotent monodromy at infinity)?  }
\end{proposition}

Before prove this, we prepare two lemmas.

\begin{lem}
Let $G$ be an almost simple algebraic group over a field $K$ of characteristic zero. Then its Lie algebra $\mathrm{Lie}(G)$ is simple. 
\end{lem}
\begin{proof} 
		Since $G$ is almost simple, $\mathrm{Lie}(G)$ is semi-simple thanks to the assumption $\mathrm{char}(K)=0$ (cf. \cite[Prop 4.1.]{Mil17}).
		We may decompose as $\mathrm{Lie}(G)=\mathfrak{g}_1\oplus \cdots\oplus \mathfrak{g}_r$, where $\mathfrak{g}_i$ is a simple ideal of $\mathrm{Lie}(G)$.
		We need to show $r=1$.
		So assume contrary that $r>1$.
		Then $\mathfrak{g}_1\subset \mathrm{Lie}(G)$ is neither trivial nor $\mathrm{Lie}(G)$ itself.
Set $H=C_G(\mathfrak{g}_2\oplus \cdots\oplus \mathfrak{g}_r)$, which is a subgroup of $G$ defined by 
		$$
	 C_G(\mathfrak{g}_2\oplus \cdots\oplus \mathfrak{g}_r)=(R \leadsto g \in G(R) \mid g x=x \text { for all } x \in \mathfrak{g}_2\oplus \cdots\oplus \mathfrak{g}_r(R))
		$$
		for all $k$-algebra $R$ 
		(cf. \cite[Proposition 3.40]{Mil17}). It means that  $H$ acts trivially on $\mathfrak{g}_2\oplus \cdots\oplus \mathfrak{g}_r$.
		Then $\mathrm{Lie}(H)=\mathfrak{g}_1$. 
		Let $H^0\subset H$ be the identity component of $H$.
		Then $\mathrm{Lie}(H^0)=\mathrm{Lie}(H)=\mathfrak{g}_1$.
		Hence $\mathrm{Lie}(H^0)\subset \mathrm{Lie}(G)$ is an ideal.
		Hence $H^0\subset G$ is normal (cf. \cite[Thm 3.31]{Mil17}), which is neither trivial nor $G$ itself.
		Since $G$ is almost simple, this is a contradiction.
		Thus $r=1$ and $\mathrm{Lie}(G)$ is simple. 
\end{proof}
\begin{lem}\label{lem:20220827}
Let $G$ be an almost simple algebraic group defined over a field $K$ of characteristic zero.
Let $\mathrm{Ad}:G\to \mathrm{Aut}(\mathrm{Lie}(G))$  be the adjoint representation.
Let $W\subset \mathrm{Lie}(G)$ be a $K$-linear subspace which is $G$-invariant.
Then either $W=\mathrm{Lie}(G)$ or $W=\{ 0\}$. 
\end{lem}

\begin{proof}
Write $\kg:=\mathrm{Lie}(G)$. 	Consider the adjoint representation $G \rightarrow \mathrm{GL}_{\kg}$. For any $K$-subspace $P$  of $V$,  the functor
	$$
	R \rightsquigarrow\left\{g \in G(R) \mid g P_R=P_R\right\}\  \mbox{for all}\ K\mbox{-algebra } R
	$$
	is a subgroup of $G$, denoted $G_P$. Then by our assumption $G_W=G$.  By \cite[Proposition 3.38]{Mil17}, it follows that $\mathrm{Lie}(G_W)=\mathrm{Lie}(G)_W$,
	where 
	$$
	\mathrm{Lie}(G)_W:=\{x\in \kg\mid  \mathrm{ad}(x)(W)\subset W \}. 
	$$   
In other words, $[x,y]\in W$ for all $x\in\kg$, $y\in W$, i.e., $W\subset  \kg$ is an ideal.
Note that $\kg$ is a simple Lie algebra. 
Hence $W=\kg$ or $W=\{ 0\}$. 
\end{proof}

\medskip

\begin{proof}[Proof of \Cref{lem:20220819}] 
Since $G$ is almost simple,  by the classification of almost simple linear algebraic groups defined over $\bC$,   $G$ is isogenous to exactly one of the following:  $A_n, B_n, C_n, D_n, E_6, E_7, E_8, F_4, G_2$. Therefore, $G$ is defined over some number field $k\subset \bar{\bQ}$. Moreover, it is absolutely almost simple, i.e. for any field extension $L/k$, the base change $G_L$  is also an almost simple algebraic group over $L$.  

Since $\pi _1(X)$ is finitely presented, there exists an affine scheme $R$ defined over $k$ such that
$$R (L)=\Hom (\pi _1(X),G(L))$$
for every field extension $L/k$.
This space is defined as follows: We choose generators $\gamma _1,\ldots ,\gamma _\ell$ for $\pi _1(X)$. Let $\mathcal R$ be the set of relations among the generators $\gamma _i$. Then 
$$R \subset \underbrace{G\times \cdots \times G}_{\ell \ \text{times}}$$
is the closed subscheme defined by the equations $r(m_1,\ldots ,m_\ell)=1$ for $r\in \mathcal R$. A representation $\tau :\pi _1(X)\to G(L)$ corresponds to the point $(m_1,\ldots ,m_\ell)\in R (L)$ with $m_i=\tau (\gamma _i)$. 
Note that $R$ is an affine scheme, since it is a closed subscheme of an affine variety.

\medskip

\begin{claim}\label{claim:ZD}
There exists a Zariski closed subset $E\subset R$ defined over $k$ with the following property:
Let $L$ be a field extension of $k$ and $\tau :\pi _1(X)\to G(L)$ be a representation such that $\tau(\pi_1(X))$ is infinite.
Then $\tau(\pi_1(X))$ is not Zariski dense in $G({L})$ if and only if the corresponding point $[\tau]\in R(L)$ satisfies $[\tau]\in E(L)$. 
\end{claim} 

\begin{proof}[Proof of \cref{claim:ZD}]
A stronger result is proved in \cite[Proposition 8.2]{AB}.
Here we give a proof for the sake of completeness.
%Since $G$ is almost simple, we can consider $G$ as an algebraic group defined over the number field $k$.
 Consider  the adjoint action $G\curvearrowright \mathrm{Lie}(G)$ which is defined over $k$. 
This action induces the action $G\curvearrowright \mathrm{Gr}_d(\mathrm{Lie}(G)$ over $k$, where $0< d<\dim G$.
Here $\mathrm{Gr}_d(\mathrm{Lie}(G)$ is the grassmannian variety, which we denote by $X_d$. Recall that $G_{L}$ is  almost simple for the extension $L$ of $k$.   
We remark that the action $G(L)\curvearrowright X_d(L)$ has no fixed point. 
Indeed if there exists a fixed point $[P]\in X_d(L)$, where $P\subsetneqq \mathrm{Lie}(G)_L$ is a $d$-dimensional subspace defined over $L$, we get a $G(L)$ invariant subspace $P\subset \mathrm{Lie}(G)_L=\mathrm{Lie}(G_L)$ which is impossible by Lemma \ref{lem:20220827}. 
On the other hand, if $H\subset G_L$ is an algebraic subgroup of dimension $d$, then $H(L)$ fixes the point $[\mathrm{Lie}(H)]\in X_d(L)$. 

Let $S=\{\gamma_1,\ldots,\gamma_l\}\subset \pi_1(X)$ be a finite subset which generates $\pi_1(X)$.
We define a Zariski closed subset $W_{d}\subset X_d\times G^{S}$ by
$$W_{d}=\{ (x,h_1,\ldots,h_l); h_ix=x, \forall i\in\{1,\ldots,l\}\}.
$$
Then if a representation $\tau :\pi _1(X)\to G(L)$ satisfies $\tau(\pi_1(X))\subset H(L)$ for some algebraic subgroup of dimension $d$, where $0<d<\dim G$, then we have $([\mathrm{Lie(H)}],\tau(\gamma_1),\ldots,\tau(\gamma_l))\in W_{d}(L)$.

Let $p:X_d\times G^{S}\to G^{S}$ be the projection.
Then $p$ is a proper map, for $X_d$ is complete.
Hence $p(W_{d})\subset G^{S}$ is a Zariski closed subset defined over $k$.
We set $E_{d}=p(W_{d})\cap R$, which is a Zariski closed subset of $R$.
Hence if a representation $\tau :\pi _1(X)\to G(L)$ satisfies $\tau(\pi_1(X))\subset H(L)$ for some algebraic subgroup of dimension $d$, then the corresponding point satisfies $[\tau]\in E_{d}(L)$.

Now let $\tau :\pi _1(X)\to G(L)$ be a representation such that $\tau(\pi_1(X))$ is infinite.
Assume first that $\tau(\pi_1(X))$ is not Zariski dense in $G(L)$.
Let $H\subset G$ be the Zariski closure of $\tau(\pi_1(X))\subset G(L)$.
Then $H$ is defined over $L$.
Since $\tau(\pi_1(X))$ is infinite, we have $d=\dim H>0$.
Then by the above consideration, we have $[\tau]\in E_{d}(L)$.  
Next suppose $[\tau]\in E_{d}(L)$ for some $d$ with $0<d<\dim G$.
Then by $E_{d}(L)\subset p(W_{d}(L))$, there exists a $d$-dimensional $L$-subspace $P\subset \mathrm{Lie}(G)_L$  such that $([P],\tau(\gamma_1),\ldots,\tau(\gamma_l))\in W_{d}(L)$.
Hence $\tau(\pi_1(X))$ fixes the point $[P]\in X_d(L)$.
If $\tau(\pi_1(X))\subset G(L)$ is Zariski dense, then $G(L)$ also fixes $[P]\in X_d(L)$.
This is impossible as we see above.
Hence $\tau(\pi_1(X))\subset G(L)$ is not Zariski dense.

We set $E=E_1\cup \cdots\cup E_{\dim G-1}$.
This concludes the proof of the claim.
\end{proof} 

Now we return to the proof of the proposition. 
The group $G$ acts on $R$ by simultaneous conjugation.
Put $M=R // G$, and let $p:R\to M$ be the quotient map which is surjective.
Then $M$ is an affine scheme defined over $k$.
Let $\varrho\in R (\mathbb C)$ be the point which correspond to the Zariski dense representation $\varrho :\pi _1(X)\to G(\mathbb C)$.
\par

Let $W$ be the set of words in $x_1,\ldots,x_\ell$.
Given $w\in W$, we define a closed subscheme $Z_{w}\subset R$ defined by $Z_w=\{ (m_1,\ldots,m_\ell)\in R;\ w(m_1,\ldots,m_\ell)=1\}$.
Then $Z_w$ is defined over $k$. 

Let $M'$ be the irreducible component of $M_{\bar{\bQ}}$ such that $[\varrho]\in M'(\bC)$. 
Since $\varrho $ is non-rigid,   one has $\dim M'>0$.    Let $R'$ be the irreducible component   of $R_{\bar{\bQ}}$ containing $\varrho$. Since  $p:R_{\bar{\bQ}}\to M_{\bar{\bQ}}$  is surjective,   $p|_{R'}: R'\to M'$  is surjective.  Then $R'$ and $M'$ are defined over some finite extension  $K$  of $k$.  Note that $R'$ and $M'$ are geometrically irreducible.  
Let $\eta\in R'$ be the schematic generic point. 
Then $K(\eta)$ is a finitely generated over $K$.
Let $\mathfrak{K}$ be a field extension of $K$, which is finitely generated over $K$, such that an embedding $K(\eta)\hookrightarrow \mathfrak{K}$ exists and that $\varrho$ is defined over $\mathfrak{K}$, i.e., $\varrho:\pi_1(X)\to G(\mathfrak{K})$.

Let $p$ be a prime number and let $\mathbb Q_p$ be the completion.
Then since the transcendental degree of $\mathbb Q_p$ over $\mathbb Q$ is infinite and $\mathfrak{K}$ is finitely generated over $\mathbb Q$, there exists a finite extension $L/\mathbb Q_p$ such that an embedding $\mathfrak{K}\hookrightarrow L$ exists.
Then we have $[\varrho ]\in R (L)$. Recall that $G_L$ is   almost simple as an algebraic group defined over $L$. 
We remark that  $\varrho:\pi_1(X)\to G(L)$ is  Zariski dense.  %in $G(\overline{L})$.
%Indeed if not, there exists a proper Zariski closed subset $V\subsetneqq G$ define over $\overline{L}$ such that $\varrho(\pi_1(X))\subset V(\overline{L})$.
%Then $V$ is defined over some subfield of $\overline{L}$ that is finitely generated over $\mathfrak{K}$.
%Then we may extend the immersion $\mathfrak{K}\hookrightarrow \mathbb C$ so that this defining field of $V$ is considered as a subfield of $\mathbb C$.
%Then we may consider that $V$ is defined over $V$ and the image of the original $\varrho:\pi_1(X)\to G(\mathbb C)$ is contained in $V(\mathbb C)$.
%This contradicts to the assumption that $\varrho:\pi_1(X)\to G(\mathbb C)$ is Zariski dense.
%Hence $\varrho:\pi_1(X)\to G(L)$ is Zariski dense in $G(\overline{L})$.
Thus by \cref{claim:ZD}, we have $[\varrho]\not\in E(L)$.
In particular, we have $E\subsetneqq R$.

We define $W_o\subset W$ by $w\in W_o$ if and only if $Z_w(\bar{L})\cap R'(\bar{L})\subsetneqq R'(\bar{L})$. 
We note that $W$, hence $W_o$ is countable. 
We may take $\eta_0\in R'(L)$ such that the corresponding map $\eta_0:\mathrm{Spec}\ L\to R'$ has image $\eta\in R'$.
Then $\eta_0\not\in Z_w(\bar{L})$ for all $w\in W_o$ and $\eta_0\not\in E(\bar{L})$.

\par
Since  $\dim M'>0$, there exists a morphism of $L$-scheme  $\psi :M'\to \mathbb A^1$ such that the image $\psi (M')$ is Zariski dense in $\mathbb A^1$.
Recall that  $p|_{R'}:R'\to M'$ is surjective.  Hence  the image $\psi \circ p(R')$ is Zariski dense in $\mathbb A^1$, and 
there exists an affine curve $C\subset R'$ defined over $L$ such that the restriction $\psi \circ p|_C:C\to \mathbb A^1$ is generically finite and $\eta_0\in C(L)$.
Since $R'$ is geometrically irreducible,  one has $C\not\subset Z_w$ for all $w\in W_o$ and $C\not\subset E$.
We may take a Zariski open subset $U\subset \mathbb A^1$ such that $\psi \circ p|_C$ is finite over $U$.
Let $x\in U(L)$ be a point, and let $y\in C(\bar{L})$ be a point over $x$.
Then $y$ is defined over some extension of $L$ whose extension degree is bounded by the degree of $\psi \circ p|_C:C\to \mathbb A^1$.
Note that there are only finitely many such field extensions.
Hence there exists a finite extension $F/L$ such that the points over $U(L)$ are all contained in $C(F)$.
Since $U(L)\subset \mathbb A^1(F)$ is unbounded, the image $\psi \circ p(C(F))\subset \mathbb A^1(F)$ is unbounded.
\par
Let $R_0\subset R(F)$ be the subset whose points correspond to $p$-bounded representations.
Let $M_0\subset M(F)$ be the image of $R_0$ under the quotient $p:R\to M$.
Then by \cite[Lemma 4.2]{Yam10}, $M_0$ is compact. 
Hence $\psi (M_0)$ is compact.
In particular it is bounded.
On the other hand, $\psi \circ p (C(F))\subset \mathbb A^1(F)$ is unbounded.
Hence $\psi \circ p (C(F))\backslash \psi (M_0)$ is an uncountable set.
Hence $C(F)\backslash R_0$ is an uncountable set.
Thus we may take $y\in C(F)$ such that $y\not\in R_0$ and $y\not\in Z_w(F)$ for all $w\in W_0$ and $y\not\in E(F)$.
Here we note that $C(F)\cap Z_w$ and $C(F)\cap E$ is a finite set, hence $\cup_{w\in W_o} (C(F)\cap Z_w)\cup (C(F)\cap E)$ is a countable set.
Let $\varrho':\pi _1(X)\to G(F)$ be the representation which corresponds to $y\in C(F)\subset R(F)$, i.e., $[\varrho']=y$.
Then $\varrho'$ is a $p$-unbounded representation such that $\varrho'$ is Zariski dense or has finite image  by \cref{claim:ZD}.

Next we show $\mathrm{ker}(\varrho')\subset \mathrm{ker}(\varrho)$.
Indeed, every element in $\mathrm{ker}(\varrho')$ is written as $w(\gamma_1,\ldots,\gamma_\ell)$ for some $w\in W$.
Then $[\varrho']\in Z_w$.
Hence $w\not\in W_0$.
Hence $R(F)\subset Z_w$.
In particular $[\varrho]\in Z_w$, hence $w(\gamma_1,\ldots,\gamma_\ell)\in \mathrm{ker}(\varrho)$.
This shows $\mathrm{ker}(\varrho')\subset \mathrm{ker}(\varrho)$.
Hence if $\varrho$ is big, then $\varrho'$ is also big.
In particular, $\varrho'$ has infinite image, hence Zariski dense image.
This completes the proof of the proposition.
\end{proof}

\subsection[Galois conjugate varieties]{Constructing linear representations of fundamental groups of Galois conjugate varieties}
Let $X$ be a complex smooth projective variety.  Recall that given any automorphism  $\sigma\in {\rm Aut}(\bC/\bQ)$, we can form the conjugate variety ${X}^\sigma$ defined as the complex variety ${X} \times_\sigma\spec \bC$, that is, by the cartesian diagram
\begin{equation*}
	\begin{tikzcd}
		{X}^\sigma\arrow[r,"\sigma^{-1}"]\arrow[d] & {X}\arrow[d]\\
		\spec\bC \arrow[r, "\sigma^*"]&\spec \bC 
	\end{tikzcd}
\end{equation*}
It is a smooth  projective variety. If $X$ is defined by homogeneous polynomials $P_1, \ldots, P_r$ in some projective space, then $X^\sigma$ is defined by the conjugates of the $P_i$ by $\sigma$. In this case, the morphism from $X^\sigma$ to $X$ in the cartesian diagram sends the closed point with coordinates $\left(x_0: \ldots: x_n\right)$ to the closed point with homogeneous coordinates $\left(\sigma^{-1}\left(x_0\right): \ldots: \sigma^{-1}\left(x_n\right)\right)$, which allows us to denote it by $\sigma^{-1}$.

The morphism $\sigma^{-1}: {X}^\sigma \rightarrow {X}$ is an isomorphism of abstract schemes, but it is not a morphism of complex varieties. It is important to note that, in general, the fundamental groups of the complex variety $X$ and $X^\sigma$ may be quite different, as demonstrated by the famous examples of Serre \cite{Ser64}. Despite this, their algebraic fundamental groups, which are the profinite completions of the topological fundamental groups, are canonically isomorphic.  

The following proposition plays a crucial role in the proof of \cref{main2}:  
\begin{proposition}\label{thm:conjugate}
	Let $X$ be a   smooth quasi-projective   variety and let $\rho:\pi_1(X)\to\mathrm{GL}_n(\mathbb C)$ be a representation.
	Let $\sigma\in\mathrm{Aut}(\mathbb C/\mathbb Q)$.
	Then there exists a representation $\tau:\pi_1(X^{\sigma})\to\mathrm{GL}_n(\mathbb C)$ such that the Zariki closures satisfy 
	\begin{equation}\label{eqn:202303111}
		\overline{\rho(\pi_1(X))}^{\mathrm{Zar}}=\overline{\tau(\pi_1(X^{\sigma}))}^{\mathrm{Zar}}.
	\end{equation}
	More precisely, $\tau$ satisfies the following property:
	If $Y\to X$ is a morphism from a smooth quasi-projective variety $Y$, we have
	\begin{equation}\label{eqn:202303115}
		\overline{\rho(\mathrm{Im}[\pi_1(Y)\to\pi_1(X)])}^{\mathrm{Zar}}=\overline{\tau(\mathrm{Im}[\pi_1(Y^{\sigma})\to\pi_1(X^{\sigma})])}^{\mathrm{Zar}}.
	\end{equation}
	In particular, if $\rho$ is big (resp. large), then $\tau$ is big (resp. large).
\end{proposition} 
\begin{proof}
	Let $\gamma_1,\ldots,\gamma_k\in \pi_1(X)$ be a system of generators of $\pi_1(X)$ (as monoid).
	For $l=1,\ldots,k$, we set $\rho(\gamma_l)=(a_{ij}(\gamma_l))_{1\leq i,j,\leq n}\in \mathrm{GL}_n(\mathbb C)$.
	Let $S\subset \mathbb C$ be a finite subset defined by $S=\{a_{ij}(\gamma_l); 1\leq i,j\leq n,1\leq l\leq k\}$.
	Let $\mathbb Q(S)\subset \mathbb C$ be the subfield generated by $S$ over $\mathbb Q$.
	Then $\mathbb Q(S)$ is a finitely generated field over $\mathbb Q$.
	By Cassels' p-adic embedding theorem (cf. \cite{Cas76})
	there exist a prime number $p\in\mathbb N$ and an embedding $\iota:\mathbb Q(S)\hookrightarrow \mathbb Q_p$ such that 
	\begin{equation}\label{eqn:20230311}
		|\iota(a)|_p=1
	\end{equation}
	for all $a\in S$.
	
	We claim that there exists an isomorphism 
	$$\mu:\overline{\mathbb Q_p}\overset{\sim}{\to} \mathbb C$$
	such that $\mu\circ\iota(a)=a$ for all $a\in S$, where $\overline{\mathbb Q_p}$ is an algebraic closure of $\mathbb Q_p$.
	We prove this.
	
	Let $B\subset \mathbb C$ be a transcendence basis of $\mathbb C/\mathbb Q(S)$ 
	Let $B'\subset \overline{\mathbb Q_p}$ be a transcendence basis of $\overline{\mathbb Q_p}/ \iota(\mathbb Q(S))$.
	Then for the cardinality, we have $\#B=\#B'=\#\mathbb R$.
	Hence there exists an extension $\iota_1: \mathbb Q(S)(B)\hookrightarrow \overline{\mathbb Q_p}$ of $\iota$ such that $\iota_1(\mathbb Q(S)(B))=\iota(\mathbb Q(S))(B')$.
	Since $\mathbb C=\overline{\mathbb Q(S)(B)}$ and $\overline{\mathbb Q_p}=\overline{\iota(\mathbb Q(S))(B')}$, $\iota_1$ extends to an isomorphism $\iota_2:\mathbb C\to \overline{\mathbb Q_p}$.
 	Then we set $\mu=\iota_2^{-1}$, which is an isomorphism such that $\mu\circ\iota(a)=a$ for all $a\in S$.
	By this isomorphism $\mu$, we consider
	$$
	\mathrm{GL}_n(\mathbb Z_p)\subset \mathrm{GL}_n(\mathbb Q_p)\subset \mathrm{GL}_n(\overline{\mathbb Q_p})=\mathrm{GL}_n(\mathbb C).
	$$
	
	By \eqref{eqn:20230311}, we have $\rho(\pi_1(X))\subset \mathrm{GL}_n(\mathbb Z_p)$.
	Thus we may consider $\rho$ as $\rho:\pi_1(X)\to \mathrm{GL}_n(\mathbb Z_p)$ so that
	\begin{equation*}
		\rho(\gamma_l)=(\iota(a_{ij}(\gamma_l)))_{1\leq i,j,\leq n}\in \mathrm{GL}_n(\mathbb Z_p).
	\end{equation*}
	Since $\mathrm{GL}_n(\mathbb Z_p)$ is a profinite group, $\rho$ extends to $\widehat{\rho}:\widehat{\pi_1(X)}\to\mathrm{GL}_n(\mathbb Z_p)$, where $\widehat{\pi_1(X)}$ is the profinite completion of $\pi_1(X)$.
	Note that $\widehat{\pi_1(X)}$ is the etale fundamental group of $X$ and the etale fundamental groups of $X$ and $X^{\sigma}$ are naturally isomorphic:
	\begin{equation}\label{eqn:202303112}
		\widehat{\pi_1(X)}\simeq \widehat{\pi_1(X^{\sigma})}.
	\end{equation}
	Hence we define $\tau:\pi_1(X^{\sigma})\to \mathrm{GL}_n(\overline{\mathbb Q_p})$ by the composite of the followings:
	$$
	\pi_1(X^{\sigma})\to \widehat{\pi_1(X^{\sigma})}\simeq \widehat{\pi_1(X)}\overset{\widehat{\rho}}{\to} \mathrm{GL}_n(\mathbb Z_p)\hookrightarrow \mathrm{GL}_n(\overline{\mathbb Q_p}).
	$$
	
	Next we prove \eqref{eqn:202303111}.
	We first prove 
	\begin{equation}\label{eqn:20230313}
		\widehat{\rho}(\widehat{\pi_1(X)})\subset \overline{\rho(\pi_1(X))}^{\mathrm{Zar}}\subset \mathrm{GL}_n(\overline{\mathbb Q_p}).
	\end{equation}
	Note that $\overline{\rho(\pi_1(X))}^{\mathrm{Zar}}\subset \mathrm{GL}_n(\overline{\mathbb Q_p})$ is Zariski closed, hence $p$-adically closed.
	Since $\widehat{\rho}:\widehat{\pi_1(X)}\to\mathrm{GL}_n(\mathbb Z_p)$ is continuous, $\widehat{\rho}^{-1}(\overline{\rho(\pi_1(X))}^{\mathrm{Zar}})\subset \widehat{\pi_1(X)}$ is a closed subset.
	Since the image of $\pi_1(X)\to\widehat{\pi_1(X)}$ is dense and contained in $\widehat{\rho}^{-1}(\overline{\rho(\pi_1(X))}^{\mathrm{Zar}})$.
	Hence $\widehat{\rho}^{-1}(\overline{\rho(\pi_1(X))}^{\mathrm{Zar}})=\widehat{\pi_1(X)}$.
	This shows \eqref{eqn:20230313}.
	Hence $\overline{\widehat{\rho}(\widehat{\pi_1(X)})}^{\mathrm{Zar}}\subset\overline{\rho(\pi_1(X))}^{\mathrm{Zar}}$.
	The converse inclusion is obvious. 
	Hence we have
	\begin{equation}\label{eqn:202303133}
		\overline{\rho(\pi_1(X))}^{\mathrm{Zar}}=\overline{\widehat{\rho}(\widehat{\pi_1(X)})}^{\mathrm{Zar}}\subset \mathrm{GL}_n(\overline{\mathbb Q_p}).
	\end{equation}
	Similarly we have
	$$
	\overline{\tau(\pi_1(X^{\sigma}))}^{\mathrm{Zar}}=\overline{\widehat{\rho}(\widehat{\pi_1(X)})}^{\mathrm{Zar}}\subset \mathrm{GL}_n(\overline{\mathbb Q_p}).$$
	Thus we have \eqref{eqn:202303111} in $\mathrm{GL}_n(\overline{\mathbb Q_p})$.
	Thus \eqref{eqn:202303111} holds in $\mathrm{GL}_n(\mathbb C)$.
	
	Finally we take a morphism $Y\to X$ from a smooth quasi-projective variety $Y$.
	Then we have a natural isomorphism $\widehat{\pi_1(Y)}=\widehat{\pi_1(Y^{\sigma})}$ which commutes with \eqref{eqn:202303112}:
\begin{equation*}
\begin{tikzcd}
	 	\widehat{\pi_1(Y^{\sigma})} \arrow[r,equal] \arrow[d]& \widehat{\pi_1(Y)} \arrow[d]\\ 
	 	\widehat{\pi_1(X^{\sigma})}\arrow[r,equal]  & \widehat{\pi_1(X)}
\end{tikzcd}  
\end{equation*}
	Since the image of $\mathrm{Im}[\pi_1(Y)\to\pi_1(X)]\to \mathrm{Im}[\widehat{\pi_1(Y)} \to   \widehat{\pi_1(X)}]$ is dense, a similar argument for the proof of \eqref{eqn:202303133} yields
	$$
	\overline{\rho(\mathrm{Im}[\pi_1(Y)\to\pi_1(X)])}^{\mathrm{Zar}}=
	\overline{\widehat{\rho}(\mathrm{Im}[\widehat{\pi_1(Y)}\to\widehat{\pi_1(X)}])}^{\mathrm{Zar}}\subset \mathrm{GL}_n(\overline{\mathbb Q_p}).
	$$  
	Similarly we have
	$$
	\overline{\tau(\mathrm{Im}[\pi_1(Y^{\sigma})\to\pi_1(X^{\sigma})])}^{\mathrm{Zar}}=
	\overline{\widehat{\rho}(\mathrm{Im}[\widehat{\pi_1(Y)}\to\widehat{\pi_1(X)}])}^{\mathrm{Zar}}\subset \mathrm{GL}_n(\overline{\mathbb Q_p}).
	$$
	Hence we get \eqref{eqn:202303115} in $\mathrm{GL}_n(\overline{\mathbb Q_p})$, thus in $\mathrm{GL}_n(\mathbb C)$.      
	
	Note that $	\overline{\tau(\mathrm{Im}[\pi_1(Y^{\sigma})\to\pi_1(X^{\sigma})])}^{\mathrm{Zar}}$ is positive dimensional if and only if $\tau(\mathrm{Im}[\pi_1(Y^{\sigma})\to\pi_1(X^{\sigma})])$ is infinite. This concludes the last claim of the theorem.  
\end{proof} 

 \subsection{Proof of \Cref{main2}}
 Now we are able to prove the first main result  of this paper. 
\begin{thm}[=Theorem \ref{main2}]\label{thm:20220819}
Let $X$ be a  complex quasi-projective normal variety and $G$ be a semi-simple linear algebraic group over $\bC$. Suppose that $\varrho: \pi_1(X) \to G(\bC)$ is a Zariski dense and big representation. For any Galois conjugate variety $X^\sigma$ of $X$ under $\sigma \in \mathrm{Aut}(\bC/\bQ)$, 
\begin{thmlist}
	\item \label{item:log general type} there exists a proper Zariski closed subset $Z\subsetneqq X^\sigma$ such that any closed subvariety $V$ of $X^\sigma$ not contained in $Z$ is of log general type. 
	\item \label{item:pseudo Picard}Furthermore, $X^\sigma$ is pseudo Picard hyperbolic.
\end{thmlist}
   \end{thm}
\begin{proof} 
Using \cref{lem:fun}, we can replace $X$ by a desingularization, and thus we may assume that $X$ is smooth.  We first prove the theorem for $X$ itself.   We prove in two steps.

\medskip

\noindent{\it Step 1. We assume that $G$ is almost simple.}   Note that $G$  is   isogenous to exactly one of the following:  $A_n, B_n, C_n, D_n, E_6, E_7, E_8, F_4, G_2$, and is defined over some number field $L'$. Moreover, it is absolutely almost simple.  
We prove the theorem in two cases.

\noindent{\em Case 1: $\varrho$ is rigid.}\quad 	 
Since $\varrho$ is  $G$-rigid,   it is conjugate to some Zariski dense representation $\tau:\pi_1(X)\to G(L)$ where $L$ is a finite extension of $L'$. Moreover, for every embedding $v:L\to \bC$, the representation $v\tau:\pi_1(X)\to G(\bC)$ is rigid, \emph{i.e.} $[v\tau]$ is an isolated point in the complex affine variety $M_B(\pi_1(X),G)_{\bC}$.   

\noindent {\em Case 1.1:} Assume that for some non-archimedean place $v$ of $L$, the associated representation $\tau_v:\pi_1(X)\to G(L_v)$ is unbounded, where $L_v$ denotes the completion of $L$ with respect to $v$. Note that $\tau_v$ is still Zariski dense and big, and $G_{L_{v}}$ is an absolutely simple algebraic group over $L_{v}$.   We apply \cref{thm:main33} to conclude \cref{thm:20220819}. 

\noindent {\em Case 1.2:} Fix a    faithful representation $G\to {\rm GL}_N$. Assume that  for  every non-archimedean place $v$ of $L$, the associated representation $\tau_v:\pi_1(X)\to G(L_v)$ is  bounded.  Then there is a factorisation $\tau:\pi_1(X)\to {\rm GL}_N(\cO_L)$, where $\cO_L$ is the ring of integers.    For every embedding $\nu:L\to \bC$, since   $\nu\tau:\pi_1(X)\to G(\bC)$ is rigid, by \cref{thm:vhs},  $\nu\tau:\pi_1(X)\to {\rm GL}_N(\bC)$ underlies a $\bC$-VHS.  We apply \cite[Proposition 7.1 \& Lemma 7.2]{LS18} to conclude that $\tau$ is  a complex direct factor of a $\bZ$-variation of Hodge structures $\tau''$. Since $\tau$ is big,  so is $\tau''$. Hence the period mapping $p:X\to \mathcal{D}/\Gamma$ of this  $\bZ$-VHS satisfies $\dim X=\dim p(X)$, where $\Gamma$ is the monodromy group.      Let $Z\subset X$ be a proper Zariski closed subset of $X$ so that $p|_{X-Z}$ is finite. 
Then by \cref{thm:PicardVHS}, we obtain the   pseudo Picard hyperbolicity of $X$.  It follows from \cite{BC20} (see also \cite{CD21}) that any closed   subvariety $V$ of $X$ not contained in $Z$ is  of log general type.

\smallskip

\noindent 
{\em Case 2: $\varrho$ is non-rigid.} \quad By \cref{lem:20220819}, it follows that there is a Zariski dense, big and unbounded repsentation $\pi_1(X)\to G(K)$ where $K$ is some non-archimedean local field of zero characteristic. 
 By \cref{thm:main33} again,  we conclude the proof.
 
 Thus we have proved the theorem for $X$  when $G$ is almost simple. 

\medskip

\noindent {\it Step 2. 
We treat the general case that $G$ is semi-simple.}
 Note that $G$ has only finitely many minimal normal subgroup varieties $H_1, \ldots, H_k$, and   the multiplication map
$$
H_1 \times \cdots \times H_k \rightarrow G
$$
is an isogeny. Each $H_i$ is almost-simple and it is centralized by the remaining ones. Then the product $H_1\cdots H_{i-1}H_{i+1}\cdots H_k$ is also a normal subgroup of $G$.  Define $G_i:=G/(H_1\cdots H_{i-1}H_{i+1}\cdots H_k)$. It follows that the natural map $G\to G_1\times\cdots \times G_k$ is an isogeny.   This enables us to replace $G$ by $ G_1\times\cdots \times G_k$,   and the representation $ \pi_1(X)\to G_1(\bC)\times\cdots \times G_k(\bC)$ induced by $\varrho$ is also Zariski dense and big. We use the same letter $\varrho$ to denote this representation.   Take the projection $\varrho_i:\pi_1(X)\to G_i(\mathbb C)$.
Then $\varrho_i$ is Zariski dense for all $i$.

We apply \cref{lem:kollar}.
Then after passing to an \'etale cover of $X$, there exists a rational map $p_i:X\dashrightarrow Y_i$ and a big representation $\tau_i:\pi_1(Y_i)\to G_i(\mathbb C)$ such that $p_i^*\tau_i=\varrho_i$. 
Then $\tau_i$ is big and Zariski dense.
So we may apply Step 1 above to get the proper Zariski closed set $Z_i\subsetneqq Y_i$ such that \cref{item:pseudo Picard} holds for $Y_i$ with $Z_i$ the exceptional set. 
Let $q:X\dashrightarrow Y_1\times \cdots\times Y_k$ be the natural map and let $\alpha:X\dashrightarrow S$ be the quasi-Stein factorisation of $q$ (cf. \cref{lem:Stein}).
Then $\alpha$ is birational.
Indeed let $F\subset X$ be a general fiber of $\alpha$.
To show $\dim F=0$, we assume contrary $\dim F>0$.
Since the induced map $F\to Y_i$ is constant, we have $\varrho_i(\mathrm{Im}[\pi_1(F)\to \pi_1(X)])=\{ 1\}$.
Hence $\varrho(\mathrm{Im}[\pi_1(F)\to \pi_1(X)])=\{ 1\}$.
This contradicts to the assumption that $\varrho$ is big.
Hence $\dim F=0$, so $\alpha$ is birational. 

Now we set $Z=\mathrm{Ex}(\alpha)\cup (\cup_ip_i^{-1}(Z_i))$.
Let $f:\bD^*\to X$ be holomorphic such that the image is not contained in $Z$.
Then by Step 1 above, the map $p_i\circ f:\bD^*\to Y_i$ does not have essential singularity at $0\in\bD$, hence the same holds for $\alpha\circ f:\bD^*\to S$.
This proves \cref{item:pseudo Picard} for $X$.

\medspace

Let us prove \cref{item:log general type}.  %After taking a resolution of indeterminacy of $p_i$, we may assume that $p_i:X\to Y_i$ is a regular morphism. %Since $p_1^*\tau_1=\varrho_1$, it follows that for general smooth fibers $F$ of $p_1$, $\varrho_1({\rm Im}(\pi_1(F)\to \pi_1(X)))$ is trivial. 
We will prove the result by induction on $k$. The case $k=1$ is proved by Step 1.  Assume now the statement is true for $k-1$. For the representation $\varrho':\pi_1(X)\to G_2(\bC)\times\cdots\times G_k(\bC)$ defined by the composition of $\varrho$ and the quotient $G(\bC)\to G_2(\bC)\times\cdots\times G_k(\bC)$, by \cref{lem:kollar},  after we replace $X$ by a finite \'etale cover and a birational proper morphism, there is a dominant morphism  $p:X\to Y$   (resp. $p_1:X\to Y_1$) with connected general fibers and  a big and Zariski dense representation $\tau:\pi_1(Y)\to G_2(\bC)\times\cdots\times G_k(\bC)$ (resp. $\tau_1:\pi_1(Y_1)\to G_1(\bC)$) such that $p^*\tau=\varrho$ (resp. $p_1^*\tau_1=\varrho_1$). By the induction, there is a  proper Zariski closed set  $Z_0\subsetneqq Y$  (resp. $Z_1\subsetneqq Y_1$) such that any closed   subvariety $V\not\subset Z_0$ (resp. $V_1\not\subset  Z_1$) is of log general type.  As we have seen above, the natural morphism
\begin{align*}
	 q: X&\to Y_1\times Y\\
	 x&\mapsto (p_1(x), p(x))
\end{align*} 
satisfies $\dim X=\dim q(X)$.   Let $\alpha:X\dashrightarrow S$ be the quasi-Stein factorisation of $q$. 
Then $\alpha$ is birational.
  Set $Z:=p^{-1}(Z_0)\cup p_1^{-1}(Z_1)\cup {\rm Exc}(\alpha)$. Let $V\subset X$ be  any closed   subvariety  not contained in $Z$.  Then the  closure   $\overline{p(V)}$ 
 is of log general type.  Since $p_1(V)\not\subset Z_1$,      $\overline{p_1(F)}$  is not contained in $Z_1$ for general fibers $F$ of $p|_{V}:V\to \overline{p(V)}$. Hence $\overline{p_1(F)}$ is of log general type. Note that    $\alpha|_F:F\to \overline{\alpha(F)}$ of $p$ is  generically finite. Hence $p_1|_F:F\to \overline{p_1(F)}$  is also generically finite. It follows that $F$ is also of log general type.   
We use Fujino's addition formula for logarithmic Kodaira dimensions \cite[Theorem 1.9]{Fuj17} to show that $V$ is of log general type, which proves \cref{item:log general type} for $X$. Therefore, we have established the theorem for $X$.

Let $\sigma\in {\rm Aut}(\bC/\bQ)$. Using \cref{thm:conjugate}, we can construct a representation $\varrho^\sigma:\pi_1(X^\sigma)\to G(\bC)$ that is also Zariski dense and big, satisfying the conditions of the theorem. Hence, we have proven the theorem for $X^\sigma$.  
\end{proof}

\begin{rem}\label{rem:sharp}
Note that the condition in \cref{thm:20220819} is sharp. For example, consider an abelian variety $X$ of dimension $n$. The representation 
\begin{align*}
	\bZ^{2n}\simeq \pi_1(X)&\to (\bC^*)^{2n}\\
	(a_1,\ldots,a_{2n})&\mapsto (\exp({a_1}),\ldots, \exp(a_{2n}))
	\end{align*}
 is a Zariski dense representation. However, $X$ is not of general type and contains Zariski dense entire curves. This example demonstrates that the semisimplicity of $G$ is necessary for \cref{thm:20220819} to hold.
	
Another example to consider is a curve $C$ of genus at least 2. There exists a Zariski dense representation $\varrho:\pi_1(C)\to G(\bC)$ where $G$ is some semisimple algebraic group over $\bC$. The representation $\varrho:\pi_1(C\times \bP^1)\to G(\bC)$ is Zariski dense but not big. It is clear that $C\times \bP^1$ is neither pseudo Brody hyperbolic nor of general type. Thus, the bigness condition in \cref{thm:20220819} is also essential.
\end{rem}

\section[On the GGL conjecture]{On the generalized Green-Griffiths-Lang conjecture}\label{sec:proof2}

In this section we prove \cref{main:GGL,main:special}.
We first prove the following lemma.

\begin{lem}\label{lem:202305101} 
Let $X$ be a quasi-projective normal variety and let $\varrho:\pi_1(X)\to {\rm GL}_N(\bC)$ be a reductive and big representation.  
Then there exist 
\begin{itemize}
\item a semi-abelian variety $A$,
\item a smooth quasi projective variety $Y$ satisfying $\Spp(Y)\subsetneqq Y$ and $\Spalg(Y)\subsetneqq Y$,
\item a birational modification $\widehat{X}'\to \widehat{X}$ of a finite \'etale cover $\widehat{X}\to X$,
\item a morphism $g:\widehat{X}'\to A\times Y$ 
\end{itemize}
such that $\dim g(\widehat{X}')=\dim \widehat{X}'$.
Moreover $p:\widehat{X}'\to Y$ is dominant, where $p$ is the composite of $g$ and the second projection $A\times Y\to Y$.
\end{lem}

\begin{proof}
Let $G$ be the Zariski closure of $\varrho$ which is reductive.  
Let $G_0$ be the connected component of $G$ which contains the identity element of $G$. 
Then after replacing $X$ by a finite \'etale cover   corresponding to the finite index subgroup $\rho^{-1}(\varrho(\pi_1(X))\cap G_0(\bC))$ of $\pi_1(X)$, we can assume that the Zariski closure $G$ of $\varrho$ is connected.  
Hence the radical $R(G)$ of $G$ is a torus, and the derived group $\cD G$ is semisimple or trivial.
Write $G_2:=G/\cD G$ and $G_1=G/R(G)$. 
Then   $G_1$ is either semisimple or trivial and $G_2$ is a torus.  
Moreover, the natural  morphism $G\to G_1\times G_2$ is an isogeny.  
Let $\varrho':\pi_1(X)\to G_1(\bC)\times G_2(\bC)$  be the composition of $\varrho$ and $G(\bC)\to G_1(\bC)\times G_2(\bC)$. 
Then it is also big and Zariski dense. 
Denote by  $\varrho_i:\pi_1(X)\to G_i(\bC)$  the  composition of $\varrho:\pi_1(X)\to G_1(\bC)\times G_2(\bC)$ and $ G_1(\bC)\times G_2(\bC)\to G_i(\bC)$, which is Zariski dense. 
After replacing $X$ by a finite \'etale cover, we may assume that $ H_1(X,\bZ)$ is torsion free.
Hence $\varrho_2$ factors the quasi-albanese map $a:X\to A$, i.e., there exists $\varrho_2':H_1(X,\bZ)\to G_2(\bC)$ such that $a^*\varrho_2'=\varrho_2$.

If $G_1$ is not trivial, we apply \cref{lem:kollar}.
Then after replacing $X$ by a finite \'etale cover $\nu:\widehat{X}\to X$ and a birational modification $\mu:\widehat{X}'\to X$, there exists a dominant morphism $p:\widehat{X}'\rightarrow Y$ with connected general fibers and a representation $\tau:\pi_1(Y)\to G_1(\mathbb C)$ such that $p^*\tau=(\nu\circ\mu)^*\varrho_1$ and
	$\tau$ is big and Zariski dense. 
By \cref{main2}, $\Spalg(Y)\subsetneqq Y$ and $\Spp(Y)\subsetneqq Y$.
If $G_1$ is trivial, then we set $Y$ to be a point and   $p:\widehat{X}'\to Y$ is the constant map.  
In this case, we also have $\Spalg(Y)\subsetneqq Y$ and $\Spp(Y)\subsetneqq Y$.

Consider the morphism $g:\widehat{X}'\to A\times Y$ defined by $x\mapsto (a\circ\nu\circ\mu(x),p(x))$.
Let $\beta:\widehat{X}'\to S$ be the quasi-Stein factorisation of $g$ defined in \cref{lem:Stein}.  
 Then $\beta$ is birational.
	Indeed let $F\subset \widehat{X}'$ be a general fiber of $\beta$.
We shall show $\dim F=0$.
	Since the induced map $F\to Y$ is constant, we have $(\nu\circ\mu)^*\varrho_1(\mathrm{Im}[\pi_1(F^{\mathrm{norm}})\to \pi_1(\widehat{X}')])=\{ 1\}$. 
Since the induced map $F\to A$ is constant, we have $(\nu\circ\mu)^*\varrho_2(\mathrm{Im}[\pi_1(F^{\mathrm{norm}})\to \pi_1(\widehat{X}')])=\{ 1\}$. 
Therefore, one has $(\nu\circ\mu)^*\varrho(\mathrm{Im}[\pi_1(F^{\mathrm{norm}})\to \pi_1(\widehat{X}')])=\{ 1\}$.
Since $(\nu\circ\mu)^*\varrho:\pi_1(\widehat{X}')\to G$ is big, we have $\dim F=0$.
Hence $\beta$ is birational.
Thus $\dim g(\widehat{X}')=\dim S=\dim \widehat{X}'$.
\end{proof}

\begin{lem}[{\cite[Lemma 1.6]{CDY25}}]\label{LEM:20230509}  
	Let $\alpha:X\to \cA$ be  a  (possibly non-proper)  morphism  from a smooth quasi-projective variety $X$ to a semi-abelian variety $\cA$ with   $\overline{\kappa}(X)=0$. 
	Assume that $\dim X=\dim \alpha(X)$ and $\dim X>0$.
	Then $\Spab(X)=X$.\qed
\end{lem}

%\begin{proof}
%	As in the step 1 of the proof of \cref{lem:abelian pi}, we may assume that $\alpha:X\to \cA$ is birational.
%	We apply \cref{lem:abelian pi0} to get the isomorphism $\alpha|_{X^{\circ}}:X^{\circ}\to \cA^{\circ}$, where $\cA-\cA^\circ$ of codimension at least two.
%	Then we have the inverse $\cA^{\circ}\to X$ of the isomorphism $X^{\circ}\to \cA^{\circ}$.
%	Thus $\Spab(X)=X$.
%\end{proof}

We recall one theorem from Part I of our paper series.

  \begin{thm}[{\cite[Corollary 10.8]{CDY25}}]\label{cor:GGL2}
  	Let $X$ be a smooth quasi-projective variety and let $a:X\to A\times S^\circ$ be a morphism such that $\dim X=\dim a(X)$, where $S^\circ$ is a smooth quasi-projective variety ($S$ can be a point). Write $b:X\to S^\circ$ as the composition of $a$ with the projection map $A\times S^\circ\to S^\circ$. Assume that $b$ is dominant.
  	\begin{thmlist}
  		\item\label{coritem1} 
  		Suppose $S^\circ$ is pseudo Picard hyperbolic.
  		If $X$ is of log general type, then $X$  is pseudo Picard hyperbolic.
  		\item \label{coritem2} Suppose $\Spalg(S^\circ)\subsetneqq S^\circ$.
  		If $\Spab(X)\subsetneqq X$, then $\Spalg(X)\subsetneqq X$.   \qed
  	\end{thmlist}  
  \end{thm}

Let us prove \cref{main:GGL}.
\begin{thm}[=\cref{main:GGL}]\label{thm:GGL}
	Let $X$ be a complex  smooth  quasi-projective  variety admitting a big and reductive representation  $\varrho:\pi_1(X)\to {\rm GL}_N(\bC)$. 
Then for any  automorphism $\sigma\in {\rm Aut}(\bC/\bQ)$, the following properties are equivalent:
	\begin{enumerate}[ font=\normalfont, label=(\alph*)] 
		\item $X^\sigma$ is of log general type. 
		\item   $\Spp(X^\sigma)\subsetneqq X^\sigma$.
\item  $\Sph(X^\sigma)\subsetneqq X^\sigma$.
\item  $\Spab(X^\sigma)\subsetneqq X^\sigma$.
\item  $\Spalg(X^\sigma)\subsetneqq X^\sigma$.
	\end{enumerate} 
\end{thm}

\begin{proof}  
	We use \cref{thm:conjugate} to show that it suffices to prove the theorem for $X$ itself. 
	We apply \cref{lem:202305101}.
Then by replacing $X$ with a finite \'etale cover and a birational modification, we obtain a smooth quasi-projective variety $Y$ (might be zero-dimensional), a semiabelian variety $A$, and a morphism $g:X\to A \times Y$ that satisfy the following properties:
	\begin{itemize}
		\item  $\dim X=\dim g(X)$.
		\item Let $p:X\to Y$ be  the composition of $g$ with the projective map $A\times Y\to Y$. 
		Then  $p$   is dominant.
		\item $\Spp(Y)\subsetneqq Y$ and $\Spalg(Y)\subsetneqq Y$.
	\end{itemize} 
Therefore, the conditions in \cref{cor:GGL2} are satisfied.
This yields the two implications: $(a)\implies (b)$ and $(d)\implies (e)$. 
By \cite[Lemma 2.1]{CDY25}, we have the implications: $(b)\implies (c)$ and $(c)\implies (d)$.
 Since the implication of $(e)\implies (a)$ is direct, we have proved the equivalence of $(a)$, $(b)$, $(c)$, $(d)$ and $(e)$.
\end{proof}

\begin{thm}[=\cref{main:special}]\label{thm:special}
	Let $X$ be a smooth quasi-projective   variety  and $\varrho:\pi_1(X)\to {\rm GL}_N(\bC)$ be a large and reductive representation. Then   for any  automorphism $\sigma\in {\rm Aut}(\bC/\bQ)$,  
	\begin{thmlist}
		\item \label{itemize:same} the four special subsets defined in \cref{def:special2} are the same, i.e., $$\Spalg(X^\sigma)=\Spab(X^\sigma)=\Sph(X^\sigma)=\Spp(X^\sigma).$$  
		\item These special subsets are conjugate under  automorphism   $\sigma$,  i.e., 
		\begin{align}\label{eq:conjugate}
			 \Sp_{\bullet}(X^\sigma)=\Sp_{\bullet}(X)^\sigma,
		\end{align} 
		where $\Sp_{\bullet}$  denotes any of $\Spab$, $\Sph$, $\Spalg$ or $ \Spp$.
	\end{thmlist}
\end{thm}
\begin{proof}
\noindent {\em Step 1. }  Let $Y$ be a closed subvariety of $X$ that is not of log general type. Let $\iota:Z\to Y$ be a desingularization. 
Then $\iota^*\varrho$ is big and reductive. 
By \cref{thm:GGL}, we have $\Spab(Z)=Z$.
Hence $Y\subset \Spab(X)$, which implies $\Spalg(X)\subset \Spab(X)$.

\medspace

\noindent {\em Step 2. }   Let $f:\bD^*\to X$ be a holomorphic map with essential singularity at the origin. Let $Z$ be a desingularization of the Zariski closure of $f(\bD^*)$. 
Note that the natural morphism $\iota:Z\to X$ induces a big and reductive representation $\iota^*\varrho$.   
Since $\Spp(Z)=Z$,
\cref{thm:GGL} implies that $Z$ is not of log general type. 
Hence $\iota(Z)\subset \Spalg(X)$, which implies $\Spp(X)\subset \Spalg(X)$.
By  \cite[Lemma 2.1]{CDY25} and Step 1, we conclude  that $\Spalg(X)=\Spab(X)=\Sph(X)=\Spp(X)$.

\medspace

\noindent {\em Step 3. }  	 We use \cref{thm:conjugate} to show that there is a large and reductive representation $\varrho^\sigma:\pi_1(X^\sigma)\to {\rm GL}_N(\bC)$. By Step 2, we conclude that $\Spalg(X^\sigma)=\Spab(X^\sigma)=\Sph(X^\sigma)=\Spp(X^\sigma)$.

\medspace

\noindent {\em Step 4. }   A quasi-projective variety $V$ is of log general type if and only if its conjugate $V^\sigma$ is of log general type.   It follows that $\Spalg(X^\sigma)=\Spalg(X)^\sigma.$ By Step 3 we conclude \eqref{eq:conjugate}. We complete the proof of the theorem. 
\end{proof}
 We provide a class of quasi-projective varieties whose fundamental groups have reductive and large representations.  
\begin{proposition}\label{prop:large}
	Let $X$ be  a normal quasi-projective variety.  If $a:X\to A$ is a morphism to a semiabelian variety $A$ that satsifies $\dim a(X)=\dim X$, then there exists a big and reductive representation $\varrho:\pi_1(X)\to \mathrm{GL}_N(\mathbb C)$. Moreover, if $a$ is quasi-finite, then $\varrho$ is furthermore large.  
\end{proposition}  
\begin{proof} 
We start from the following two claims.

\begin{claim}\label{claim:abelian large}
	If $A_0$ is an abelian variety, then $\pi_1(A_0)$ is large.  
\end{claim}
\begin{proof}[Proof of \cref{claim:abelian large}]
	Let $Z$ be any closed positive-dimensional subvariety of $A_0$. Suppose	 that  ${\rm Im}[\pi_1(Z^{\rm norm})\to \pi_1(A_0)]$ is finite. Then there is a finite \'etale cover $Y\to Z^{\rm norm}$ such that ${\rm Im}[\pi_1(Y)\to \pi_1(A_0)]=\{1\}$. Therefore, the natural morphism $Y\to A_0$ lifts to a holomorphic map $Y\to \widetilde{A}_0$, where $\widetilde{A}_0$ is the universal covering of $A_0$ that is  isomorphic to $\bC^N$.    Therefore $Y\to \widetilde{A}_0$ is constant, a contradiction.
\end{proof} 
\begin{claim}\label{claim:large}
Let $Y$ be a closed subvariety of $X$. If $a(Y)$ is not a point, then ${\rm Im}[\pi_1(Y^{\rm norm})\to \pi_1(A)]$ is infinite.   
\end{claim} 
\begin{proof}[Proof of \cref{claim:large}] 
 Note that $A$ admits a short exact sequence
\begin{align}\label{eq:short}
 0\to (\bC^*)^k\to A\stackrel{\pi}{\to} A_0\to 0.
\end{align}
 Let $Z$ be the closure of $\pi\circ a(Y)$.  If $Z$ is positive-dimensional, then ${\rm Im}[\pi_1(Z^{\rm norm})\to \pi_1(A_0)]$  is infinite by \cref{claim:abelian large}. It follows from    \cref{lem:finiteindex} that  ${\rm Im}[\pi_1(Y^{\rm norm})\to \pi_1(A)]$  is also infinite and the claim is proved.  
 
Assume now $\pi\circ a(Y)$ is a point.   Let $W$ be the closure of $a(Y)$. By  \eqref{eq:short}, $W$ is contained in $(\bC^*)^k$.   Since $W$ is assumed to be positive-dimensional, then it must dominate  some factor $\bC^*$ of   $(\bC^*)^k$.  By \cref{lem:finiteindex} again, ${\rm Im}[\pi_1(W^{\rm norm})\to \pi_1((\bC^*)^k)]$ is infinite.  Since	 we have the following short exact sequence:
 $$
 0=\pi_2(A_0)\to \pi_1((\bC^*)^k)\to \pi_1(A),
 $$
it follows that ${\rm Im}[\pi_1(W^{\rm norm})\to \pi_1(A)]$ is infinite.  Applying \cref{lem:finiteindex} once again, we conclude that ${\rm Im}[\pi_1(Y^{\rm norm})\to \pi_1(A)]$ is infinite. The claim is proved. 
\end{proof}
Since $\dim (X)=\dim a(X)$, there exists a proper Zariski closed subset  $\Xi\subsetneqq X$ such that $a|_{X\backslash \Xi}:X\backslash \Xi\to A$ is quasi-finite. Note that $\Xi=\varnothing$ if $a$ is quasi-finite.  Therefore, for any positive-dimensional closed subvariety $Y$ with $Y\not\subset  \Xi$,  we have $\dim(a(Y))>0$. According to Claim \ref{claim:large}, we can conclude that ${\rm Im}[\pi_1(Y^{\rm norm})\to \pi_1(A)]$ is infinite. 

Since $\pi_1(A)$ is abelian, we can embed it faithfully into $(\bC^*)^N\hookrightarrow {\rm GL}_N(\bC)$, where the later is the diagonal embedding.  Thus, the composition $\varrho:\pi_1(X)\to {\rm GL}_N(\bC)$ with $a_*:\pi_1(X)\to \pi_1(A)$   is a big representation. Furthermore, when $a$ is quasi-finite,  $\varrho$ is large. As  $\varrho(\pi_1(X))$ is contained in $(\bC^*)^N$, its Zariski closure is   a torus in ${\rm GL}_N(\bC)$. Hence, $\varrho$ is reductive. This completes the proof of the proposition.
\end{proof}  
 
We conclude the paper with some more recent developments.
\begin{rem}
Building on the ideas and results developed here, in \cite{DY23b} the second and third authors proved  results analogous to \cref{main2,main:GGL} for Zariski dense and big representations $\pi_{1}(X)\to G(K)$, where $G$ is a reductive or semisimple algebraic group defined over an algebraically closed field $K$ of \emph{positive characteristic}. 
\end{rem}

  %\bibliography{biblio}

\begin{thebibliography}{BDDM22}
 	
 	\bibitem[AB94]{AB}
 	{\scshape N.~A'Campo {\normalfont \smfandname} M.~Burger} -- {\og Arithmetic
 		lattices and commensurator according to {G}. {A}. {Margulis}\fg},
 	\emph{Invent. Math.} \textbf{116} (1994), no.~1-3, p.~1--25 (French),
 	\url{https://doi.org/10.1007/BF01231555}.
 	
 	\bibitem[AB08]{AB08}
 	{\scshape P.~Abramenko {\normalfont \smfandname} K.~S. Brown} --
 	\emph{Buildings. {Theory} and applications.}, Grad. Texts Math., vol. 248,
 	Berlin: Springer, 2008 (English),
 	\url{https://doi.org/10.1007/978-0-387-78835-7}.
 	
 	\bibitem[ADH16]{Ara16}
 	{\scshape D.~Arapura, A.~Dimca {\normalfont \smfandname} R.~Hain} -- {\og On
 		the fundamental groups of normal varieties\fg}, \emph{Commun. Contemp. Math.}
 	\textbf{18} (2016), no.~4, p.~17 (English), Id/No 1550065,
 	\url{https://doi.org/10.1142/S0219199715500650}.
 	
 	\bibitem[BC20]{BC20}
 	{\scshape Y.~Brunebarbe {\normalfont \smfandname} B.~Cadorel} -- {\og
 		Hyperbolicity of varieties supporting a variation of {Hodge} structure\fg},
 	\emph{Int. Math. Res. Not.} \textbf{2020} (2020), no.~6, p.~1601--1609
 	(English), \url{https://doi.org/10.1093/imrn/rny054}.
 	
 	\bibitem[BDDM22]{BDDM}
 	{\scshape D.~{Brotbek}, G.~{Daskalopoulos}, Y.~{Deng} {\normalfont \smfandname}
 		C.~{Mese}} -- {\og {Pluriharmonic maps into buildings and symmetric
 			differentials}\fg}, \emph{arXiv e-prints} (2022), p.~arXiv:2206.11835,
 	\url{https://doi.org/10.48550/arXiv.2206.11835}.
 	
 	\bibitem[Bor72]{Bor72}
 	{\scshape A.~Borel} -- {\og Some metric properties of arithmetic quotients of
 		symmetric spaces and an extension theorem\fg}, \emph{J. Differential
 		Geometry} \textbf{6} (1972), p.~543--560,
 	\url{http://projecteuclid.org/euclid.jdg/1214430642}.
 	
 	\bibitem[{Bru}22]{Bru22}
 	{\scshape Y.~{Brunebarbe}} -- {\og {Hyperbolicity in presence of a large local
 			system}\fg}, \emph{arXiv e-prints} (2022), p.~arXiv:2207.03283,
 	\url{https://doi.org/10.48550/arXiv.2207.03283}.
 	
 	\bibitem[{Bru}23]{Bru23}
 	\bysame , {\og {Existence of the Shafarevich morphism for semisimple local
 			systems on quasi-projective varieties}\fg}, \emph{arXiv e-prints} (2023),
 	p.~arXiv:2305.09741, \url{https://doi.org/10.48550/arXiv.2305.09741}.
 	
 	\bibitem[BV12]{BV12}
 	{\scshape N.~Borne {\normalfont \smfandname} A.~Vistoli} -- {\og Parabolic
 		sheaves on logarithmic schemes\fg}, \emph{Adv. Math.} \textbf{231} (2012),
 	no.~3-4, p.~1327--1363 (English),
 	\url{https://doi.org/10.1016/j.aim.2012.06.015}.
 	
 	\bibitem[Cam91]{Cam91}
 	{\scshape F.~Campana} -- {\og {On twistor spaces of the class
 			$\mathcal{C}$}\fg}, \emph{Journal of Differential Geometry} \textbf{33}
 	(1991), no.~2, p.~541 -- 549, \url{https://doi.org/10.4310/jdg/1214446329}.
 	
 	\bibitem[Cam11]{Cam11}
 	{\scshape F.~Campana} -- {\og Orbifoldes g\'{e}om\'{e}triques sp\'{e}ciales et
 		classification bim\'{e}romorphe des vari\'{e}t\'{e}s k\"{a}hl\'{e}riennes
 		compactes\fg}, \emph{J. Inst. Math. Jussieu} \textbf{10} (2011), no.~4,
 	p.~809--934, \url{https://doi.org/10.1017/S1474748010000101}.
 	
 	\bibitem[Cas76]{Cas76}
 	{\scshape J.~W.~S. Cassels} -- {\og An embedding theorem for fields\fg},
 	\emph{Bull. Austral. Math. Soc.} \textbf{14} (1976), no.~2, p.~193--198,
 	\url{https://doi.org/10.1017/S000497270002503X}.
 	
 	\bibitem[CCE15]{CCE15}
 	{\scshape F.~Campana, B.~Claudon {\normalfont \smfandname} P.~Eyssidieux} --
 	{\og Linear representations of {K{\"a}hler} groups: factorizations and linear
 		{Shafarevich} conjecture\fg}, \emph{Compos. Math.} \textbf{151} (2015),
 	no.~2, p.~351--376 (French), \url{https://doi.org/10.1112/S0010437X14007751}.
 	
 	\bibitem[CD21]{CD21}
 	{\scshape B.~{Cadorel} {\normalfont \smfandname} Y.~{Deng}} -- {\og {Picard
 			hyperbolicity of manifolds admitting nilpotent harmonic bundles}\fg},
 	\emph{arXiv e-prints} (2021), p.~arXiv:2107.07550,
 	\url{https://arxiv.org/abs/2107.07550}.
 	
 	\bibitem[CDY22]{CDY22original}
 	{\scshape B.~{Cadorel}, Y.~{Deng} {\normalfont \smfandname} K.~{Yamanoi}} --
 	{\og {Hyperbolicity and fundamental groups of complex quasi-projective
 			varieties}\fg}, \emph{arXiv e-prints: arXiv:2212.12225} (2022),
 	\url{https://doi.org/10.48550/arXiv.2212.12225}.
 	
 	\bibitem[CDY25]{CDY25}
 	\bysame , {\og {Hyperbolicity and fundamental groups of complex
 			quasi-projective varieties (I): Maximal quasi-Albanese dimension by
 			Nevanlinna theory}\fg}, \emph{arXiv e-prints: arXiv:2511.04405} (2025),
 	\url{https://doi.org/10.48550/arXiv.2511.04405}.
 	
 	\bibitem[CP19]{CP19}
 	{\scshape F.~Campana {\normalfont \smfandname} M.~P{\u{a}}un} -- {\og
 		Foliations with positive slopes and birational stability of orbifold
 		cotangent bundles\fg}, \emph{Publ. Math., Inst. Hautes {\'E}tud. Sci.}
 	\textbf{129} (2019), p.~1--49 (English),
 	\url{https://doi.org/10.1007/s10240-019-00105-w}.
 	
 	\bibitem[CS08]{CS08}
 	{\scshape K.~Corlette {\normalfont \smfandname} C.~Simpson} -- {\og On the
 		classification of rank-two representations of quasiprojective fundamental
 		groups\fg}, \emph{Compos. Math.} \textbf{144} (2008), no.~5, p.~1271--1331
 	(English), \url{https://doi.org/10.1112/S0010437X08003618}.
 	
 	\bibitem[Dem20]{Dem20}
 	{\scshape J.-P. Demailly} -- {\og Recent results on the {Kobayashi} and
 		{Green}-{Griffiths}-{Lang} conjectures\fg}, \emph{Jpn. J. Math. (3)}
 	\textbf{15} (2020), no.~1, p.~1--120 (English),
 	\url{https://doi.org/10.1007/s11537-019-1566-3}.
 	
 	\bibitem[{Den}22]{Den22}
 	{\scshape Y.~{Deng}} -- {\og A characterization of complex quasi-projective
 		manifolds uniformized by unit balls\fg}, \emph{Math. Ann.} (2022) (English),
 	\url{https://doi.org/https://doi.org/10.1007/s00208-021-02334-z}.
 	
 	\bibitem[{Den}23]{Denarxiv}
 	\bysame , {\og {Big Picard theorems and algebraic hyperbolicity for varieties
 			admitting a variation of Hodge structures}\fg}, \emph{{Épijournal de
 			Géométrie Algébrique}} \textbf{{Volume 7}} (2023),
 	\url{https://doi.org/10.46298/epiga.2023.volume7.8393}.
 	
 	\bibitem[DM26]{DM24}
 	{\scshape Y.~{Deng} {\normalfont \smfandname} C.~{Mese}} -- {\og Existence and
 		unicity of pluriharmonic maps to euclidean buildings and applications\fg}, in
 	\emph{New Aspects of Teichmüller Theory}, Adv. Stud. Pure Math., Math. Soc.
 	Japan, Tokyo, 2026, to appear,
 	\url{https://doi.org/10.48550/arXiv.2410.07871}.
 	
 	\bibitem[DMW24]{DMW24}
 	{\scshape Y.~{Deng}, C.~{Mese} {\normalfont \smfandname} B.~{Wang}} -- {\og
 		{Deformation Openness of Big Fundamental Groups and Applications}\fg},
 	\emph{arXiv e-prints} (2024), p.~arXiv:2412.08636,
 	\url{https://doi.org/10.48550/arXiv.2412.08636}.
 	
 	\bibitem[DY25]{DY23b}
 	{\scshape Y.~{Deng} {\normalfont \smfandname} K.~{Yamanoi}} -- {\og Linear
 		shafarevich conjecture in positive characteristic, hyperbolicity and
 		applications\fg}, \emph{J. Reine Angew. Math.} (2025),
 	\url{https://doi.org/doi:10.1515/crelle-2025-0078}.
 	
 	\bibitem[DYK23]{DY23}
 	{\scshape Y.~{Deng}, K.~{Yamanoi} {\normalfont \smfandname} L.~{Katzarkov}} --
 	{\og {Reductive Shafarevich Conjecture}\fg}, \emph{arXiv e-prints} (2023),
 	p.~arXiv:2306.03070, \url{https://doi.org/10.48550/arXiv.2306.03070}.
 	
 	\bibitem[EKPR12]{EKPR12}
 	{\scshape P.~Eyssidieux, L.~Katzarkov, T.~Pantev {\normalfont \smfandname}
 		M.~Ramachandran} -- {\og Linear {Shafarevich} conjecture\fg}, \emph{Ann.
 		Math. (2)} \textbf{176} (2012), no.~3, p.~1545--1581 (English),
 	\url{https://doi.org/10.4007/annals.2012.176.3.4}.
 	
 	\bibitem[{Eys}04]{Eys04}
 	{\scshape P.~{Eyssidieux}} -- {\og {Sur la convexit\'e holomorphe des
 			rev\^etements lin\'eaires r\'eductifs d'une vari\'et\'e projective
 			alg\'ebrique complexe}\fg}, \emph{{Invent. Math.}} \textbf{156} (2004),
 	no.~3, p.~503--564 (French), \url{https://doi.org/10.1007/s00222-003-0345-0}.
 	
 	\bibitem[Fuj17]{Fuj17}
 	{\scshape O.~Fujino} -- {\og Notes on the weak positivity theorems\fg}, in
 	\emph{Algebraic varieties and automorphism groups. Proceedings of the
 		workshop held at RIMS, Kyoto University, Kyoto, Japan, July 7--11, 2014},
 	Tokyo: Mathematical Society of Japan (MSJ), 2017, p.~73--118 (English).
 	
 	\bibitem[Fuj24]{Fuj15}
 	\bysame , {\og On quasi-albanese maps\fg}, \emph{Bollettino dell'Unione
 		Matematica Italiana} (2024), p.~1--41.
 	
 	\bibitem[GGK22]{GGK22}
 	{\scshape M.~{Green}, P.~{Griffiths} {\normalfont \smfandname} L.~{Katzarkov}}
 	-- {\og {Shafarevich mappings and period mappings}\fg}, \emph{arXiv e-prints}
 	(2022), p.~arXiv:2209.14088, \url{https://arxiv.org/abs/2209.14088}.
 	
 	\bibitem[GKP16]{GKP16}
 	{\scshape D.~Greb, S.~Kebekus {\normalfont \smfandname} T.~Peternell} -- {\og
 		\'etale fundamental groups of {K}awamata log terminal spaces, flat sheaves,
 		and quotients of abelian varieties\fg}, \emph{Duke Math. J.} \textbf{165}
 	(2016), no.~10, p.~1965--2004,
 	\url{https://doi.org/10.1215/00127094-3450859}.
 	
 	\bibitem[GS92]{GS92}
 	{\scshape M.~{Gromov} {\normalfont \smfandname} R.~{Schoen}} -- {\og {Harmonic
 			maps into singular spaces and $p$-adic superrigidity for lattices in groups
 			of rank one}\fg}, \emph{{Publ. Math., Inst. Hautes \'Etud. Sci.}} \textbf{76}
 	(1992), p.~165--246 (English), \url{https://doi.org/10.1007/BF02699433}.
 	
 	\bibitem[GW10]{GW10}
 	{\scshape U.~G\"ortz {\normalfont \smfandname} T.~Wedhorn} -- \emph{Algebraic
 		geometry {I}}, Advanced Lectures in Mathematics, Vieweg + Teubner, Wiesbaden,
 	2010, Schemes with examples and exercises,
 	\url{https://doi.org/10.1007/978-3-8348-9722-0}.
 	
 	\bibitem[Iit77]{Iitaka1977}
 	{\scshape S.~Iitaka} -- {\og On logarithmic kodaira dimension of algebraic
 		varieties\fg}, in \emph{Complex Analysis and Algebraic Geometry} (W.~L.
 	Baily~Jr. {\normalfont \smfandname} T.~Shioda, \smfedsname), Iwanami Shoten,
 	Tokyo, 1977, p.~175--190.
 	
 	\bibitem[IS07]{IS07}
 	{\scshape J.~N.~N. Iyer {\normalfont \smfandname} C.~T. Simpson} -- {\og A
 		relation between the parabolic {C}hern characters of the de {R}ham
 		bundles\fg}, \emph{Math. Ann.} \textbf{338} (2007), no.~2, p.~347--383,
 	\url{https://doi.org/10.1007/s00208-006-0078-7}.
 	
 	\bibitem[Kat97]{Kat97}
 	{\scshape L.~Katzarkov} -- {\og On the {Shafarevich} maps\fg}, in
 	\emph{Algebraic geometry. Proceedings of the Summer Research Institute, Santa
 		Cruz, CA, USA, July 9--29, 1995}, Providence, RI: American Mathematical
 	Society, 1997, p.~173--216 (English).
 	
 	\bibitem[KO71]{KO71}
 	{\scshape S.~Kobayashi {\normalfont \smfandname} T.~Ochiai} -- {\og Satake
 		compactification and the great {P}icard theorem\fg}, \emph{J. Math. Soc.
 		Japan} \textbf{23} (1971), p.~340--350,
 	\url{https://doi.org/10.2969/jmsj/02320340}.
 	
 	\bibitem[Kol93]{Kol93}
 	{\scshape J.~Koll{\'a}r} -- {\og Shafarevich maps and plurigenera of algebraic
 		varieties\fg}, \emph{Invent. Math.} \textbf{113} (1993), no.~1, p.~177--215
 	(English), \url{https://doi.org/10.1007/BF01244307}.
 	
 	\bibitem[Kol95]{Kol95}
 	\bysame , \emph{Shafarevich maps and automorphic forms}, Princeton University
 	Press, Princeton (N.J.), 1995 (eng).
 	
 	\bibitem[KP23]{KP23}
 	{\scshape T.~Kaletha {\normalfont \smfandname} G.~Prasad} --
 	\emph{Bruhat-{T}its theory---a new approach}, New Mathematical Monographs,
 	vol.~44, Cambridge University Press, Cambridge, 2023.
 	
 	\bibitem[KS93]{KS}
 	{\scshape N.~J. {Korevaar} {\normalfont \smfandname} R.~M. {Schoen}} -- {\og
 		{Sobolev spaces and harmonic maps for metric space targets}\fg},
 	\emph{{Commun. Anal. Geom.}} \textbf{1} (1993), no.~4, p.~561--659 (English),
 	\url{https://doi.org/10.4310/CAG.1993.v1.n4.a4}.
 	
 	\bibitem[KS97]{KS2}
 	{\scshape N.~J. Korevaar {\normalfont \smfandname} R.~M. Schoen} -- {\og Global
 		existence theorems for harmonic maps to nonlocally compact spaces\fg},
 	\emph{Commun. Anal. Geom.} \textbf{5} (1997), no.~2, p.~333--387 (English),
 	\url{https://doi.org/10.4310/CAG.1997.v5.n2.a4}.
 	
 	\bibitem[Lan91]{Lan97}
 	{\scshape S.~Lang} -- \emph{Number theory. {III}}, Encyclopaedia of
 	Mathematical Sciences, vol.~60, Springer-Verlag, Berlin, 1991, Diophantine
 	geometry, \url{https://doi.org/10.1007/978-3-642-58227-1}.
 	
 	\bibitem[Laz04]{Laz04}
 	{\scshape R.~Lazarsfeld} -- \emph{Positivity in algebraic geometry. {I}.
 		{Classical} setting: line bundles and linear series}, Ergeb. Math. Grenzgeb.,
 	3. Folge, vol.~48, Berlin: Springer, 2004 (English).
 	
 	\bibitem[LS18]{LS18}
 	{\scshape A.~Langer {\normalfont \smfandname} C.~Simpson} -- {\og Rank 3 rigid
 		representations of projective fundamental groups\fg}, \emph{Compos. Math.}
 	\textbf{154} (2018), no.~7, p.~1534--1570 (English),
 	\url{https://doi.org/10.1112/S0010437X18007182}.
 	
 	\bibitem[Mil17]{Mil17}
 	{\scshape J.~S. Milne} -- \emph{Algebraic groups. {The} theory of group schemes
 		of finite type over a field}, Camb. Stud. Adv. Math., vol. 170, Cambridge:
 	Cambridge University Press, 2017 (English),
 	\url{https://doi.org/10.1017/9781316711736}.
 	
 	\bibitem[Moc06]{Moc06}
 	{\scshape T.~Mochizuki} -- {\og Kobayashi-{H}itchin correspondence for tame
 		harmonic bundles and an application\fg}, \emph{Ast\'{e}risque} (2006),
 	no.~309, p.~viii+117.
 	
 	\bibitem[Moc07a]{Moc07}
 	\bysame , {\og Asymptotic behaviour of tame harmonic bundles and an application
 		to pure twistor {$D$}-modules. {I}\fg}, \emph{Mem. Amer. Math. Soc.}
 	\textbf{185} (2007), no.~869, p.~xii+324,
 	\url{https://doi.org/10.1090/memo/0869}.
 	
 	\bibitem[Moc07b]{Moc07b}
 	\bysame , {\og Asymptotic behaviour of tame harmonic bundles and an application
 		to pure twistor {$D$}-modules. {II}\fg}, \emph{Mem. Amer. Math. Soc.}
 	\textbf{185} (2007), no.~870, p.~xii+565,
 	\url{https://doi.org/10.1090/memo/0870}.
 	
 	\bibitem[Mok92]{Mok92}
 	{\scshape N.~Mok} -- {\og Factorization of semisimple discrete representations
 		of {K{\"a}hler} groups\fg}, \emph{Invent. Math.} \textbf{110} (1992), no.~3,
 	p.~557--614 (English), \url{https://doi.org/10.1007/BF01231345}.
 	
 	\bibitem[NW14]{NW13}
 	{\scshape J.~Noguchi {\normalfont \smfandname} J.~Winkelmann} --
 	\emph{Nevanlinna theory in several complex variables and diophantine
 		approximation}, Grundlehren Math. Wiss., vol. 350, Tokyo: Springer, 2014
 	(English), \url{https://doi.org/10.1007/978-4-431-54571-2}.
 	
 	\bibitem[NWY13]{NWY13}
 	{\scshape J.~Noguchi, J.~Winkelmann {\normalfont \smfandname} K.~Yamanoi} --
 	{\og Degeneracy of holomorphic curves into algebraic varieties. {II}\fg},
 	\emph{Vietnam J. Math.} \textbf{41} (2013), no.~4, p.~519--525 (English),
 	\url{https://doi.org/10.1007/s10013-013-0051-1}.
 	
 	\bibitem[Rou09]{Rou09}
 	{\scshape G.~Rousseau} -- {\og Euclidean buildings\fg}, in \emph{G\'eom\'etries
 		\`a courbure n\'egative ou nulle, groupes discrets et rigidit\'es}, Paris:
 	Soci{\'e}t{\'e} Math{\'e}matique de France (SMF), 2009, p.~77--116 (English).
 	
 	\bibitem[Rou23]{Guy23}
 	\bysame , \emph{Euclidean buildings---geometry and group actions}, EMS Tracts
 	in Mathematics, vol.~35, EMS Press, Berlin, [2023] \copyright 2023,
 	\url{https://doi.org/10.4171/etm/35}.
 	
 	\bibitem[Ser64]{Ser64}
 	{\scshape J.-P. Serre} -- {\og Exemples de vari\'{e}t\'{e}s projectives
 		conjugu\'{e}es non hom\'{e}omorphes\fg}, \emph{C. R. Acad. Sci. Paris}
 	\textbf{258} (1964), p.~4194--4196.
 	
 	\bibitem[Sik12]{Sik12}
 	{\scshape A.~S. Sikora} -- {\og Character varieties\fg}, \emph{Trans. Am. Math.
 		Soc.} \textbf{364} (2012), no.~10, p.~5173--5208 (English),
 	\url{https://doi.org/10.1090/S0002-9947-2012-05448-1}.
 	
 	\bibitem[Sim88]{Sim88}
 	{\scshape C.~T. Simpson} -- {\og Constructing variations of {H}odge structure
 		using {Y}ang-{M}ills theory and applications to uniformization\fg}, \emph{J.
 		Amer. Math. Soc.} \textbf{1} (1988), no.~4, p.~867--918,
 	\url{https://doi.org/10.2307/1990994}.
 	
 	\bibitem[Sim90]{Sim90}
 	\bysame , {\og Harmonic bundles on noncompact curves\fg}, \emph{J. Amer. Math.
 		Soc.} \textbf{3} (1990), no.~3, p.~713--770,
 	\url{https://doi.org/10.2307/1990935}.
 	
 	\bibitem[Sim92]{Sim92}
 	\bysame , {\og Higgs bundles and local systems\fg}, \emph{Inst. Hautes
 		\'{E}tudes Sci. Publ. Math.} (1992), no.~75, p.~5--95,
 	\url{http://www.numdam.org/item?id=PMIHES_1992__75__5_0}.
 	
 	\bibitem[Siu75]{Siu75}
 	{\scshape Y.~T. Siu} -- {\og Extension of meromorphic maps into {K}\"{a}hler
 		manifolds\fg}, \emph{Ann. of Math. (2)} \textbf{102} (1975), no.~3,
 	p.~421--462, \url{https://doi.org/10.2307/1971038}.
 	
 	\bibitem[{Sta}22]{stacks-project}
 	{\scshape T.~{Stacks project authors}} -- {\og The stacks project\fg},
 	\url{https://stacks.math.columbia.edu}, 2022.
 	
 	\bibitem[Uen75]{Uen75}
 	{\scshape K.~Ueno} -- \emph{Classification theory of algebraic varieties and
 		compact complex spaces. {Notes} written in collaboration with {P}.
 		{Cherenack}}, Lect. Notes Math., vol. 439, Springer, Cham, 1975 (English).
 	
 	\bibitem[Yam10]{Yam10}
 	{\scshape K.~Yamanoi} -- {\og On fundamental groups of algebraic varieties and
 		value distribution theory\fg}, \emph{Ann. Inst. Fourier} \textbf{60} (2010),
 	no.~2, p.~551--563 (English), \url{https://doi.org/10.5802/aif.2532}.
 	
 	\bibitem[Yam15]{yamanoi2015kobayashi}
 	\bysame , {\og Kobayashi hyperbolicity and higher-dimensional nevanlinna
 		theory\fg}, in \emph{Geometry and Analysis on Manifolds: In Memory of
 		Professor Shoshichi Kobayashi}, Springer, 2015, p.~209--273.
 	
 	\bibitem[Zuo96]{Zuo96}
 	{\scshape K.~Zuo} -- {\og Kodaira dimension and {Chern} hyperbolicity of the
 		{Shafarevich} maps for representations of {{\(\pi_ 1\)}} of compact
 		{K{\"a}hler} manifolds\fg}, \emph{J. Reine Angew. Math.} \textbf{472} (1996),
 	p.~139--156 (English), \url{https://doi.org/10.1515/crll.1996.472.139}.
 	
 \end{thebibliography}
 %\bibliographystyle{smfalpha-url}
   
 \providecommand{\bysame}{\leavevmode ---\ }
 \providecommand{\og}{``}
 \providecommand{\fg}{''}
 \providecommand{\smfandname}{\&}
 \providecommand{\smfedsname}{\'eds.}
 \providecommand{\smfedname}{\'ed.}
 \providecommand{\smfmastersthesisname}{M\'emoire}
 \providecommand{\smfphdthesisname}{Th\`ese}

\end{document}